\documentclass[reqno,a4paper,11pt]{amsart}
\usepackage{hyperref}
\usepackage{tikz}
\usepackage[active]{srcltx}
\usepackage[curve]{xypic}
\usepackage{graphicx}
\usepackage{mathrsfs}

\usepackage{color}

\usepackage[active]{srcltx}
\usepackage{epsfig,amsmath}
\usepackage[english]{babel}
\usepackage{amsmath,amsfonts,amssymb,amscd,color,epsfig,amsthm}
\newtheorem{newthm}{Theorem}
\newtheorem{theorem}{Theorem}[section]
\newtheorem{lemma}[theorem]{Lemma}
\newtheorem{proposition}[theorem]{Proposition}
\newtheorem{corollary}[theorem]{Corollary}

\newtheorem{definition}[theorem]{Definition}

\theoremstyle{remark}

\theoremstyle{plain}

\newtheorem{conj}[theorem]{\bf Conjecture}

\numberwithin{equation}{section}

\newcommand{\wt}{\widetilde}
\newcommand{\wh}{\widehat}

\def\EEE{{\cal E}}
\def\FFF{{\cal F}}

\def\HHH{{\cal H}}

\def\JJJ{{\cal J}}
\def\KKK{{\cal K}}
\def\OOO{{\cal O}}

\def\NNN{{\cal N}}
\def\PPP{{\cal P}}
\def\QQQ{{\cal Q}}
\def\RRR{{\cal R}}

\def\VVV{{\cal V}}
\def\WWW{{\cal W}}

\def\FFFF{\mathscr F}
\def\JJJJ{\mathscr J}

\def\LLLL{\mathscr L}
\def\EEEE{\mathscr E}
\def\XXXX{\mathscr X}

\def\c{\mathbf c}

\def\s{\mathbf s}

\def\DD{\mathbf D}
\def\XX{\mathbf X}
\def\ZZ{\mathbf Z}

\def\g{\gamma}
\def\G{\Gamma}

\def\De{\Delta}
\def\de{\delta}

\def\sm{\smallsetminus}
\def\R{\mbox{$\mathbb R$}}

\def\C{\mbox{$\mathbb C$}}
\def\T{\mbox{$\mathbb T$}}

\def\D{\mathbb D}
\def\Q{\mathbb Q}

\def\Z{\mbox{$\mathbb Z$}}

\def\N{\mbox{$\mathbb N$}}

\def\lv{ \left(\begin{matrix} }
 \def\rv{\end{matrix}\right)}

\def\cal{\mathcal}

\def\dw{{\dw}}

\def\ds{\displaystyle}

\newcommand{\mylabel}[1]{\label{#1}}

\newcommand{\REFEQN}[1] { \begin{equation}\mylabel{#1} }
\newcommand{\ENDEQN}{\end{equation}}
\newcommand{\REFTHM}[1] { \begin{theorem}\mylabel{#1} }
\newcommand{\ENDTHM}{\end{theorem}}
\newcommand{\REFNTH}[1] { \begin{newthm}\mylabel{#1} }
\newcommand{\ENDNTH}{\end{newthm}}
\newcommand{\REFPROP}[1]{\begin{proposition}\mylabel{#1} }
\newcommand{\ENDPROP}{\end{proposition} }
\newcommand{\REFLEM}[1]{\begin{lemma}\mylabel{#1} }
\newcommand{\ENDLEM}{\end{lemma} }
\newcommand{\REFCOR}[1]{\begin{corollary}\mylabel{#1} }
\newcommand{\ENDCOR}{\end{corollary} }

\def\ds{\displaystyle }

\def\pf{postcritically-finite }
\def\Mi{Misiurewicz }

\def\ov{\overline}

\def\T{{\mathbb T}}

\usepackage{tikz}
\usetikzlibrary{arrows}
\tikzstyle{every picture}=[> = to]
\tikzset{cdlabel/.style={execute at begin node=$\scriptstyle,execute at end node=$}}
\tikzset{implication/.style={double equal sign distance, -implies}}
\tikzset{biimplication/.style={double equal sign distance, implies-implies}}

\title{On the core entropy of Newton maps}
\author{Yan Gao}
\date{\today}
\begin{document}
\maketitle
\begin{abstract}
In this paper, we define the core entropy for \pf Newton maps and study its continuity within this family. We show that the entropy function is not continuous in this family, which is different from the polynomial case studied by Thurston, Gao, Dudko-Schleicher,  Tiozzo \cite{TG,GT,DS,Ti}, and describe completely the continuity of the entropy function at generic parameters.

\vspace{0.1cm}
\noindent{\bf Keywords and phrases}: core entropy, extended Newton graph, critical marking, polynomial, Newton map.

\vspace{0.1cm}

\noindent{\bf AMS(2010) Subject Classification}: 37B40, 37F10, 37F20.
\end{abstract}
\section{Introduction}

A classical way to measure the topological complexity of a dynamical system is its \emph{entropy}.
In particular, to each real polynomial map $f$ one can associate the topological entropy of $f$ as
a dynamical system on the real line.

The classical entropy theory goes back to the seminal work of Milnor-Thurston \cite{MT},
who proved that the topological entropy of real quadratic polynomials depends continuously and monotonically on the parameter.

In order to generalize the entropy theory to complex polynomials and  develop a ``qualitative picture'' of the parameter space of degree $d$ polynomials, W. Thurston introduce a notion of entropy for complex polynomial around 2010. In his view, the entropy of a complex polynomial $P$ should be defined as the topological entropy of the restriction of $P$ on its ``core'', where a set as a candidate of a core of $P$ should satisfy:
\begin{enumerate}
\item it is compact and connected;
\item it is a minimal $P$-invariant set; and
\item it contains all critical points of $P$.
\end{enumerate}
Such defined entropy is called, by Thurston, the ``core entropy'' of $P$.

 At least in \emph{postcritically-finte} case, i.e., the orbit of each critical point is finite, such cores for polynomials  exist and known as \emph{Hubbard tree}: the minimal invariant tree containing all critical points (see Section \ref{sec:polynomials} below).
This tree completely captures the dynamics of the corresponding polynomial (Poirier \cite[Theorem 1.3]{Poi2}), and hence an appropriate candidate as a core.
By Thurston, the \emph{core entropy} of a polynomial $P$ is defined as the topological entropy of the restriction of $P$ on its Hubbard tree $\HHH_f$ if existing, i.e.,
\begin{equation}\label{eq:core-entropy-polynomial}
h(P):=h_{top}(P|_{\HHH_P}),
\end{equation}

We remark that the invariant segment for a real polynomial is just its Hubbard tree when considering it as a complex one. So the notion of core entropy is indeed a generalization of the topological entropy for real polynomials. Using core entropy, the classical entropy theory of real quadratic polynomial can be generalized to the complex case: the core entropy of quadratic polynomials depends continuously and monotonically along each \emph{vein} of the Mandelbrot set  (\cite{Li}, \cite{Ti},\cite{Ze}).

One can also consider the continuity of the entropy function on the whole space. Let $\PPP_d$ denote the set of monic, centered polynomials of degree $d\geq2$, and $\PPP_d^{\rm pf}$  the subspace of $\PPP_d$ consisting of only \pf ones. Then we get an entropy function
\[h:\PPP_d^{\rm pf}\to\R,\]
mapping $f$ to its core entropy $h(f)$. Thurston asked whether this function is continuous. The answer is YES. The result are independently proved by Tiozzo \cite{Ti} and Dudko-Schleicher \cite{DS} in quadratic case, and proved by Gao-Tiozzo \cite{GT} in general case.

Following Thurston's spirit, in more general case, one can define the core entropy of a rational map if its core is found. Recently,  Lodge,  Mikulich and Schleicher construct the core for any \pf Newton map.

A rational map of degree $d\geq3$ is called a \emph{Newton map} if there exists a polynomial $P$ so that
\[f=z-\frac{P(z)}{P'(z)}\text{ for all $z\in\C$}.\]
We denote by $\NNN_d$ the space of Newton maps of degree $d$, and by $\NNN_d^{\rm pf}$ the subspace of $\NNN_d$ consisting of \pf ones.

 The cases $d<3$ are excluded because they are trivial. Observed that $f$ arises naturally when Newton＊s method is applied to find
the roots of $P$. Each root of $P$ is an attracting fixed point of $f$, and the point at infinity is
a repelling fixed point of $f$. The degree $d$ coincides with the number of
distinct roots of $P$. Since the relation with the root-finding problem, the study of Newton maps became one of the major themes
with general interest, both in discrete dynamical system (pure mathematics), and in root-finding algorithm (applied mathematics), see for example \cite{AR,HSS,Ro,RS,RWY,Tan,WYZ}.

In recent works, Lodge,  Mikulich and Schleicher solve the long-standing classification problem for \pf Newton maps.
The authors first specifically construct for any \pf Newton map $f$ a finite connected $f$-invariant graph which satisfies a sequence of properties, called an \emph{extended Newton graph} (see the proof of Theorem 6.2, Definition 7.3 and the proof of Theorem 1.2 in \cite{LMS1}), and then prove that the \pf Newton maps (up to affine conjugation) can be classified by the equivalence classes of these extended Newton graphs (\cite[Theorems 1.4, 1.5]{LMS2}).

By their results, the extended Newton graphs completely capture the  dynamics  of  \pf Newton maps, analogy to the status of Hubbard trees in \pf polynomial family. Therefore, it is reasonable to consider the extended Newton graph as a ``core'' of a \pf Newton map, and define the \emph{core entropy} of a \pf Newton maps $f$ as the topological entropy of the restriction of $f$ on its extended Newton graph $G_f$, i.e.,
\begin{equation}\label{eq:entropy-Newton}
h(f):=h_{top}(f|_{G_f}).
\end{equation}
The goal of this paper is to study the continuity of the core entropy of \pf Newton maps.

To present the main results of the paper, we give a brief overview of the structure of extended Newton graphs. All materials come from \cite{LMS1}, referring to Section \ref{sec:newton-graph} for details.

Let $f$ be a \pf Newton map of degree $d\geq3$. The \emph{channel diagram} $\De$ of $f$ is the union of the accesses from
finite fixed points of $f$ to $\infty$. The \emph{Newton graph of
level $n$} is constructed to be the connected component of $f^{-n}(\De)$ containing
$\infty$ and is denoted by $\De_n$. For a sufficiently high level $n$, the Newton graph
captures the critical/postcritical points mapping onto fixed points.
The remaining critical/postcritical points, if any,  are disjoint from the Newton graph $\De_n$ for any level $n$.
 Factually, they are captured by another $f$-invariant combinatorial object called  the \emph{canonical Hubbard forest} of $f$. It is the disjoint union of finite trees which contain critical/postcritical point of $f$.

 Thus far, all critical/postcritical points are contained in either the Newton graph
or the Hubbard forest, but the Hubbard forest are disjoint
from the Newton graph. To remedy this, preperiodic \emph{Newton ray} are used to connect
each component of the Hubbard forest to the Newton graph. Therefore, an \emph{extended Newton graph} is a finite connected graph composed of:
\begin{itemize}
\item the Newton graph, which contains all critical/postcritical points mapping to fixed points;
\item canonical Hubbard forest $\HHH_f$, which contains the remaining critical/postcritical points;
\item preperiodic Newton rays connecting each subtree of $\HHH_f$ to the Newton graph.
\end{itemize}

 We will say more about the canonical Hubbard forest. In fact, each non-trivial periodic component $H$ of $\HHH_f$ corresponds a renormalization triple $\rho_H=(f^p,U,V)$ of $f$, where $p$ is the renormalization period of $\rho_H$, such that $H$ is an extended Hubbard tree of $\rho_H$  (see Section \ref{sec:renormalization} for the related concepts).  By Straightening Theorem \cite[Theorem 1]{DH2}, the polynomial-like map $f^p:U\to V$ is hybrid equivalent to a \pf polynomial $P_H\in \PPP_\de$ with $\de:={\rm deg}(f^p|_U)$, and the image of $H$ under the conjugation is an extended Hubbard tree of $P_H$. It follows immediately that $h_{top}(f^p|_H)=h(P_H)$. The polynomial $P_H$ is called a \emph{renormalization polynomial} of $f$ (associated to $H$) and $p$ called the \emph{renormalization period} of $P_H$. We denote
 \begin{equation}\label{eq:renormalization-polynomials}
 \mathfrak{R}(f):=\{P_H:\text{ $H$ is a non-trivial periodic component of $\HHH_f$.}\}
 \end{equation}
 The Newton map $f$ is called \emph{renormalizable} if $\mathfrak{R}(f)\not=\emptyset$ and \emph{non-renormalizable} otherwise.

 The first result in this paper is a  formula to compute the core entropy of \pf Newton maps.

 \begin{proposition}\label{pro:entropy-formula}
 Let $f$ be a \pf Newton map of degree $d\geq3$. Then we have the core entropy formula
 \[h(f)=\left\{
          \begin{array}{ll}
            \max_H \frac{1}{p_{_H}}h_{top}(f^{p_{_H}}|_H)=\max_{P\in\mathfrak{R}(f)}\frac{1}{p}h(P),, & \hbox{if $f$ is renormalizable;} \\[5pt]
            0, & \hbox{otherwise,}
          \end{array}
        \right.
 \]
 where $H$ go though all non-trivial periodic components of $\HHH_f$ with period $p_{_H}$, and $p$ denotes the renormalization period of $P$. The renormalization polynomials achieving the maximum are called \emph{entropy-maximal}.
 \end{proposition}

 Let $f,\ f_n,n\geq1$ be \pf Newton maps of degree $d\geq 3$ such that $f_n\to f$ as $n\to\infty$. To study the continuity of the entropy function, we need to compare $h(f)$ and the limit behavior of $h(f_n)$ as $n\to\infty$. Let $H$ be a non-trivial periodic component  with period $p$ of the canonical Hubbard forest $\HHH_f$, and $\rho_H=(f^p,U,V)$ a renormalization triple associated to $H$. As $f_n\to f$, we know that there exist polynomial-like maps $f^p_n:U_n\to V_n$ for all sufficiently large $n$ such that $f_n^p:U_n\to V_n$ converge to $f^p:U\to V$ as $n\to\infty$. Note that the filled-in Julia set of $(f^p_n,U_n,V_n)$ is not necessarily connected. Let $H_n$ denote the intersection of the canonical Hubbard forest $\HHH_{f_n}$ of $f_n$ and the filled-in Julia set of $(f^p_n,U_n,V_n)$. We get the topological entropy $h_{top}(f^p_n|_{H_n})$. According to Proposition \ref{pro:entropy-formula}, the continuity of the entropy function at $f$ is quite related to the limit behavior of $h_{top}(f^p_n|_{H_n})$ compared with $h_{top}(f^p|_H)$ for every non-trivial periodic component $H$ of $\HHH_f$.

Using Straightening Theorem,  the polynomial-like map $f^p:U\to V$ is hybrid equivalent to a polynomial $P_H\in\PPP_\de$, and $f^p_n:U_n\to V_n$ is hybrid equivalent to a polynomial $P_{H,n}\in\PPP_\de$. By suitable choice of the conjugations, we can assume $P_{H,n}\to P_H$ as $n\to\infty$. Note that $P_H$ is \pf and $h(P_H)=h_{top}(f^p|_H)$, but $P_{H,n}$ is possibly not postcritically-finte: it may have disconnected Julia set although its critical points in the filled-in Julia set have finite orbits, and we haven't defined the core entropy for such polynomials. This motivate us to define the core entropy for some kinds of polynomials with disconnected filled-in Julia sets, including these $P_{H,n}$, so that $h(P_{H,n})=h_{top}(f^p_n|_{H_n})$.

A polynomial is called \emph{partial \pf} if its critical points in the filled-in Julia set have finite orbits. For any partial \pf polynomial $P$, we define its \emph{Hubbard forest} $\HHH_P$ to be the minimal $P$-invariant forest containing all critical points in the filled-in Julia set, and define its core entropy $h(P)$ as the topological entropy of the restriction of $P$ on $\HHH_P$, i.e.,
\[h(P):=h_{top}(P|_{\HHH_P}).\]
 Note that $\HHH_P=\emptyset$ if and only if all critical points escape to $\infty$, and we define $h(P)=0$ in this case. We denote by $\PPP_d^{\rm ppf}$ the subspace of $\PPP_d$ consisting of all partial \pf polynomials. Then $\PPP_d^{\rm pf}\subseteq \PPP_d^{\rm ppf}$ and the definition domain of the entropy function is enlarged from $\PPP_d^{\rm pf}$ to $\PPP_d^{\rm ppf}$.

 It is easy to check that the polynomials $P_{H,n}, n\geq1$ discussed above are  partial \pf and $h(P_{H,n})=h_{top}(f^p_n|_{H_n})$. It then follows from the discussion above that the continuity at $f$ of the entropy function within the \pf Newton family has a close relationship to the limit behavior of $h(P_{H,n})$ compared with $h(P_H)$. In other words, to study the continuity of the core entropy of \pf Newton maps, we should first make clear of the continuity of the entropy function $h:\PPP_d^{\rm ppf}\to\R$ at \pf parameters. The following  result completely solve this problem, and it plays an essential role in this paper.

 \begin{proposition}\label{pro:key}
 Let $P$ be a \pf polynomial in $\PPP_d$. Then the entropy function $h:\PPP_d^{\rm ppf}\to\R$ is upper semi-continuous at $P$, i.e.,
 \[\limsup_{\PPP_d^{\rm ppf}\ni Q\to P}h(Q)\leq h(P).\]
 Furthermore, we can find a $P$-invariant forest $\HHH_P^*$ in $\HHH_P$ such that
\[\liminf_{\PPP_d^{\rm ppf}\ni Q\to P}h(Q)=h_{top}(P|_{\HHH_P^*}).\]
As a consequence,
the function $h$ is continuous at $P$ if and only if $h(P)=h_{top}(P|_{\HHH_P^*})$.
 \end{proposition}

 We remark that by Gao-Tiozzo's result \cite{GT}, the entropy function is continuous within the \pf polynomial family. But the proposition above shows that in a larger family $\PPP_d^{\rm ppf}$, the entropy function is not continuous, for example at a \Mi parameter $P$ with positive core entropy, since $h_{top}(P|_{\HHH_P^*})=0$ in this case.

 As discussed before,  we can learn the continuity of the core entropy of \pf Newton maps by application of Proposition \ref{pro:key} and Straightening Theorem. In this process,  a method of perturbing rational maps, called \emph{capture surgery} (see Section \ref{sec:capture}), is also needed.
We now state the main result in this paper. A \pf Newton map is called \emph{generic} if the orbits of its critical points avoid $\infty$.

\begin{theorem}\label{thm:main2}
Let $f$ be a generic \pf Newton map $f$ of degree $d$. Then the entropy function $h:\NNN_d^{\rm pf}\to\R$ is upper semi-continuous at $f$, i.e.,
\[\limsup_{\NNN_d^{\rm ppf}\ni g\to f}h(g)\leq h(f).\]
Furthermore, if $f$ is non-renormalizable, then the function $h$ is continuous at $f$; otherwise,
 the entropy function  is continuous at $f$ if and only if:
 there exists an entropy-maximal renormalization polynomial $P\in \mathfrak{R}(f)$ such that  the entropy function $h:\PPP_\de^{\rm ppf}\to\R$ is continuous at $P$ with $\de:={\rm deg}(P)$.
\end{theorem}

This theorem only tell us the continuity of the entropy function $h:\NNN_d^{\rm pf}\to \R$ at generic parameters, but we conjecture that it is correct for all parameters

\begin{conj}
Theorem \ref{thm:main2} holds for any \pf Newton map $f$.
\end{conj}

We are able to prove the conjecture in the cubic case, and it display a simple form.
\begin{theorem}\label{thm:main3}
The entropy function $h:\NNN_3^{\rm pf}\to\R$ is continuous at $f$ if and only if the map $f$ is either hyperbolic or non-renormalizable.
\end{theorem}

For the completeness of the paper, we finally describe, as a byproduct, of the continuity of the entropy function $h:\PPP_d^{\rm ppf}\to\R$ at all parameters, not just at \pf ones as given in Proposition \ref{pro:key}. The statement is parallel to that of Theorem \ref{thm:main2} since
the partial \pf polynomials are very similar to the \pf Newton maps in the view of entropy: both of them have the Hubbard forests on which the core entropy concentrate, and every non-trivial periodic component of the Hubbard forest induces  a renormalization triple.

Let $P$ be a partial \pf polynomial. We call $P$ \emph{renormalizable} if $\KKK_P$ has non-trivial components, and \emph{non-renormalizable} otherwise. Suppose that $P$ is renormalizable. Then for any non-trivial periodic component $K$ of $\KKK_P$ with period $p$, there exists a monic,centered polynomial $P_K$ of degree ${\rm deg}(P^p|_K)$ such that the restriction of $P^p$ on $K$ conjugates to the restriction of $P_K$ on its filled-in Julia set. Using these notations, we state the following result.

\begin{theorem}\label{thm:main1}
Let $P$ be any partial \pf polynomial in $\PPP_d^{\rm ppf}$. Then the entropy function $h:\PPP_d^{\rm ppf}\to\R$ is upper semi-continuous at $P$.
 Furthermore, if $P$ has connected Julia set, the continuity of $h$ at $P$ is characterized by Proposition \ref{pro:key}; in the case that $\JJJ_P$ is not connected, if $P$ is non-renormalizable, the function $h$ is continuous at $P$, otherwise, the entropy function is continuous at $P$ if and only if : there exists a non-trivial component $K$ of $\KKK_P$ of period $p$ such that $h(P)=h(P_K)/p$ and the entropy function $h:\PPP_\de^{\rm ppf}\to\R$ is continuous at $P_K$, where $\de$ denotes the degree of $P_K$.
\end{theorem}

\noindent{\bf Organization of the paper.} The paper is organized as follows:

In Section $2$, we summarize some basic facts used in the paper.
\vspace{-5pt}

 In Section $3$, we develop a method to perturb a sub-hyperbolic rational map to a hyperbolic one, called \emph{capture surgery}, which is used in the proof of Proposition \ref{pro:key} and Theorem \ref{thm:main2}.
\vspace{-5pt}

In Section $4$, we study the convergence of internal/external rays when perturbing rational maps.
\vspace{-5pt}

In Section $5$, we will consider the dynamics of partial \pf polynomials. We focus on the critical markings of partial \pf polynomials, including its  construction, properties and convergence.
\vspace{-5pt}

In Section $6$, we describe the continuity of the entropy function $h:\PPP_d^{\rm ppf}\to\R$ at the \pf parameters and prove Proposition \ref{pro:key}. This is the key part of the paper. The proof of Proposition \ref{pro:key} is divided into two parts: we first show that the entropy function within the partial \pf polynomial family is upper semi-continuous at any \pf parameter $P$; and then construct a sequence $\{P_n,n\geq1\}\subseteq \PPP_d^{\rm ppf}$ converging to $P$ by capture surgery such that the limit of their core entropies obtain the
limit inferior $\liminf_{Q\to P}h(Q)$.
\vspace{-5pt}

In Section $7$, we introduce the concept of core entropy for \pf Newton maps, and prove Proposition \ref{pro:entropy-formula}, which provides a formula for the computation of the core entropy.
\vspace{-5pt}

In Section $8$, we explore the continuity of the entropy function on \pf Newton family. We prove Theorem \ref{thm:main2} by applying Proposition \ref{pro:key} and Straightening Theorem, and testify Theorem \ref{thm:main3} with the help of Theorem \ref{thm:main2} and some specific properties of cubic Newton maps given in \cite{Ro}.
\vspace{-5pt}

Finally, in Section $9$, we will give an outline of the proof of Theorem \ref{thm:main1}, which is parallel to the proof of Theorem \ref{thm:main2}.

{\noindent \bf Acknowledgement.} I would like to thank Laura De Marco for the very useful discussion
and suggestions.  The author is supported by NSFC grant no. 11871354.

\section{Preliminary}
\subsection{Topological entropy on graphs}\label{sec:topological-entropy}
We will not use the general definition of the topological entropy in the paper (see \cite{AKM}). Instead, we summarize some basic results about the topological entropy that will be applied below.

Let $h_{top}(f|_X)$ denote the topological entropy of $f$ on $X$. Throughout the paper, for the uniform of the statement, we stipulate {\bf $h_{top}(f|_\emptyset):=0$.}
The following  propositions can be found in \cite{Do}.

\REFPROP{Do1}
For any positive integer $k$ we have $k\cdot h_{top}(f|_X)=h(f^k|_X)$.
\ENDPROP

\REFPROP{Do2}
If $X=X_1\cup X_2$, with $X_1$ and $X_2$ compact, $f(X_1)\subset X_1$ and $f(X_2)\subset X_2$, then $h_{top}(f|_X)=\sup\bigl(h_{top}(f|_{X_1}),h_{top}(f|_{X_2})\bigr)$. \ENDPROP

\REFPROP{Do3}
Let $Z$ be a closed subset of $X$ such that $f(Z)\subset Z$. Suppose that for any $x\in X$, the distance of $f^n(x)$ to $Z$ tends to $0$, uniformly on any compact set in $X-Z$. Then $h_{top}(f|_X)=h_{top}(f|_Z)$. \ENDPROP

\REFPROP{Do4} Assume that $\pi$ is a  surjective semi-conjugacy $$\begin{array}{rcl}Y &\xrightarrow[]{\ g\ }  &Y\\ \pi\Big\downarrow &&\Big\downarrow \pi \vspace{-0.1cm} \\ X &  \xrightarrow[]{\ f\ } & X\vspace{-0.1cm}.\end{array}$$
Then $h_{top}(f|_X)\leq h_{top}(g|_Y)$. Furthermore, if  $\ds \sup_{x\in X}\#\pi^{-1}(x)< \infty$ then  $ h_{top}(f|_X)=h_{top}(g|_Y)$.\ENDPROP


A (finite topological) \emph{graph} $G$ is a compact  Hausdorff space which contains a
finite non-empty set $V_G$, called \emph{vertex set} of $G$, such that every connected component
of $G\setminus V_G$ is homeomorphic to an open interval of the real line. The closure of each component of $G\setminus V_G$ is called an \emph{edge} of $G$.
There is a special kind of connected graphs, called \emph{trees}, which are graphs without cycles. A finite disjoint union of trees is called a \emph{forest}.

Let $G$ be a finite graph with vertex set $V_G$.
A continuous map $f:G\to G$ is called a \emph{Markov map} if  $f(V_G)\subseteq V_G$  and the restriction of $f$ on each edge is injective.

Let $f:G\to G$ be a Markov graph map. By the definition,  any edge of $G$ is mapped to the union of several edges of $G$.
Enumerate the edges of  $G$  by $e_i$ , $i=1,\cdots,k$. We then obtain an \emph{incidence matrix} $D_{(G,f)}=(a_{ij})_{k\times k}$ of $(G,f)$
 such that  $a_{ij}=1$ if $f(e_i)$ covers $e_j$ and $0$ otherwise. Note that choosing  different enumerations of the edges gives rise to conjugate incidence matrices, so in particular, the eigenvalues are independent of the choices.

 Denote by $\lambda$ the greatest non-negative eigenvalue of $D_{(G,f)}$. By the Perron-Frobenius theorem  such an eigenvalue exists
  and equals
the growth rate of $\|D_{(G,f)}^n\|$ for any matrix norm.
The following result is classical (see \cite{AM,MS}):

\begin{proposition}\label{entropy-formula}
The topological entropy $h_{top}(f|_G)$ is equal to $0$ if $D_{(G,f)}$ is nilpotent, i.e., all eigenvalues of $D_{(G,f)}$ are zero; and equal to $\log\lambda$ otherwise.
\end{proposition}

\subsection{Dynamics of polynomials and rational maps}\label{sec:polynomials}

\subsubsection{General dynamical properties of rational maps}
Let $f$ be a rational map of degree $d\ge2$. We denote by $\JJJ_f$ and $\FFF_f$ the Fatou set and Julia set of $f$ respectively. Let ${\rm Crit}_f$ denote the set of critical points of $f$, and the set
\[{\rm Post}_f:=\{f^n(c):c\in{\rm Crit}_f,n\geq1\}\]
is called the \emph{postcritical set} of $f$. A point $z\in\ov{\C}$ is called a \emph{Fatou point/Julia point} if $z\in\FFF_f/\JJJ_f$, and called a \emph{pre-critical point} if $f^n(z)$ is a critical point of $f$ for some $n\geq0$.

The map $f$ is called \emph{\pf} if $\#{\rm Post}_f<\infty$; called \emph{hyperbolic} if all its critical points are attracted by the attracting cycles; and called \emph{sub-hyperbolic} if all its Fatou critical points are attracted by the attracting cycles and all its Julia critical points have finite orbits.

\begin{lemma}\label{lem:sub-hyperbolic}
Let $f$ be a hyperbolic (resp. sub-hyperbolic) rational map. Then there exists a conformal metric (resp. orbifold metric) $\omega$ in a neighborhood of $\JJJ_f$ such that $f$ is \emph{uniformly expanding} with respect to $\omega$, i.e., $\exists \lambda>1, \text{ s.t }||Df_z||_\omega>\lambda$ for all $z$ in this neighborhood (resp. whenever $z$ and $f(z)$ are not ramified points). Moreover, each connected component of $\JJJ_f$ is locally-connected.
\end{lemma}
\begin{proof}
The existence of uniformly expanding (orbifold) metric can be found in \cite[Section 19]{Mil}, and the local connectivity of Julia components is duo to the results of McMullen \cite{Mc2} and Tan-Pilgrim \cite{PT}.
\end{proof}

Let $f$ be a rational map, and $U$ a periodic Fatou component of $f$. Assume that $U$ contains a unique postcritical point of $f$. Then $U$ must be simply connected and there is a system of Riemann mappings
$$\Big\{\phi_W: \D\to W\,\Big|\, W \text{ Fatou component with $f^n(W)=f^m(U)$ for some $m,n\geq0$}\Big\}$$
so that  the following diagram commutes for all $W$:
\begin{equation*}
 \begin{tikzpicture}
   \matrix[row sep=0.8cm,column sep=2.4cm] {
     \node (Gammai) {$ \D $}; &
       \node (Gamma) {$ \D$}; \\
     \node (S2i) {$ W$}; &
       \node (S2) {$f(W)$,}; \\
   };
   \draw[->] (Gamma) to node[auto=left,cdlabel] {\phi_{f(W)}} (S2);
   \draw[->] (S2i) to node[auto=right,cdlabel] {f} (S2);w
   \draw[->] (Gammai) to node[auto=left,cdlabel] {\text{ power map }z^{d_{\tiny\mbox{$W$}}}} (Gamma);
   \draw[->] (Gammai) to node[auto=right,cdlabel] {\phi_W} (S2i);
 \end{tikzpicture}
 \end{equation*}
where $d_W$ denotes the degree of $f$ on $W$. The conformal map $\phi_W^{-1}$ is called a \emph{B\"{o}ttcher coordinate} of $W$, the image $\phi_W(0)$ is called the \emph{center} of $W$ and the image in $W$ under $\phi_W$ of radial lines in $\D$ are called \emph{internal rays} of $W$. Note that $f$ maps internal rays of $W$ to those of $f(W)$.

\subsubsection{Dynamics of Polynomials}
Let $P$ be a monic polynomial of degree $d\geq2$. Basic tools to understand
the dynamics of $P$ are the Green function $g_P$ and the B\"{o}ttcher map $\phi_P$.

The \emph{Green function} $g_P$ of $P$ is defined as
\[g_P(z):=\lim\log^+|P^n(z)|/d^n,\text{ for all $z\in\C$}.\]
It is a well-defined continuous function which vanishes on the \emph{filled-in Julia set}
\[\KKK_P:=\{z\in\C:P^n(z)\not\to\infty\text{ as $n\to\infty$}\},\]
and satisfies the functional relation $g_P(P(z))=dg_P(z)$. In the basin of infinity $\Omega(P):=\C\setminus\KKK_P$, the
derivative of $g_P$
vanishes at $z$ if and only if $z$ is a pre-critical points of $P$. We say that $z$ is a singularity of $g_P$.

The B\"{o}ttcher map $\phi_P$
conjugates $P$ with $z\mapsto z^d$
in a neighborhood of $z$.  Since $P$ is monic, we can normalize $\phi_P$ so that $\phi_P(z)/z\to 1$ as $z\to\infty$. We extend, along flow lines, to the basin of infinity under the gradient flow $\nabla g_P$. Following Levin and Sodin \cite{LS}, we define the \emph{reduced basin of
infinity} $\Omega^*(P)$ to be the maximal basin in which the gradient flow $\nabla g_P$ is smooth. Now
\begin{equation}\label{eq:reduced-basin}
\phi_f:\Omega^*(P)\to D_P\subseteq \C\setminus \ov{\D}
\end{equation}
is a conformal isomorphism from $\Omega^*(P)$
onto a starlike (around $\infty$) domain $D_P$. We have the commutative diagram:
\[\begin{array}{ccc}\Omega^*(P) &\xrightarrow[]{\ P\ }  &\Omega^*(P) \\ \phi_P\Big\downarrow &&\Big\downarrow \phi_P  \\ D_P &  \xrightarrow[]{\ z\mapsto z^d\ } & D_P\vspace{-0.1cm}.\end{array}\]

A flow line of $\nabla g_P$ in $\Omega^*(P)$ is called an \emph{external radius}. External radii are parameterized $\T=\R/\Z$. More precisely, for $t\in\T$, let
$(r,\infty)e^{2\pi it}$
be the maximal portion $(1,\infty)e^{2\pi it}$
contained in $D_P$. The \emph{external
radius of argument $t$} is defined as
\[\RRR^*_P(t):=\phi_P^{-1}\big((r,\infty)e^{2\pi it}\big).\]
In the case of $r>1$, the radius  terminate at a singularity of $g_P$. While in the case of $r=1$, the radius accumulate at the Julia set, and we call it a \emph{smooth external ray}, writing also as $\RRR_P(t)$.

Let $\tau=\tau_d:\T\to\T$ denote the mapping which sends $\theta$ to $d\cdot\theta\ ({\rm mod}\Z)$.
Now let $\theta_1,\ldots,\theta_r$ be the arguments of the external radii that terminate at
critical points of $P$. Since every pre-critical point of $P$ is a singularity of $g_P$, the
external radii with arguments in
\[\Sigma=\bigcup_{n\geq0}\tau^{-n}_d(\{\theta_1,\ldots,\theta_r\})\]
also terminate at singularities. Since every singularity is a pre-critical point, we
have smooth external rays defined for arguments in $\T\setminus\Sigma$.
Following Goldberg
and Milnor, for $t\in\T$, let
\[\RRR_P^{\pm}(t):=\lim_{\T\setminus\Sigma\ni s\to t^{\pm}}\RRR_P(s).\]
If $t\not\in \Sigma$ then $\RRR_P^{\pm}(t)=\RRR_P(t)$ is a smooth external ray. If $t\in\Sigma$, then $\RRR_P^{\pm}(t)$ do not agree, and we say that they are \emph{non-smooth external ray}. By \emph{external rays} we mean either smooth or non-smooth external rays. An external ray is called \emph{landing} at $z$ if its accumulation set on $\JJJ_P$ is $\{z\}$. We denote by ${\rm arg}_P(z)$ the set of angles such that $\theta\in{\rm arg}_P(z)$ if and only if  an external ray $\RRR_P^\sigma(\theta)$ lands at $z$, with $\sigma\in\{+,-\}$.

\begin{definition}[supporting ray/argument]\label{def:support}
Let $U$ be a bounded Fatou component of a polynomial $P$, and $z \in \partial U$.  The external rays landing at $z$ (if existing) divide the
plane into finite regions. We label the arguments of these rays by $\theta_1,\ldots, \theta_k$ in counterclockwise cyclic order,
 so that $U$ belongs to the region delimited by $\RRR_P^{\sigma_1}(\theta_1)$
and $\RRR_P^{\sigma_2}(\theta_2)$ with $\sigma_1,\sigma_2\in\{+,-\}$. The ray $\RRR_P^{\sigma_1}(\theta_1)$ (resp. $\RRR_P^{\sigma_2}(\theta_2)$) is called the \emph{left-supporting} (resp. \emph{right-supporting}) ray of $U$ at $z$, and the argument $\theta_1$ (resp. $\theta_2$) is called the \emph{left-supporting} (resp. \emph{right-supporting}) argument of $U$ at $z$.
\end{definition}
\begin{figure}[htpb]
\centering
\includegraphics[width=2.6in]{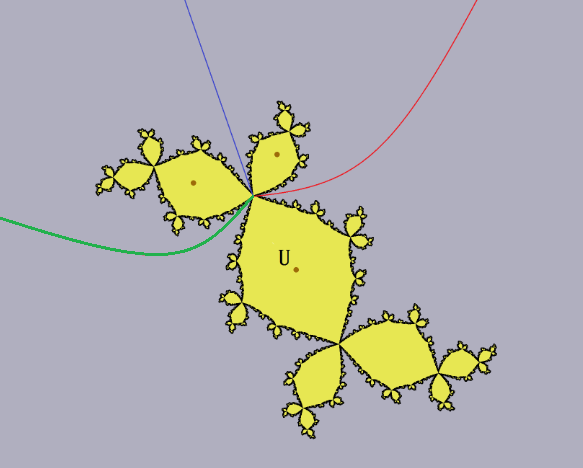}
\caption{The red and green rays are right and left supporting rays of the Fatou component $U$, but the blue one does not support $U$.} \label{supporting-ray}
\end{figure}

According to Levin and
Przytycki \cite{LP}, the landing theorem stated for connected Julia sets generalizes as follows.

\begin{proposition}\label{pro:landing-theorem}
Let $P$ be a polynomial. Then every periodic external ray lands at a parabolic or repelling periodic point. On the other hand, let $z$ be a parabolic or repelling periodic point. Then there
exists at least one external ray landing at $z$. Moreover, either
\begin{enumerate}
\item all the external rays, smooth and non-smooth, landing at $z$ are periodic of
the same period, or
\item the arguments of the external rays, smooth and non-smooth, landing at $z$
are irrational and form a Cantor set, and moreover, $\{z\}$ is a connected
component of $\JJJ_P$ and there are non-smooth rays landing at $z$;
furthermore, there exists a non-smooth ray containing a critical point
landing at a point in the orbit of $z$.
\end{enumerate}
\end{proposition}

Recall that $\Omega^*(P)$ is the reduced basin of infinity of $P$. According to Levin and Sodin \cite{LS}, the set $\JJJ_P^*:=\partial \Omega^*(P)$ is called the \emph{extended Julia set} of $P$. It is known that $\JJJ^*_P$ is a connected set such that $\JJJ_P\subseteq\JJJ_P^*$ and $P^{-1}(\JJJ_P^*)\subseteq \JJJ_P^*$.
As all polynomials in the paper are sub-hyperbolic, we include here the following result (see \cite[Propositions 2.1,2.2]{LS}).

\begin{proposition}\label{pro:subhyperbolic}
Let $P$ be a sub-hyperbolic polynomial. Then every external ray, smooth and non-smooth, lands at the Julia set. Moreover, the extended Julia set $\JJJ^*_P$ is locally connected such that every external radius terminates or lands at a point in $\JJJ^*_P$.
\end{proposition}

We remark that if $P$ has connected Julia set, the B\"{o}ttcher map $\phi_P$ extends to the whole $\Omega(P)$, so that $\Omega^*(P)=\Omega(P)$ and all external rays are smooth.

\subsubsection{Postcritically-finite polynomials and the Hubbard trees}
Let $P$ be a \pf polynomial.
Then it has connected and locally arc-connected filled-in Julia set.
Since $\KKK_P$ is arc-connected, given two points $z_1, z_2\in \KKK_P$, there is an arc $\gamma: [0,1]\to \KKK_P$ such that $\gamma(0)=z_1$ and $\gamma(1)=z_2$.  It is proved in \cite{DH1} that the arc $\g$ can be chosen in a unique way so that  the intersection with the closure of a Fatou component consists of segment of internal rays. We call such an arc \emph{regulated} and denote it by $[z_1,z_2]$.
By \cite[Proposition 2.7]{DH1}, the set
\[\HHH_P:=\cup_{p,q\in {\rm Post}_P\cup{\rm Crit}_P}[p,q] \]
 is a finite connected tree, called the \emph{Hubbard tree} of $P$. The vertex set $V(\HHH_P)$ consists of the critical/postcritical points of $P$ and the branched points of $\HHH_P$. It is well known that $\HHH_P$ is $P$-invariant and $P:\HHH_P\to\HHH_P$ is Markov (see \cite[Section 1]{Poi2}). Furthermore, Poirier proved that the \pf polynomials can be dynamically classified by their Hubbard trees.

Following Thurston, the \emph{core entropy} of $P$ is defined as $h(P):=h_{top}(P|_{\HHH_P})$. A regulated tree within $\KKK_P$ is called an \emph{extended Hubbard tree} if it is $P$-invariant and contains $\HHH_P$.
\begin{lemma}\label{lem:extended-tree}
Let $T$ be any extend Hubbard tree of a \pf polynomial $P$. Then we have $h_{top}(P|_T)=h(P)$.
\end{lemma}
\begin{proof}
Since every point, except the endpoints, of $T$ are iterated to $\HHH_P$, the lemma follows directly from Lemma \ref{Do3}.
\end{proof}

\subsection{Critical portraits and entropy algorithm}\label{sec:weak-critical-marking}
To compute the core entropy of  polynomials,
 W. Thurston
developed a purely combinatorially algorithm (avoid knowing the topology of Hubbard trees) using the
combinatorial data called \emph{critical portraits}.

A finite collection $\Theta:=\{\ell_1,\ldots,\ell_m\}$ of finite subsets of the unit circle is called a \emph{critical portrait of degree $d$} if
\begin{enumerate}
\item each $\tau_d(\ell_i)$ is a singleton, $1\leq i\leq m$;
\item the convex hulls $hull(\ell_i),hull(\ell_j)$ in the closed unit disk intersect at most at one point of $\T$ for any $i\not=j\in\{1,\ldots,m\}$;
\item each $\#\ell_i\geq 2$, and $\sum_{i=1}^{m}(\#\ell_i-1)=d-1$.
\end{enumerate}
\begin{figure}[http]
\begin{center}
 \includegraphics[scale=0.295]{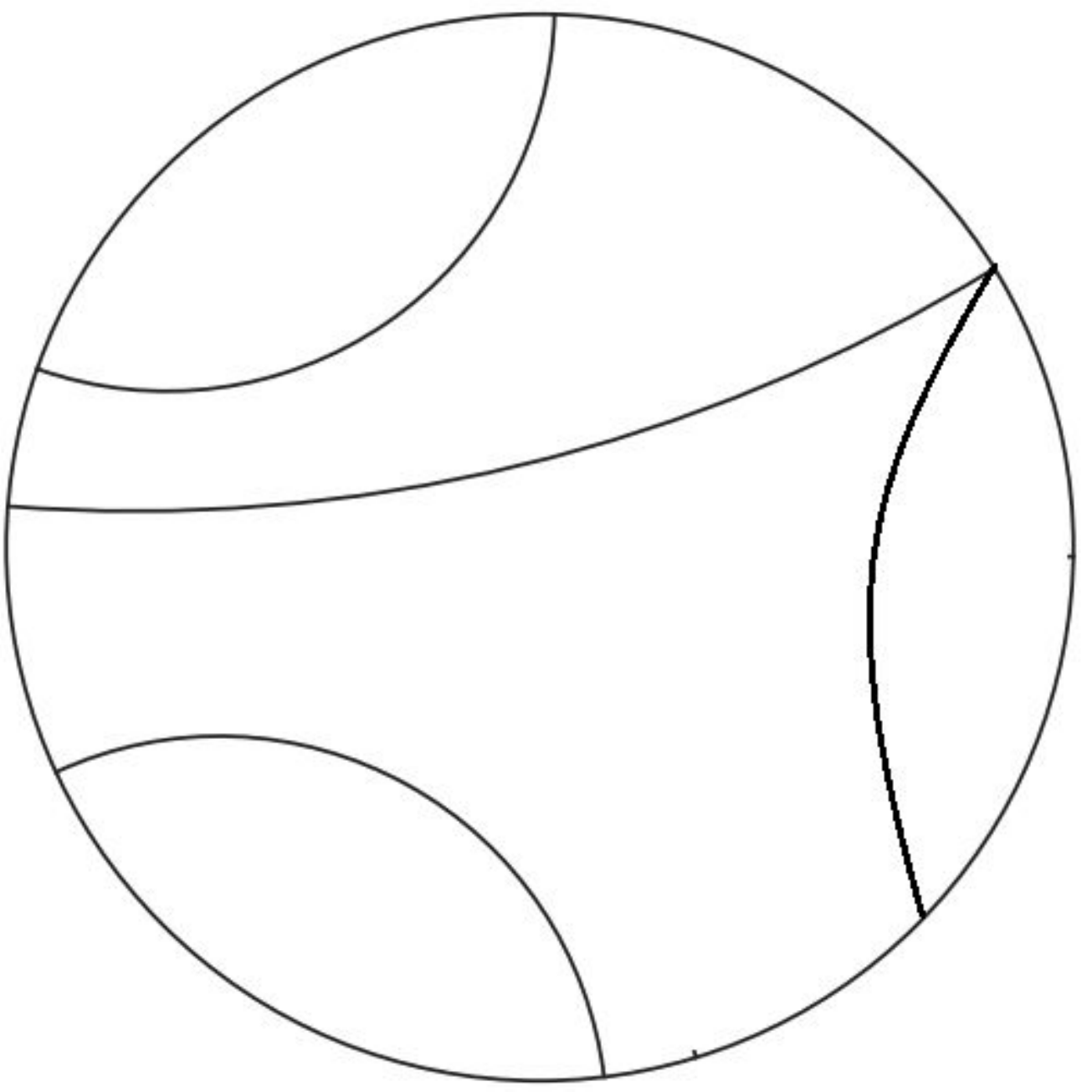}  \quad \includegraphics[scale=0.40]{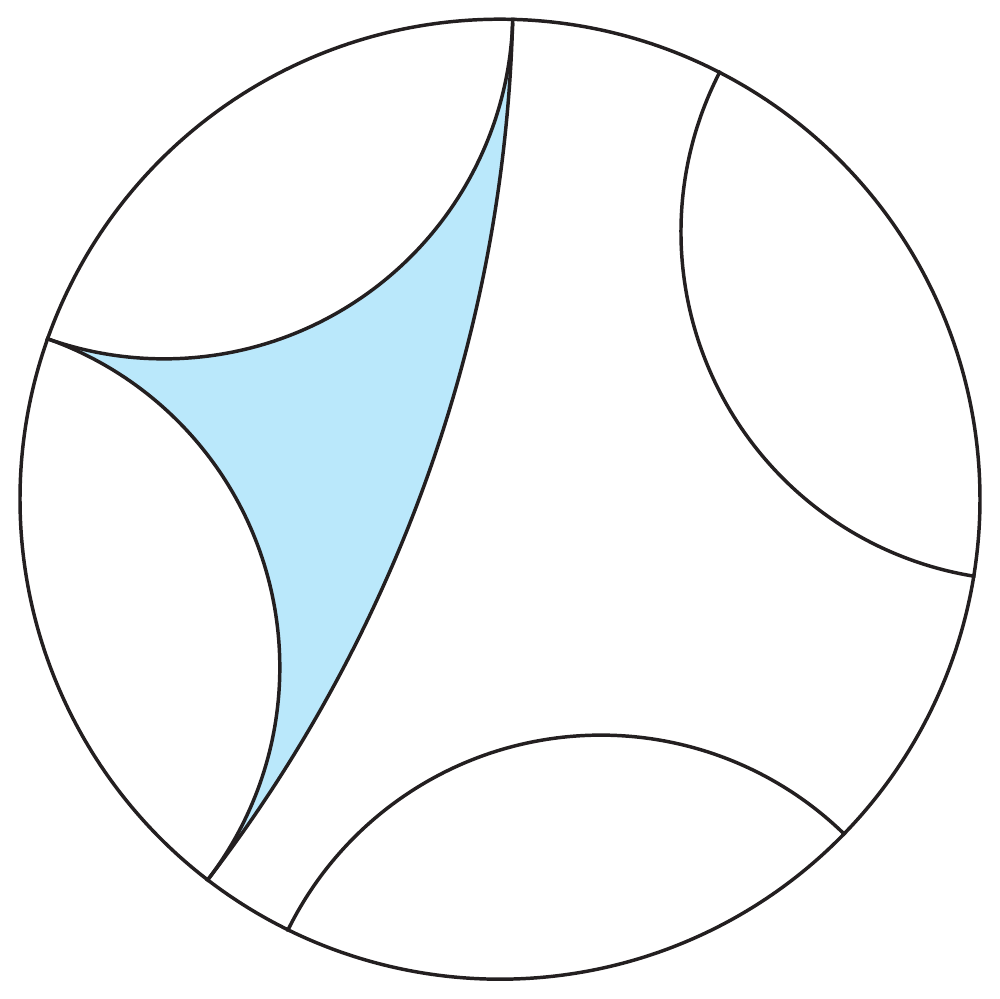}
  \caption{Two  critical portraits of degree $5$}\label{critical-portrait-5}
  \end{center}
\end{figure}
\emph{Thurston's entropy algorithm} endues every critical portrait $\Theta$ a non-negative real number $h(\Theta)$, called the \emph{output} of the algorithm (see \cite[Section 5.2]{G} for detail). Gao-Tiozzo prove that $h(\Theta)$ varies continuously with respect to $\Theta$.

\begin{proposition}[\cite{GT}, Theorem 1.1]\label{pro:GT}
Let $\{\Theta_n,n\geq1\}$ be a sequence of critical portraits Hausdorff converging to $\Theta$. Then $\lim\limits_{n\to\infty}h(\Theta_n)=h(\Theta)$.
\end{proposition}

In fact, any \pf polynomial $P$ of degree $d$ induces a critical portrait $\Theta_P$ of degree $d$, which connects the core entropy of $P$ and the output $h(\Theta)$ of Thurton's entropy algorithm. To show the construction of $\Theta_P$, we first define $\Theta(U)$  as follows for each bounded critical Fatou component $U$.
Denote $\delta_U=\text{deg}(f|_U)$.

\begin{itemize}
\item \textbf{Case 1}: We first consider the case when $U$ is a periodic, critical Fatou component. Let
\[U\mapsto f(U)\mapsto\cdots\mapsto f^n(U)=U \]
be a critical Fatou cycle of period $n$. We will construct the associated set $\Theta(U')$ for every critical Fatou component $U'$ in this cycle simultaneously.
Let $z\in\partial U$ be a periodic point  with period less than or equal to $n$. Let $\theta$ denote the left-supporting argument of $U$ at $z$. Clearly, $\theta$ is periodic with period $n$. We call $\theta$ a \emph{preferred} angle for $U$.
 Note that this choice
naturally determines a left-supporting argument of each Fatou component $f^k(U)$ for $k\in\{0,\ldots,n-1\}$, which is called a \emph{preferred angle} of $f^k(U)$.
Let $U'$ be a critical Fatou component in the cycle and $\theta'$  its preferred angle.
We now define $\Theta(U')$ as any set of $\delta_{U'}$ angles such that:
\begin{itemize}
\item[(a)]
$\theta' \in \Theta(U')$;
\item[(b)]
the rays corresponding to the elements of $\Theta(U')$ land at $\delta_{U'}$ distinct points of $\partial U'$ and are inverse images
of $f(R(\theta'))$.
\end{itemize}
\item \textbf{Case 2}: $U$ is a strictly preperiodic Fatou component. Let $k$ be the minimal number such that $U'=f^k(U)$ is a critical Fatou component. We may assume that $\Theta(U')$ is already chosen, according to the previous case. Choose an angle $\theta' \in \Theta(U')$. We define $\Theta(U)$ to be the set of arguments of the $\delta_{U}$  rays landing at $\delta_U$ distinct points of $\partial U$ that are $k$-th inverse images
of $R(\theta')$.
\end{itemize}

\begin{definition}[weak critical marking]\label{def:weak-critical-marking}
A collection of finite subsets of $\T$
\begin{equation}\label{eq:weak-critical-marking}
\Theta_P:=\{\Theta_1(c_1),\ldots,\Theta_m(c_m);\Theta(U_1),\ldots,\Theta(U_s)\}
\end{equation}
is called a \emph{weak critical marking} of $P$ if
\begin{enumerate}
\item $\Theta_P$ is a critical portrait of degree $d$;
\item each $\Theta(U_k)$ is defined as above for $k\in\{1,\ldots,n\}$ ($U_1,\ldots,U_s$ are pairwise distinct);
\item each $c_j,1\leq j\leq m$, is a Julia critical point of $P$ ($c_1,\ldots,c_m$ are not necessary pairwise distinct), and $\#\Theta_j(c_j)\geq2$ for every $j=1,\ldots,m$;
\item for each critical point $c\in \JJJ_f$,
\[{\rm deg}(f|_c)-1=\sum_{c_j=c}\big(\#\Theta_j(c_j)-1\big). \]
\end{enumerate}
We call $\FFFF_P:=\{\Theta(U_i):1\leq i\leq s\}$ a \emph{weak Fatou critical marking} of $P$, and $\JJJJ_P:=\{\Theta_j(c_j):1\leq j\leq m\}$ a \emph{weak Julia critical marking} of $P$ (see Figure \ref{fig:weak-portrait}).
\end{definition}
\begin{figure}[http]
\centering
\begin{tikzpicture}
\node at (-4,0){\includegraphics[width=6.8cm]{right.pdf}};
 \node at (3.5,0){\includegraphics[width=6.2cm]{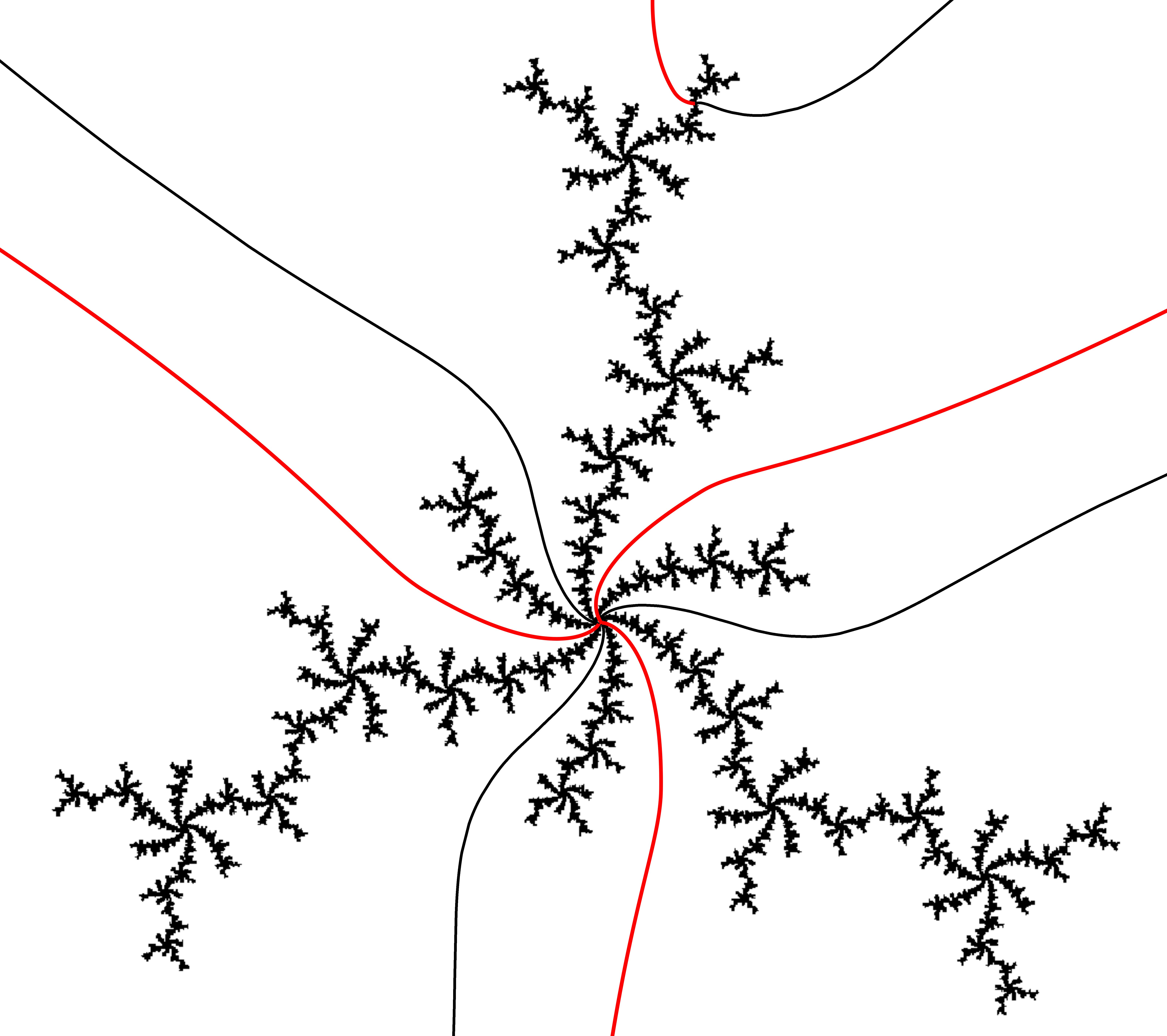}};
 \node at (-0.25,0){\footnotesize{$0$}};
 \node at (-5.25,2.5){\footnotesize{$\frac{1}{3}$}};
 \node at (-5.25,-2.5){\footnotesize{$\frac{2}{3}$}};
 \node at(3.6,2.5){\footnotesize{$\frac{17}{72}$}};
 \node at(5.5,2.5){\footnotesize{$\frac{11}{72}$}};
 \node at(0,2.5){\footnotesize{$\frac{83}{216}$}};
 \node at (0.5,1){\footnotesize{$\frac{89}{216}$}};
 \node at (2.5,-2.5){\footnotesize{$\frac{155}{216}$}};
 \node at (4,-2.5){\footnotesize{$\frac{161}{216}$}};
 \node at (6.75,0){\footnotesize{$\frac{11}{216}$}};
 \node at (6.75,1){\footnotesize{$\frac{17}{216}$}};
 \node at (-2.8,0){$U$};
\end{tikzpicture}
\caption{In the left picture, the polynomial $P(z)=z^3-3z^2/2$ has a fixed critical Fatou component $U$ and another critical point  $-1\in\partial U$. A weak critical marking of $P$ is $\Theta_P=\bigl\{\ \Theta(U)=\{0,1/3\},\ \Theta(-1)=\{1/3,2/3\}\ \bigr\}$. The right one shows the Julia set of $P(z)=z\mapsto z^3+0.22036+1.18612 i$, and the external rays landing at its critical point and critical value. One can choose $\Theta_P=\large\{\ \Theta_1(0):=\left\{11/216,83/216\right\},\Theta_2(0):=\left\{89/216,161/216\right\}\ \large\}$ as a weak critical marking of $P$.}
\label{fig:weak-portrait}
\end{figure}

The concept of critical markings of polynomials was introduced by Bielefeld-Fisher-Hubbard \cite{BFH} for \Mi \pf polynomials, and then generated by  Poirier \cite{Poi} to the general \pf case, to classify the dynamics of \pf polynomials. Our definition of weak critical marking,  first given in \cite{G} (also in \cite{GT}), is much less restrictive than that of critical marking (see \cite{Poi} or Section \ref{sec:critical-marking}) because we just use it to compute core entropy. The following result was assured by Thurston and proved by Gao.

\begin{proposition}[\cite{G}, Theorem 1.2]\label{pro:algorithm}
Let $P$ be a \pf polynomial and $\Theta_P$ a weak critical marking of $P$. Then the core entropy $h(P)$ of $P$ equals to the output $h(\Theta_P)$  in Thurston's entropy algorithm.
\end{proposition}



\subsection{Polynomial-like maps and renormalization}\label{sec:renormalization}
Polynomial-like maps
were introduced by Douady and Hubbard \cite{DH2} and have played an important
role in complex dynamics ever since.

A \emph{polynomial-like} map of degree $d\geq2$ is a triple $(g,U,V)$
where $U,V$ are topological disks in $\C$ with $\ov{U}\subseteq V$,
and $f:U\to V$ is a holomorphic proper map of degree $d$. The \emph{filled-Julia
set} of $g$ is the set of points in $U$ that never leave $V$ under iteration of $g$, i.e.,
$$K_g:=\bigcap_{n\geq0}g^{-n}(V),$$
and its \emph{Julia set} is defined as $J_g:=\partial K_g$. We can similarly define the \emph{postcritical points} of $g$ and the concept of \emph{postcritical-finite}. If $g$ is postcritically-finite, then its \emph{Hubbard tree} and \emph{extended Hubbard} tree is defined similar as the polynomial case.

Two polynomial-like maps $f$ and $g$ are \emph{hybrid equivalent} if there is a quasiconformal conjugacy $\psi$ between $f$ and $g$ that is defined on a neighborhood
of their respective filled-in Julia sets so that $\partial \ov{\psi}=0$ on $K_f$.
The crucial relation between polynomial-like maps and polynomials is
explained in the following theorem, due to Douady and Hubbard \cite{DH2}.

\begin{theorem}[Straightening Theorem]\label{thm:straigtening}
Let $f:U\to V$ be a polynomial-like map of degree $d\geq2$. Then $f$ is hybrid equivalent to a polynomial $P$ of the same degree. Moreover, if $K_f$ is connected, then $P$ is unique up to affine conjugation.
\end{theorem}

\begin{definition}[renormalization]\label{def:renormalization}
Let $f$ be a rational map. A triple $\rho=(f^p,U,V)$ is called a \emph{renormalizaton triple} of $f$ if $(f^p,U,V)$ is a polynomial-like map with connected filled-in Julia set, and the number $p$ is called the \emph{renormalization period} of $\rho$.
\end{definition}
The filled Julia set of $\rho$ is denoted by $K_\rho$, the Julia set by $J_\rho$,  the critical/postcritical
sets by ${\rm Crit}_\rho/{\rm Post}_\rho$, and the Hubbard tree by $H_\rho$, respectively.

\subsection{Basic results of Newton maps}
Let $P$ be a complex polynomial, factored as
\[P(z)=(z-a_1)^{n_1}\cdots(z-a_d)^{n_d},\]
 and $a_1,\ldots,a_d$ ($d\geq3$) are distinct roots of $P$, with multiplicities
$n_1,\ldots,n_d\geq1$, respectively.

Its Newton map $$f=f_P:=z-\frac{P(z)}{P'(z)}$$ has degree $d$ and fixes each root $a_i$ with multiplier $f'(a_i)=(n_i-1)/n_i$. Therefore, each root $a_i$ of $P$ corresponds to an attracting fixed point of $f$
with multiplier $1-1/n_i$. One may verify that $\infty$ is a repelling fixed point of $f$ with multiplier $d/(d-1)$. This discussion shows that a degree $d$ Newton map has $d+1$
distinct fixed points with specific multipliers. On the other hand, a well-known theorem of Head states that the fixed points together with the specific
multipliers can determine a unique Newton map:

\begin{proposition}[Head]\label{pro:head}
A rational map $f$ of degree $d\geq3$ is a Newton
map  if and only if f has $d+1$ distinct fixed points $a_1,\ldots,a_d,\infty$,  such that for each fixed point $a_i$, the multiplier takes the form $1-1/n_i$ with $n_i\in\N, 1\leq i\leq d$.
\end{proposition}

According to Shishikura \cite{Sh}, the Julia set of a Newton map is always
connected, or equivalently, all Fatou components are simply connected.

\section{Capture surgery for sub-hyperbolic rational maps}\label{sec:capture}
In the section, we develop a method, called \emph{capture surgery}, to perturb sub-hyperbolic rational maps such that some of the Julia critical points are captured by attracting cycles after the perturbation. This surgery will be used to construct sequences of polynomials or Newton maps on which the entropy function achieving the inferior (in the proof of Proposition \ref{pro:key} and Theorem \ref{thm:main2}). A similar idea is used in \cite{CT3,GZ} for a very special case.

Let $f$ be a sub-hyperbolic rational map of degree $d$ such that $\FFF_f\not=\emptyset$. 
Fix a vector $\c=(c_1,\ldots,c_r)$ of distinct Julia critical points of $f$. The perturbation of $f$ by capture surgery near $\c$ includes three steps.

\noindent \emph{I. The topological surgery.}

For each $c\in\c$, we choose a small open disk $D_c$ containing $c$ such that these $D_c's$ are pairwise disjoint and have iteration relation:
 if $f^m(c')=c$ for $c,c'\in \c$, then $D(c')$ is the component of $f^{-m}(D(c))$ containing $c'$.
We call the vector $\DD:=(D_c:c\in\c)$ a \emph{perturbation domain}.
In each $D_c$, we choose an open set (not necessarily connected) $Z_c\subseteq \FFF_f$ such that
\begin{enumerate}
\item if  $f^m(c')=c$ for $c,c'\in \c$, then $Z(c')=f^{-m}(Z(c))\cap D(c')$;
\item the forward orbit of any point in $Z_c$ never return to $D_c$.
\end{enumerate}
The vector $\ZZ:=(Z_c:c\in\c)$ is called an \emph{invariant domain}. Finally, we define a \emph{perturbation mapping} $\XX=(\chi_c:c\in\c)$ such that for each $c\in\c$,
\begin{enumerate}
\item $\chi_c:D_c\to f(D_c)$ is a \emph{quasi-regular map}, i.e, the composition of a rational map and a quasi-conformal map, of degree ${\rm deg}(f|_{c})$;
\item the map $\chi_c$ coincides with $f$ on the invariant domain $Z_c\subseteq D_c$, the boundary $\partial D_c$ and the point $c$;
\item the critical values of $\chi_c$ belong to $f(Z_c)=\chi_c(Z_c)$.
\end{enumerate}
The vector $\vec{\s}:=(\c,\DD,\ZZ,\XX)$ is called the \emph{perturbation data} for the capture surgery.

We now define the \emph{topological perturbation} $F=F_{\vec{\s}}$ of $f$ by capture surgery with a given perturbation date $\vec{\s}$, as
\begin{equation}\label{eq:perturbation}
F(z)=\left\{
    \begin{array}{ll}
      \chi_c(z), & \hbox{if $z\in D(c)$ for some $c\in\c$;} \\[5pt]
      f(z), & \hbox{otherwise.}
    \end{array}
  \right.
\end{equation}
By the definition, the perturbation map $F$ differs with $f$ only at $\cup_{c\in\c}D(c)$, and each  Julia critical point $c$ of $f$ in $\c$ splits into several critical points of $F$ captured by attracting cycles of $F$.

\noindent\emph{II. The rational realization.}

To realize the topological perturbation $F$ as a rational map, we need Thurston-Cui-Tan's result about the topological characterization of
sub-hyperbolic rational maps \cite{DH3,CT1}.

Let $G$ be a branched covering of degree $d\geq2$. Its \emph{postcritical set} is defined as
\[{\rm Post}_G=\ov{\{G^n(c):n\geq1,\text{ $c$ is a critical point of $G$}\}}.\]
We say that $G$ is a \emph{sub-hyperbolic semi-rational map} if the \emph{accumulation set} ${\rm Post}_G'$ of ${\rm Post}_G$ is finite (or empty); and in the
case ${\rm Post}_G'\not=\emptyset$, the map $G$ is holomorphic in a neighborhood of ${\rm Post}_G'$ and every periodic point
${\rm Post}_G'$ is (super-)attracting.

By a \emph{marked sub-hyperbolic, semi-rational map $(G,\QQQ)$}, we mean that $G:\ov{\C}\to\ov{\C}$ is a sub-hyperbolic, semi-rational map and the \emph{marked set} $\QQQ\subseteq \ov{\C}$ is a closed set such that ${\rm Post}_G\subseteq \QQQ$, $G(\QQQ)\subseteq \QQQ$ and $\#(\QQQ\setminus{\rm Post}_{G})<\infty$.

A Jordan curve $\g\subseteq \ov{\C}\setminus\QQQ$ is called \emph{peripheral} in $\ov{\C}\setminus \QQQ$ if one of its complementary components contains at most one point of ${\rm Post}_G$; and is otherwise called non-peripheral in $\ov{\C}\setminus\QQQ$. We say that $\G=\{\g_1,\ldots,\g_k\}$ is a \emph{multicurve} in $\ov{\C}\setminus\QQQ$, if each $\g_i$ is a
non-peripheral Jordan curve in $\ov{\C}\setminus\QQQ$, these curves are pairwise disjoint and
pairwise non-homotopic in $\C\setminus\QQQ$. Its \emph{$(G,\QQQ)$-transition matrix} $D_\G=(a_{ij})$ is
defined by:
\[a_{ij}=\sum_\alpha\frac{1}{{\rm deg}(G:\alpha\to\g_j)},\]
where the summation is taken over the components $\alpha$ of $G^{-1}(\g_j)$
homotopic to $\g_i$ in $\ov{\C}\setminus \QQQ$. The leading eigenvalue of $D_\G$ is also called the \emph{leading eigenvalue of $\G$}, denoted by $\lambda_\G$.

We say that a multicurve $\G\subseteq\ov{\C}\setminus\QQQ$ is $G$-stable if, for any curve $\g\in\G$, every component of $G^{-1}(\g)$ is either peripheral or homotopic in $\ov{\C}\setminus\QQQ$ to a curve in $\G$.
A multicurve $\G\subseteq \ov{\C}\setminus\QQQ$ is called a \emph{Thurston obstruction} of $(G,\QQQ)$ if it is $G$-stable and $\lambda_\G\geq1$.

Two marked sub-hyperbolic, semi-rational maps $(G_1,\QQQ_1)$ and $(G_2,\QQQ_2)$ are called \emph{c-equivalent}, if there is a
pair $(\phi,\psi)$ of homeomorphisms of $\ov{\C}$, and a neighborhood $U_0$ of $\QQQ_1'$ such that
\begin{itemize}
\item $\phi\circ G_1=G_2\circ \psi$;
\item $\phi$ is holomorphic in $U_0$;
\item the maps $\phi$ and $\psi$ are equal on $\QQQ_1$, and thus on $\QQQ_1\cup \ov{U_0}$  (by the isolated zero theorem);
\item the two maps $\phi$ and $\psi$ are isotopic to each other relatively to $\QQQ_1\cup \ov{U_0}$.
\end{itemize}
If $\QQQ_1={\rm Post}_{G_1}$ (hence $\QQQ_2={\rm Post}_{G_2}$), we say that $G_1$ is \emph{c-equivalent} to $G_2$.

\begin{theorem}[Thurston-Cui-Tan\cite{DH3,CT1}]\label{thm:cui-tan-Thurston}
Let $(G,\QQQ)$ be a sub-hyperbolic, semi-rational marked map, not Latt\`{e}s type. Then $(G,\QQQ)$ is c-equivalent to a rational marked map if and only if $(G,\QQQ)$ has no Thurston obstructions. In this case the
rational map is unique up to M\"{o}bius conjugation.
\end{theorem}

According to the requirements on the perturbation data,  any topological perturbation $F$ of a sub-hyperbolic rational map $f$ by capture surgery is sub-hyperbolic, semi-rational. We believe that $F$ is c-equivalent to a rational map in general case. But in this paper, we just prove this point in the polynomial and Newton-map cases, since these are the only cases we encounter and the argument is relatively simple. We will deal with the polynomial case in this section, and leave the Newton-map case to Section \ref{sec:capture-N}.

\begin{lemma}\label{lem:no-Thurston-obstruction}
Let $f$ be a sub-hyperbolic polynomial, and the invariant domain $\ZZ$ belong to the basin of infinity.  Then the topological perturbation $F$ of $f$ defined in Step I is c-equivalent to a polynomial.
\end{lemma}

\noindent\emph{Proof.}
Let $\QQQ:={\rm Post}_F\cup{\rm Post}_f$. According to Property (2) of the perturbation mapping $\XX$, the set $\QQQ$ is a marked set of $F$. By Theorem \ref{thm:cui-tan-Thurston}, we just need to check that the marked map $(F,\QQQ)$ has no Thurston obstructions.
In this case the map $F$ is a topological polynomial, then Thurston obstructions are equivalent to \emph{levy cycles}, i.e., a collection of non-peripheral Jordan curves $\G=\{\g_1,\ldots,\g_k\}\subseteq \ov{\C}\setminus\QQQ$  such that, for each $\g_i$, there exists a unique component $\wt{\g}_i$ of $F^{-1}(\g_{i+1})$ so that $\wt{\g}_i$ is homotopic to $\g_i$ in $\ov{\C}\setminus\QQQ$ and $F:\wt{\g}_i\to \g_{i+1}$ is a homeomorphism.

On the contrary, assume that $(F,\QQQ)$ has a Levy cycle $\G$. For the simplicity of the statment, we assume that $\G=\{\g\}$. Throughout the proof, the curve $\wt{\g}$ denotes the component of $F^{-1}(\g)$ isotopic to $\g$ in $\ov{\C}\setminus \QQQ$.

\noindent\emph{Claim 1. By suitably choosing $\g$ in its isotopic class in $\ov{\C}\setminus\QQQ$, we have $\wt{\g}$ is disjoint with $\cup_{c\in\c}D_c$.}

\noindent\emph{Proof of Claim 1}. For the simplicity of the statement, we also assume that all critical points in $\c$ lie in one grand orbit and the periodic points in this orbit are fixed by $F$, which we denote by $z_0$. We specify the construction of $\DD$ and $\ZZ$ as follows. Let $D_{z_0}\ni z_0$ be a open disk disjoint with all postcritical points of $f$ except $z_0$. For each $c\in\c$ with $f^m(c)=z_0$, we define $D_c$ the component of $f^{-m}(D_{z_0})$ containing $c$. In each component of $D_{z_0}\setminus\KKK_f$, we choose a disk in the basin of infinity  such that its image by $f$ leaves $D_{z_0}$, and converge to $\infty$ under the iteration of $f$.  We denote by $Z_{z_0}$ the union of all such disks. For each $c\in\c$ with $f^m(c)=z_0$, we define $Z_c:=f^{-m}(Z_{z_0})\cap D_c$. Let $\DD:=(D_c:c\in\c)$ and $\ZZ:=(Z_c:c\in\c)$ be the perturbation domain and invariant domain of the perturbation map $F$.

Observe that if $\wt{\g}$ avoids all components of $F^{-1}(D_{z_0})$ disjoint with $D_{z_0}$, then the claim holds. To see this, we call $D$ a $k$-th preimage of $D_{z_0}$ ($k\geq1$) if $D$ is a component of $F^{-k}(D_{z_0})$, but $F^{k-1}(D)\cap D_{z_0}=\emptyset$. If $\wt{\g}$ avoids all $1$-th preimages of $D_{z_0}$, then
\begin{equation}\label{eq:homeo}
F:\wt{\g}\cap \wt{D}_{z_0}\to \g\cap D_{z_0} \text{ is a homeomorphism},
\end{equation}
where $\wt{D}_{z_0}$ denotes the component of $F^{-1}(D_{z_0})$ contained in $D_{z_0}$.
We modify $\g$ outside $D_{z_0}$ in its isotopic class so that $\g$ is disjoint with all $1$-th preimages of $D_{z_0}$. Then, on one hand, since the new $\wt{\g}$ differs from the original one outside $\wt{D}_{z_0}$,  it still avoids the $1$-th preimages of $D_{z_0}$ by \eqref{eq:homeo}; on the other hand, the property that $\g$ is disjoint with $1$-th preimages of $D_{z_0}$ implies that $\wt{\g}$ is disjoint from $2$-th preimages of $D_{z_0}$. Therefore the new $\wt{\g}$ avoids $1$-th, $2$-th preimages of $D_{z_0}$. Then the claim holds by inductively use this argument. So we just need to show $\wt{\g}\cap(F^{-1}(D_{z_0})\setminus D_{z_0})=\emptyset$ by suitably choosing $\g$ in its isotopic class.

We say that $\g$ \emph{essential intersects} a set $S$ if all curves in the isotopic class of $\g$ intersects $S$. If $\g$ does not essentially intersect $D_{z_0}$, then the conclusion above obviously holds. So, in the following, we always assume $\g$ essentially intersects $D_{z_0}$. Note that $\QQQ\cap D_{z_0}$ is contained in $\{c\}\cup Z_{z_0}$. We first show that $\g$ does not essentially intersects $Z_{z_0}$. If not, the curve $\wt{\g}$ also essentially intersects $Z_{z_0}$. It implies that $\g=F(\wt{\g})$ essentially intersects $F(Z_{z_0})$. Inductively, we have $\infty\in\g$ since $F^n(Z_{z_0})\to\infty$, a contradiction.
Thus, the $\g$ must separate the components of $Z_{z_0}$ which contain postcritical points of $F$.

Without loss of generality, we assume that the rotation number of $z_0$ is $0$, i.e., all external rays of $f$ landing at $z_0$ is fixed. Then we can construct an $F$-invariant ray for each component of $Z_{z_0}$ intersecting $\QQQ$ such that the ray contains a point of $\QQQ$ in this component. Let $B$ be a component of $Z_{z_0}$ and $a\in B\cap\QQQ$. Since $D_{z_0}$ is in the linearizable domain of $z_0$, then there exists an arc $\de\subseteq D_{z_0}\cap \Omega(f)$ such that $\de$ joins $z_0$ and $a$, $\de\subseteq f(\de)$ and the gradients of the points in $\de$ are pairwise disjoint. It follows immediately that $f:\delta\to f(\delta)$ is a homeomorphism since the gradients of the points in $f(\de)$ are also pairwise disjoint by $g_f(f(z))=dg_f(z)$. Inductively, the arc $\RRR_B:=\cup_{k\geq0}f^k(\de)$ is $f$-invariant. Note that such constructed $\RRR_B$ is disjoint with the perturbation domain $\DD$, then $\RRR_B$ is also $F$-invariant. We denote by $\RRR_{z_0}$ the union of $\RRR_B$ with $B$ going through all components of $Z_{z_0}$ intersecting $\QQQ$.

By suitably choosing $\g$ in its isotopic class, we can assume that the number of the components of $\g\cap D_{z_0}$ is minimal, and the number of $\g\cap \RRR_{z_0}$ is minimal. By the minimality of $\#(\g\cap \RRR_{z_0})$, we have $\#(\wt{\g}\cap\RRR_{z_0})\geq \#(\g\cap\RRR_{z_0})$. As $F:\wt{\g}\to\g$ is a homeomorphism and $F(\RRR_{z_0})=\RRR_{z_0}$, then $\#(\wt{\g}\cap\RRR_{z_0})\leq \#(\g\cap\RRR_{z_0})$ and $F(\wt{\g}\cap\RRR_{z_0})\subseteq \g\cap \RRR_{z_0}$. Combining these two aspects, we see that $F(\wt{\g}\cap\RRR_{z_0})=\g\cap\RRR_{z_0}$. Hence, for each component $\xi$ of $\g\cap D_{z_0}$ which intersects $\RRR_{z_0}$, there exists a component $\wt{\xi}$ of $\wt{\g}\cap D_{z_0}$ such that $\xi\subseteq F(\wt{\xi})$. Assume that $\xi'$ is a component of $\wt{\g}\cap D$ with $D$ a component of $F^{-1}(D_{z_0})$ disjoint with $D_{z_0}$, then $F(\xi')$ is a component of $\g\cap D_{z_0}$. To get a contradiction, we just need to show that $F(\xi')$ intersects $\RRR_{z_0}$ (this implies $F:\wt{\g}\to\g$ is not homeomorphic by the argument above). In fact, every component of $\g\cap D_{z_0}$ essentially intersects $\RRR_{z_0}$. This is because each component of $\g\cap D_{z_0}$ separates the components of $Z_{z_0}$ which intersect $\QQQ$,
since $\g$ does not essentially intersect $Z_{z_0}$ and  $\#{\rm Comp}(\g\cap D_{z_0})$ is minimal.
We then proved  Claim 1.

We now assume that $\wt{\g}$ is disjoint with $\cup_{c\in\c}D_c$.

\noindent \emph{Claim 2. Let $U\subseteq \C$ be the disk bounded by $\g$. If $a\in\QQQ$ belongs to $U$, then $F(a)\in U$.}

\noindent \emph{Proof of Claim 2.} Since $\wt{\g}$ is disjoint with $\cup_{c\in\c}D_c$ and $f=F$ outside $\cup_{c\in\c}D_c$, the curve $\wt{\g}$ is also a component of $f^{-1}(\g)$. Since ${\rm Post}_f\subseteq \QQQ$, then $\wt{\g}$ and $\g$ are homotopic in $\ov{\C}\setminus{\rm Post}_f$. It follows immediately that $U$ contains at most one postcritical point of $f$: otherwise $\g$ is a Levy cycle of the polynomial $f$, a contradiction. Hence, each component of $f^{-1}(U)$ is a disk and its boundary is a component of $f^{-1}(\g)$.   We denote by $\wt{U}$ the component of $f^{-1}(U)$ bounded by $\wt{\g}$. As $\g$ is isotopic to $\wt{\g}$ in $\ov{\C}\setminus\QQQ$, we have $a\in\wt{U}$. Note that $\wt{\g}$ avoids $\cup_{c\in\c}D_c$, then $F(\wt{U})=f(\wt{U})=U$, and hence $F(a)\in U$, which complete the proof of Claim 2.

By repeatedly using Claim 2, we see that $U$ can not contain the escaping postcritical points of $F$. Since the remaining points of $\QQQ$ are  postcritical points of $f$ and $\g$ is peripheral in $\ov{\C}\setminus{\rm Post}_f$, then $\g$ is also peripheral in  $\ov{\C}\setminus\QQQ$, a contradiction.
\hfill\qedsymbol

Fix three points $a_1,a_2,a_3\in {\rm Post}_f\cap\FFF_f$. Then $a_1,a_2,a_3$ are also postcritical points of $F=F_{\vec{\s}}$. If $F$ has no Thurston obstructions,
by Theorem \ref{thm:cui-tan-Thurston}, it is c-equivalent to a rational map, denoted as $f_{\vec{\s}}$, by a pair of homeomorphisms $(\phi_0,\phi_1)$. We normalize $f_{\vec{\s}}$ such that $\phi_0$  fixes $a_1,a_2,a_3$, and call it a \emph{normalized rational perturbation of $f$ by capture surgery} with perturbation data $\vec{\s}$.

\noindent\emph{III. The convergence of perturbation maps by capture surgery.}

We fix the vector of perturbation critical points $\c$, and choose a sequence of perturbation domains $\big\{\DD_n=(D_{n,c}:c\in\c),n\geq1\big\}$ such that
\begin{equation}\label{eq:perturbation-data}
{\rm diam}(\DD_n):=\max_{c\in\c}{\rm diam}(D_{n,c})\to 0\text{ as }n\to\infty,
\end{equation}
a sequence of invariant domains $\{\ZZ_n=(Z_{n,c}:c\in\c,),n\geq1\}$,
and a sequence of perturbation mappings $\big\{\XX_n=(\chi_{n,c}:c\in\c),n\geq1\big\}$.  We then get a sequence of perturbation data $\{\vec{\s}_n:=(\c,\DD_n,\ZZ_n,\XX_n),n\geq1\}$.

For each $n\geq1$, let $F_n=F_{\vec{\s}_n}$ denote the topological perturbation of $f$ by capture surgery with data $\vec{\s}_n$, and $f_n=f_{\vec{\s}_n}$ the normalized rational perturbation of $f$ by capture surgery with data $\vec{\s}_n$ if $F_n$ has no Thurston obstructions.

\begin{proposition}\label{pro:surgery-convergence}
If all $F_n$ have no Thurston obstruction, then the normalized rational maps $f_n,n\geq1,$ uniformly converge to $f$ as $n\to\infty$.
\end{proposition}

The crucial tool to prove this proposition is the following distortion theorem, which is a special case of \cite[Theorem 8.8]{CT2}

\begin{lemma}\label{thm:Cui_Tan}
 Let $X$ be a finite set in the Julia set of $f$,  and $V$ an open set compactly contained in the Fatou set of $f$. Fix three points $x_1,x_2,x_3\in V$.  Then, for any $\epsilon>0$, there exists a $\delta>0$ such that
	$$\sup_{z\in V}\{{\rm dist}(\varphi(z),z)\}\leq \epsilon,$$
	for any univalent map
$$\varphi:\ov{\C}\setminus \bigcup_{z\in X}\bigcup_{1\leq i\leq m}f^{-i}(B_\de(z))\to\ov{\C}$$
 fixing $x_1,x_2,x_3$ and any $m\geq0$, where $B_\de(z)$ denotes the ball with center $z$ and diameter $\de$ with respect to the standard sphere metric.
\end{lemma}

\begin{proof}[Proof of Proposition \ref{pro:surgery-convergence}]
For any $n\geq1$, we can obtain a sequence of rational maps $\{f_{n,k},k\geq0\}$ by applying Thurston algorithm on $F_n$ as follows.
Let $a_1,a_2,a_3$ be three chosen points in ${\rm Post}_f\cap\FFF_f$ used to normalize $f_n$.

Set $\eta_{n,0}:={\rm id}$. Then $F_n\circ\eta_{n,0}$ defines a complex structure on $\ov{\C}$ by pulling back the standard complex structure. By  Uniformization Theorem, there exists a unique homeomorphism $\eta_{n,1}:\ov{\C}\to \ov{\C}$ normalized by fixing $a_1,a_2,a_3$ such that $f_{n,0}:=\eta_{n,0}\circ F_{n}\circ\eta_{n,1}^{-1}$ is holomorphic. Note that $\eta_{n,1}$ is holomorphic except on $\cup_{c\in\c}D_{n,c}$.

Recursively, for each $k\ge0$, there exists a normalized homeomorphism $\eta_{n,k+1}$ and a rational map $f_{n,k+1}$, such that
$\eta_{n,k}\circ F_{n}=f_{n,k}\circ\eta_{n,k+1}$ (see diagram below), and
 $\eta_{n,k+1}$ is univalent on $\ov{\C}\sm \bigcup_{0\leq j\leq k} F_n^{-j}(\cup_{c\in\c}D_{n,c})=\ov{\C}\sm \bigcup_{0\leq j\leq k} f^{-j}(\cup_{c\in\c}D_{n,c})$.
\begin{equation}\label{eq:commutative2}
\begin{CD}
{\ov{\C}} @> {\cdots} >> {\ov{\C}}@>{F_n}>>{\ov{\C}}@>{\cdots}>>{\ov{\C}}@>{F_n}>>{\ov{\C}}@>{F_n}>>{\ov{\C}} \\
@ VV{\cdots}V @ VV {\eta_{n,k+1}} V@VV {\eta_{n,k}}V@VV {\eta_{n,2}}V@VV {\eta_{n,1}}V@VV {\eta_{n,0}}V \\
{\ov{\C}} @> {\cdots} >> {\ov{\C}}@>{f_{n,k}}>>{\ov{\C}}@>{\cdots}>>{\ov{\C}}@>{f_{n,1}}>>{\ov{\C}}@>{f_{n,0}}>>{\ov{\C}} \\
\end{CD}
\end{equation}
The sequence of rational maps $\{f_{n,k},k\geq0\}$ is called a \emph{Thurston sequence} of $F_n$.

Let $V$ be an open set compactly contained in $\FFF_f\setminus \ov{\cup_{c\in\c}D_{0,c}}$ and such that  $a_1,a_2,a_3\in V$. Because of the univalent property of $\eta_{n,k}$ described above, Lemma \ref{thm:Cui_Tan} gives a crucial distortion estimate
\begin{equation}\label{equ:distortion}
\sup_{k\geq 0,z\in V}{\rm dist}\,(\eta_{n,k}(z),z)\to 0{\; \rm as \;}n\to\infty.
\end{equation}
Note that $F_n=f$ on $\ov{\C}\setminus \cup_{c\in\c}D_{n,c}\supseteq V$, then \eqref{eq:commutative2} and \eqref{equ:distortion} imply
\begin{equation}\label{equ:distortion_f}
\sup_{k\geq 0,z\in V}{\rm dist}\,(f_{n,k}(z),f(z))\to0{\; \rm as \;}n\to\infty.
\end{equation}
It was known that $f_n\to f$ on $V$ implies $f_n\to f$ on $\ov{\C}$ (see e.g \cite[Lemma 2.8]{CT2}). Hence, by \eqref{equ:distortion_f}, this lemma will be proved provided we can show $f_{n,k}\to f_n$ on $V$ as $k\to\infty$ for each $n\geq1$.

To prove this point, let $(\phi_{n,0},\phi_{n,1})$ be a pair of normalized  homeomorphisms by which $F_n$ and $f_n$ are c-equivalent.
We can further assume that $\phi_{n,0}$ is quasi-conformal on $\ov{\C}$ and holomorphic on a neighborhood $U$ of all (supper-)attracting cycles of $f_n$.
By Homotopy Lifting Lemma, we get a sequence of homeomorphisms $\{\phi_{n,k}\}_{k\geq 0}$ such that the following commutative graph holds:
\begin{equation}\label{eq:commutative1}
\begin{CD}
{\ov{\C}} @> {\cdots} >> {\ov{\C}}@>{F_n}>>{\ov{\C}}@>{\cdots}>>{\ov{\C}}@>{F_n}>>{\ov{\C}}@>{F_n}>>{\ov{\C}} \\
@ VV{\cdots}V @ VV {\phi_{n,k+1}} V@VV {\phi_{n,k}}V@VV {\phi_{n,2}}V@VV {\phi_{n,1}}V@VV {\phi_{n,0}}V \\
{\ov{\C}} @> {\cdots} >> {\ov{\C}}@>{f_n}>>{\ov{\C}}@>{\cdots}>>{\ov{\C}}@>{f_n}>>{\ov{\C}}@>{f_n}>>{\ov{\C}} \\
\end{CD}
\end{equation}

For each $k\geq0$, let $\psi_{n,k}:=\eta_{n,k}\circ\phi_{n,k}^{-1}$. Combining diagrams \eqref{eq:commutative1} and \eqref{eq:commutative2}, we get the following commutative graph
\[\begin{CD}
{\ov{\C}} @> {\cdots} >> {\ov{\C}}@>{f_n}>>{\ov{\C}}@>{\cdots}>>{\ov{\C}}@>{f_n}>>{\ov{\C}}@>{f_n}>>{\ov{\C}} \\
@ VV{\cdots}V @ VV {\psi_{n,k+1}} V@VV {\psi_{n,k}}V@VV {\psi_{n,2}}V@VV {\psi_{n,1}}V@VV {\psi_{n,0}}V \\
{\ov{\C}} @> {\cdots} >> {\ov{\C}}@>{f_{n,k}}>>{\ov{\C}}@>{\cdots}>>{\ov{\C}}@>{f_{n,1}}>>{\ov{\C}}@>{f_{n,0}}>>{\ov{\C}} \\
\end{CD}\]
Note that $\psi_{n,0}$ is quasi-conformal on $\ov{\C}$ and holomorphic on $U$, then the diagram implies that each $\psi_{n,k}$ is quasi-conformal on $\ov{\C}$ with a uniformly bounded dilation $K_n$, and holomorphic on $\bigcup_{0\leq i\leq k}f_n^{-i}(U)$.  Thus $\{\psi_{n,k},k\geq0\}$ is a normal family since each $\psi_{n,k}$ fixes $a_1,a_2,a_3$.
Let $\psi_n:\ov{\C}\to\ov{\C}$ be a limit of a subsequence of $\{\psi_{n,k},k\geq0\}$. By the argument above $\psi_n$ is $K_n$-quasiconformal on $\ov{\C}$ and  holomorphic on $\bigcup_{i\geq 0}f_n^{-i}(U)=\FFF_{f_n}$. It follows that $\psi_n$ is holomorphic on $\ov{\C}$ since  $\mathcal{J}_{f_n}$ is removable. Since $\psi_n$ fixes $a_1,a_2,a_3$, we have $\psi_n={\rm id}$. Applying this argument to any convergence subsequence of $\{\psi_{n,k},k\geq0\}$, one see that the entire sequence $\{\psi_{n,k},k\geq0\}$ uniformly converges to the identity on $\ov{\C}$.
As a consequence, the Thurston sequence $\{f_{n,k},k\geq1\}$ uniformly converges to $f_n$. We then complete the proof of Proposition \ref{pro:surgery-convergence}.
\end{proof}

\section{Perturbation of  rational maps}

Let $f$ be a rational map and $z$ be a repelling preperiodic point of $f$ such that its orbit avoids all critical points of $f$. By implicit function theorem,  there exists a neighborhood $\Lambda$ of $f$ and a holomorphic map $\zeta_z:\Lambda\to \C$ such that $\zeta_z(f)=z$ and $\zeta_z(g)$ is the unique repelling preperiodic points $g$ near $z$ with the same preperiod and period as those of $z$ for all $g\in \Lambda$.
\begin{definition}[continuation of repelling points]\label{def:continuation}
Under the assumption and notations above, the point $\zeta_z(g)$  is called the \emph{continuation of $z$ at $g$}.
\end{definition}

Suppose that $\{f_n,n\geq1\}$ is a sequence of rational maps converging to $f$, and $U$ an (supper-)attracting periodic Fatou component of $f$ with an attracting periodic point $a\in U$. The point $a$ belongs to a unique attracting domain $U_n$ of $f_n$ for large $n$, which we call the \emph{deformation of $U$} at $f_n$. Recall that ${GO}(U)$ denotes the collection of Fatou components of $f$ contained in the grand orbit of $U$, and $f$ is said to be \emph{\pf in $GO(U)$} if each $U'\in GO(U)$ contains at most one postcritical points of $f$.

\begin{lemma}\label{lem:Fatou-component}
Follow the notations above and assume that $f,f_n$ are \pf in $GO(U)$, $GO(U_n)$ respectively. Let $U'$ be any element of $GO(U)$ with the center $x$. Then there exists a unique element $U_n'\in GO(U_n)$, such that the center $x_n$ of $U_n'$ converges to $x$ as $n\to\infty$, and ${\rm deg}(f_n|_{U_n'})={\rm deg}(f|_{U'})$ for all sufficiently large $n$. We call $x_n$ the \emph{continuation of $x$ at $f_n$}. Furthermore, for any repelling preperiodic point $z\in\partial U'$,  there is a unique point $z_n\in\partial U_n'$, having the same preperiod and period as $z$, such that $z_n\to z$ as $n\to\infty$. The point $z_n$  is called the \emph{continuation of $z$ at $\partial U_n'$}.
\end{lemma}
\begin{proof}
Let $x,x_n$ be the centers of $U',U_n'$ respectively.
If $x$ is periodic, the continuation $y_n$ of $x$ at $f_n$ is an attracting periodic point contained in $U_n'$. Hence $y_n=x_n$.
Let us now deal with the preperiodic case by induction. Let us assume that $f_n(x_n)\to f(x)$ as $n\to\infty$: we need to show that
$x_n\to x$ and ${\rm deg}(f_n|_{U_n'})=\text{deg}(f|_{U'})$ as $n\to\infty$.
Set $\delta:={\rm deg}(f|_{U'})$. By Rouch\'{e}'s theorem, any given small neighborhood of $x$  contains exactly $\delta$ preimages by $f_n$ of $f_n(x_n)$ (counting with multiplicity) for every sufficiently large $n$. Note that all these preimages belong to $U_n$, and are the centers of some Fatou components of $f_n$. So these preimages must coincide with $x_n$. It follows that $x_n\to x$ as $n\to\infty$ and $\text{deg}(f_n|_{U_n})=\delta$ for all large $n$.

For the remaining result of this lemma, we first assume that $z$ is periodic. In this case, the conclusion holds by Goldberg and Milnor's proof in \cite[Appendix B]{GM}.
Now, let $z\in\partial U'$ be a preperiodic point. Set $v:=f(z)\in \partial f(U')$. Inductively, we assume that $v_n$ is the unique preperiodic point of $f_n$ in $\partial f_n(U_n')$ such that $v_n$ has the same preperiod and period as $v$, and $v_n\to v$ as $n\to\infty$.  Since $f_n$ uniformly converges to $f$, given any small disk neighborhood $W_z$ of $z$, there is a disk neighborhood $V_v$ of $v$ such that the component of $f^{-1}_n(V_v)$ that contains $z$, denoted by $D_{n,z}$, belong to $W_z$, for all sufficiently large $n$. Given any sufficiently large $n$, choose a point $a_n\in D_{n,z}\cap U_n'$ and set $b_n:=f_n(a_n)$. Then $b_n\in V_v\cap f_n(U_n')$.
By the inductive assumption, the point $v_n$ belongs to $\partial f_n(U_n')\cap V_v$. One can then choose an arc $\g_n\subset f_n(U_n')\cap V_v$ joining $b_n$ and $v_n$. Lifting $\g_n$ by $f_n$ with the starting point $a_n$, we get an arc $\wt{\g}_n\subset D_{n,z}\cap U_n'$. Its ending point, denoted by $z_n$,  belongs to $\partial U_n'$ and satisfies that $f_n(z_n)=v_n$.
By the argument above, we in fact proved that for any point $z'\in\partial U'$ with $f(z')=v$, and any small neighborhood $W_{z'}$ of $z'$, there exists a point $z_n'\in\partial U_n'$ with the property that $z_n'\in W_{z'}$ and $f_n(z_n')=v_n$ for all large $n$. Since $\text{deg}(f_n|_{U_n'})=\text{deg}(f|_{U'})$ , the points which have the same properties as $z_n'$ are unique.
This completes the proof of the lemma.
\end{proof}

Let now $\{S_n\}\subset \C$ be a sequence of sets. We denote as
$\limsup_{n\to\infty} S_n$
the set of points $z\in\C$ such that every neighborhood of $z$ intersects infinitely many $S_n$. It follows immediately from the definition that $\limsup S_n$ is closed.

\begin{lemma}\label{lem:internal-ray}
Under the assumption of Lemma \ref{lem:Fatou-component}, for each large $n$, let $I_n$ be an preperiodic internal ray of $f_n$ in $U_n'$ with fixed preperiod $k\geq0$ and period $p\geq1$. If the landing point $z_n$ of $I_n$ converge to $z$, then $\limsup_{n\to\infty}I_n=I,$ where $I$ is the internal ray of $f$ in $U'$ landing at $z$.
\end{lemma}
\begin{proof}
If $I$ is period, the conclusion holds by Goldberg and Milnor's proof in \cite[Appendix B]{GM}. By induction, it suffices to prove $\limsup I_n=I$ provided that $\limsup f_n(I_n)=f(I)$. Since $f_n\to f$, we can choose  B\"{o}ttcher coordinates $\varphi$ of $U'$ and $\varphi_n$ of $U_n$ such  that $\varphi_n^{-1}:D\to U_n$ converge
uniformly on compact sets to $\varphi^{-1}:D\to U$. It follows that $I':=\limsup I_n\cap U$ is an internal ray of $U$. On the other hand, note that $\limsup I_n\cap \partial U$ is compact, connected and contains the point $z$. The map $f$ sends $\limsup I_n\cap \partial U$ into the set $\limsup f_n(I_n)\cap\partial f(U)$, which is by induction a singleton. Then we get $\limsup I_n\cap \partial U=\{z\}$, and hence $I'=I$.
\end{proof}

\begin{lemma}[Goldberg-Milnor \cite{GM}]\label{lem:perturbation1}
Consider a monic polynomial $P$ and an external ray $\RRR_P(\theta)$ which lands at a repelling preperiodic point $z$ such that the orbit of $z$ avoids the critical points of $P$. Then $\RRR_{Q}(\theta)$ lands at the analytic continuation of $z$, for all monic polynomials $Q$ in a sufficiently small neighborhood of $P$, and $\limsup_{Q\to P}\ov{\RRR_{\wt{P}}(\theta)}=\ov{\RRR_P(\theta)}$. 
\end{lemma}

\section{Partial \pf polynomials}

\subsection{Hubbard forest and core entropy of partial \pf polynomials}
\

Let $P$ be any polynomial of degree $d$. A critical/postcritical point of $P$ is called \emph{bounded} if it lies in $\KKK_P$, and \emph{escaping} if it belongs to $\Omega(f)$. The polynomial $P$ is called \emph{partial \pf} if each bounded critical point has finite orbit.

Let $P$ be any partial \pf polynomial. Then it is hyperbolic or sub-hyperbolic.  By Lemma \ref{lem:sub-hyperbolic},  every component of $\KKK(P)$ is a full, locally-connected compact set. As a consequence, given any two points $z,w$ in a component of $\KKK_P$, there exists a unique arc $[z,w]$ within this component joining $z,w$ such that its intersection with every bounded Fatou components consisting of (segments) of internal rays. Such an arc is called a \emph{regular} arc. If $z,w$ lie in different components of $\KKK_P$, denote $[z,w]:=\emptyset$. We define a forest as
\[\HHH_P=\bigcup\{[z,w]\mid \text{$z,w$ are bounded critical/postcritical points of $P$}\},\]
and call it the \emph{Hubbard forest} of $P$.

It is easily check that $\HHH_P$ is a $P$-invariant forest with the vertex set $V(\HHH_P)$ consisting of bounded critical/postcritical points of $P$ and the branched points of $\HHH_P$,
and the map $P:\HHH_P\to\HHH_P$ is Markov.
Note that if $P$ has connected Julia set , i.e., $P$ is postcritically-finite, then $\HHH_P$ is just the Hubbard tree of $P$.

 We define the \emph{core entropy} of $P$ as the topological entropy of $P$ restricted on its Hubbard forest, i.e., $h(P):=h_{top}(P|_{\HHH_P})$. Then the defining domain of the entropy function $h$ is enlarged from \pf polynomials to partial \pf polynomials.

\subsection{Critical markings of partial \pf polynomials}\label{sec:critical-marking}
The objective here is to generalize the combinatorial data ``critical markings'' from the \pf case to the partial \pf case.

A partial \pf polynomial is called \emph{visible} if every escaping critical point is terminated by several external radii (see Figture \ref{fig:visible}, compare with \cite[Definition 3.1]{Kiwi2}).
 \begin{figure}[http]
 \includegraphics[scale=0.6]{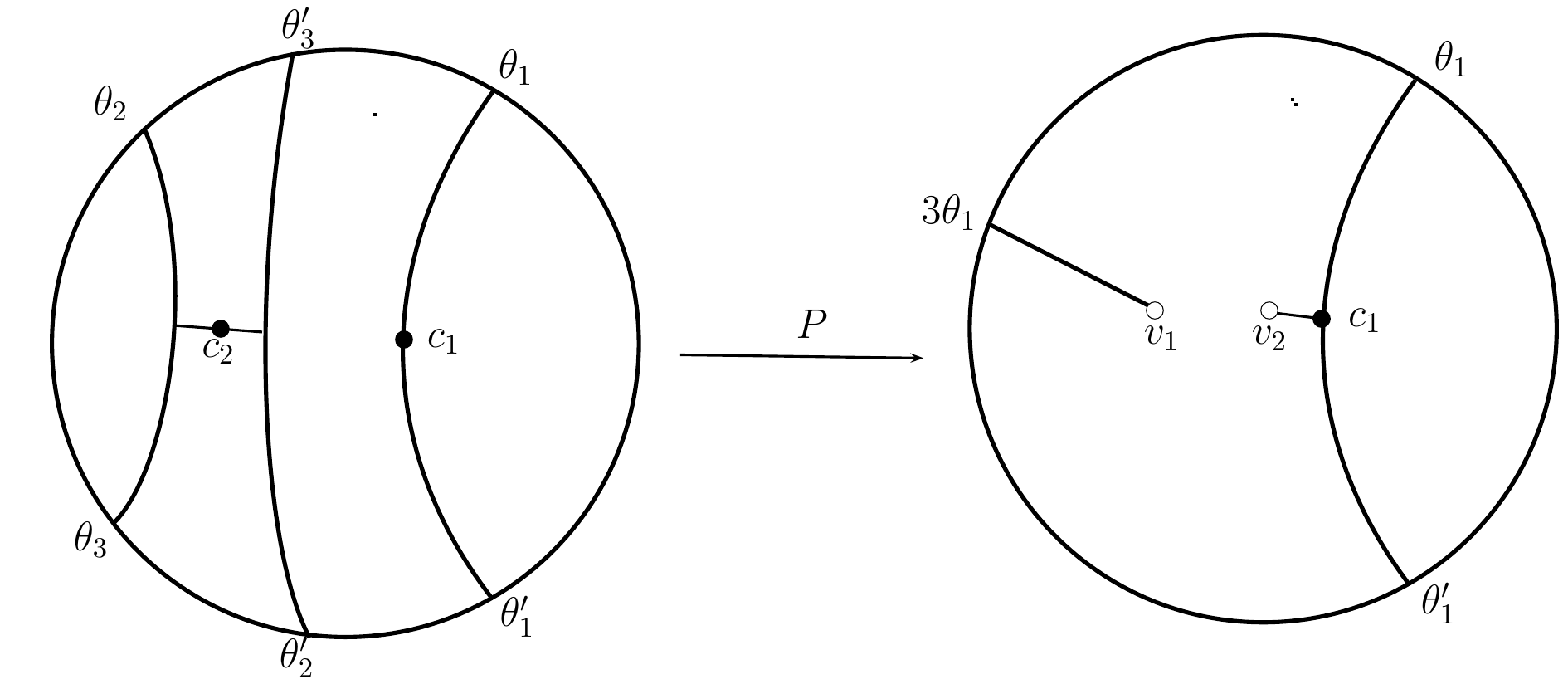}
 \caption{A non-visible example: on the left, the possible configurations of gradient flow lines connecting the
critical points of a cubic polynomial, and
on the right, their image under $\tau_3$; here $f(c_i)=v_i$, $\theta_1=3\theta_2=3\theta_2'$, $\theta_1'=3\theta_3=3\theta_3'$.}
 \label{fig:visible}
 \end{figure}
 The number of external radii terminating at $c$ is an integer multiple of the the local degree ${\rm deg}(P|_c)$, and the multiple is $1$ if and only if the orbit of $c$ avoids critical points of $P$.


Let $P$ be a visible partial \pf polynomial, such that $U_1,\ldots,U_s$ are pairwise distinct bounded critical Fatou components, and $c_1,\ldots,c_m$ are  pairwise distinct Julia or escaping critical points. A \emph{critical marking} of $P$ is a collection
$$\Theta_P:=\{\Theta(U_1),\ldots,\Theta(U_s),\Theta(c_1),\ldots,\Theta(c_r)\}$$
of finite subsets of $\R/\Z$ constructed below:

{\noindent\emph{The construction of $\Theta(U)$}.}

 We first consider that $U$ is  periodic. Let
\[U\mapsto P(U)\mapsto\cdots\mapsto P^n(U)=c \]
be the critical cycle containing $U$. We will construct the associated set $\Theta(U')$ for every critical Fatou component $U'$ in this cycle simultaneously.
Let $z\in\partial U$ be a \emph{root} of $U$, i.e., a periodic point  with period less than or equal to $n$.
Note that this choice
naturally determines a root $P^k(z)$ for each Fatou component $P^k(U)$ for $0\leq k\leq n-1$, which is called the \emph{preferred root} of $P^k(U)$.
Let  $z'$ be the preferred root of $U'$. We denote $\theta'$  the left-supporting angle of the component $U'$ at $z'$, and call it the \emph{preferred angle} of $U'$. The angle $\theta'$ is periodic by Lemma \ref{pro:landing-theorem}, and has period $n$. We define $\Theta(U')$ the set of arguments of the rays in $P^{-1}\big(P(\RRR_P^\sigma(\theta))\big)$ which left-support $U'$, where $\RRR^\sigma_P(\theta)$ denote the external ray left-supporting $U'$ at $z'$ with $\sigma\in\{+,-\}$.

 In the case that $U$ is a strictly preperiodic Fatou component, let $k$ be the minimal number such that $P^k(U)=:U'$ is a critical Fatou component with $\Theta(U')$ and the preferred angle $\theta'\in \Theta(U')$  chosen.  We define $\Theta(U)$ as the set of arguments of the rays in $P^{-k}(\RRR_P^\sigma(\theta'))$ which left-support $U$, where $\RRR_P^\sigma(\theta)$ denote the external ray left-supporting $U'$. We choose an angle $\theta\in\Theta(U)$ as the preferred angle of $U$.

{\noindent \emph{The construction of $\Theta(c)$ for Julia critical points.}}

If the orbit of $c$ avoids all critical points, we choose an angle $\theta\in{\rm arg}_P(c)$ as the preferred angle of $c$ and define $\Theta(c)$ the set of arguments of rays in $P^{-1}(P(\RRR_P^\sigma(\theta)))$ which land at $c$, where $\RRR_P^{\sigma}(\theta)$ lands at $c$. Otherwise, let $k$ be the minimal number such that $c':=f^k(c)$ is a critical point with $\Theta(c')$ and the preferred angle $\theta'\in\Theta(c')$ chosen. We define $\Theta(c)$  the set of arguments of the rays in $P^{-k}(\RRR_P^\sigma(\theta'))$ which land at $c$, where $\RRR_P^\sigma(\theta)$ denote the external ray landing at $c'$. We choose an angle $\theta\in\Theta(c)$ as the preferred angle of $c$.

{\noindent \emph{The construction of $\Theta(c)$ for escaping critical points.}}

If the orbit of $c$ avoids all critical points, by the visible assumption, there are ${\rm deg}(f|_c)$ external radii terminating at $c$. We define $\Theta(c)$ the set of arguments of all these radii and choose an angle $\theta\in\Theta(c)$ as a preferred angle. Otherwise, let $k$ be the minimal number such that $c':=f^k(c)$ is a critical point with $\Theta(c')$ and the preferred angle $\theta'\in\Theta(c')$ chosen. Then there are ${\rm deg}(P|_c)$ external radii terminating at $c$ such that they are mapped by $P^k$  to the external radius $\RRR_P^*(\theta')$. We define $\Theta(c)$  the set of arguments of these radii, and choose an angle $\theta\in\Theta(c)$ as the preferred angle of $c$.

In the content below, we will use the following notations. We write a critical marking $\Theta_P$ as $\Theta_P:=\{\FFFF,\LLLL\}$ such that
\begin{itemize}
\item $\FFFF=\FFFF_P:=\{\Theta(U): \text{$U$ is a bounded Fatou component}\}$;
\item $\JJJJ=\JJJJ_P:=\{\Theta(c):\text{$c$ is a Julia critical point}\}$;
 \item $\EEEE=\EEEE_P:=\{\Theta(c):\text{$c$ is an escaping critical point}\}$;
\end{itemize}
and $\LLLL=\LLLL_P:=\JJJJ\cup\EEEE$.
Note that the arguments participating in $\FFFF$ are rational by Proposition \ref{pro:landing-theorem}, but those in $\LLLL$  are not necessary.

\subsection{The properties of critical markings}\label{sec:admissible}
Let $P$ be a visible partial \pf polynomial with a critical marking $\Theta=\{\FFFF,\LLLL\}$ such that
\begin{equation}\label{eq:critical-marking}
\begin{split}
\FFFF&=\{\FFFF_1,\ldots,\FFFF_s\}\\
\LLLL&=\{\LLLL_1,\ldots,\LLLL_m\}
\end{split}
\end{equation}
According to the construction, it is not difficult to check that $\Theta$ is a critical portrait of degree $d$. To make $\Theta$ satisfying some expected properties, we add some restrictions on $P$.

\begin{definition}\label{def:admissible}
A partial \pf polynomial $P$ is called \emph{admissible} if it is visible and  the following properties hold:
\begin{enumerate}
\item the escaping critical points have no iteration relation, and the external radii terminating at these points have strictly preperiodic angles;
\item  let $c,c'$ be two escaping critical points such that $\theta\in\Theta(c),\theta'\in\Theta(c')$, then the landing points of $\RRR^\sigma_P(\theta)$ and $\RRR^{\sigma'}_P(\theta')$ have disjoint orbits for any $\sigma,\sigma'\in \{+,-\}$.
\item let $c$ be an escaping critical point with $\theta\in\Theta(c)$, and $c'$ a Julia critical point, then the landing points of $\RRR^\sigma_P(\theta)$ and the point $c'$ have disjoint orbits for any $\sigma\in\{+,-\}$.
\end{enumerate}
\end{definition}

\begin{lemma}\label{lem:1-5}
Let $P$ be an admissible partial \pf polynomial with a critical marking $\Theta$ having the form as \eqref{eq:critical-marking}. Then $\Theta$ satisfies the following properties:
\begin{itemize}
\item [(C1)] All arguments which participate $\Theta$ are rational.
\item [(C2)] $\FFFF$ and $\LLLL^-$ are \emph{unlinked}, i.e., for  $\epsilon>0$ small enough, the collection of sets
\[\FFFF_1,\ldots,\FFFF_s,\LLLL_1-\epsilon,\ldots,\LLLL_m-\epsilon\]
have pairwise disjoint convex hulls in $\ov{\D}$.
\item [(C3)] $\FFFF$ (resp.$\LLLL$) is \emph{hierarchic}, i.e.,  for any two elements $\theta,\theta'\in\R/\Z$ that participate
in $\FFFF$ (resp. $\LLLL$) such that $\tau^i(\theta)$ and $\tau^j(\theta')$
lie in $\FFFF_k$ (resp. $\LLLL_k$), for some $i,j>0$; we have that $\tau^i(\theta)=\tau^j(\theta')$.
\item [(C4)] Given an argument that participates in $\FFFF$, there exists a periodic argument $\tau^i(\theta)$
which also participates in $\FFFF$.
\item [(C5)] None of the arguments that participate in $\LLLL$ are periodic.
\end{itemize}
\end{lemma}
\begin{proof}
 According to Definition \ref{def:admissible}.(1), the arguments participating in $\EEEE$ are rational. It then follows from Proposition \ref{pro:landing-theorem} that the arguments participating $\FFFF$ and $\JJJJ$ are also rational,  hence property (C1) holds.
Note that there is one preferred angle in every $\FFFF_\ell,1\leq \ell\leq s,$ and $\LLLL_t,1\leq t\leq m$, and any argument participating in $\Theta$ is iterated to a preferred angle. Moreover, the preferred angle of $\FFFF_\ell$ or $\LLLL_t$ is periodic if and only if the critical point corresponding to $\FFFF_\ell$ or $\LLLL_t$ is periodic. It follows that properties (C3),(C4),(C5) hold.
By the construction of $\Theta$, we have that the elements within each of $\FFFF,\JJJJ$ and $\EEEE$ have pairwise disjoint convex hulls in $\ov{\D}$. Definition \ref{def:admissible}.(3)  further assure that the elements in $\LLLL=\JJJJ\cup\EEEE$ have pairwise disjoint convex hulls in $\ov{\D}$. Then (C2) follows from the fact that all arguments participating in $\FFFF$ are left-supporting.
\end{proof}

In fact, this critical marking $\Theta$ also satisfies two other properties. To present these two properties, we need to introduce some notations.
Let $\Theta=\{\ell_1,\ldots,\ell_m\}$ be a critical portrait of degree $d$. We say that $\theta$ and $\eta$ are in the same \emph{${\Theta}$-unlinked class} if and only if $\theta$
and $\eta$ lie in the same connected component of $\T\setminus \ell_i$ for all $i=1,\ldots,m$.

\begin{lemma}[\cite{Kiwi1},Lemma 6.6]\label{lem:d-parts}
There are exactly $d$ $\Theta$-unlinked classes, each of which can be  written as the union of finitely many open intervals:
\[L=(\theta_0,\theta_1)\cup\ldots\cup(\theta_{2p-1},\theta_{2p})\]
with subscripts mod $2p$ and respecting cyclic order. The total length of these intervals is $1/d$. Additionally, $\tau(\theta_{2i})=\tau(\theta_{2i-1})$ for $i=1,\ldots,p$. Furthermore, $\tau:L\to\tau(L)$ is a
cyclic order preserving bijection and $\tau(L)$ consists of $\T$ with the finite set of points
$\{\tau(\theta_0),\ldots,\tau(\theta_{2p-2})\}$ removed (refer to Figure \ref{critical-portrait-5}).
\end{lemma}

Throughout we denote by $L_1,\ldots,L_d$ the ${\Theta}$-unlinked classes.

\begin{definition}[itinerary of angle]
For any $t\in \T$, we call the symbol sequence $\ell_{{\Theta}}^+(t)=i_0i_1\ldots i_n\ldots$ \emph{right itinerary} of $t$ (associated to $\Theta$) if
for every $n\geq0$, there exists $\epsilon>0$ such that $(\tau^n(t),\tau^n(t)+\epsilon)\subseteq L_{i_n}$; and call the sequence  $\ell_{{\Theta}}^-(t)=j_0j_1\ldots j_n\ldots$ \emph{left itinerary} of $t$ if
for every $n\geq0$, there exists $\epsilon>0$ such that $(\tau^n(t)-\epsilon,\tau^n(t))\subseteq L_{j_n}$.
\end{definition}

Note that $\ell_{{\Theta}}^+(\theta)=\ell_{{\Theta}}^-(\theta)$ if and only if the orbit of $\theta$ under $\tau$ avoids every $\ell_i$.

\begin{lemma}\label{lem:itinarary}
Let $P$ be a visible partial \pf polynomial with a critical marking $\Theta$. Let $\theta_1,\theta_2\in\T$ and $\sigma_1,\sigma_2\in\{+,-\}$. If $\ell_{{\Theta}}^{\sigma_1}(\theta_1)=\ell_{{\Theta}}^{\sigma_2}(\theta_2)$, then $\RRR_P^{\sigma_1}(\theta_1),\RRR_P^{\sigma_2}(\theta_2)$ land at a common point. 
\end{lemma}
\begin{proof}
This lemma was proved in \cite{Poi} in the \pf case by using the uniform expansionary of $P$ with respect to the orbiford metric near the Julia set and the connectedness of the Julia set. In the partial \pf case, the polynomial $P$ is still uniformly expanding with respect to the orbiford metric in some neighborhood of $\JJJ_P$ (Lemma \ref{lem:sub-hyperbolic}), and the extended Julia set $\JJJ_P^*$ is connected, locally-connected and satisfies $P^{-1}(\JJJ_P^*)\subseteq \JJJ_P^*$ (Proposition \ref{pro:subhyperbolic}). Then by substituting the Julia set with the extended Julia set, the argument in \cite{Poi} also works in our case.
\end{proof}

\begin{lemma}\label{lem:6-7}
Under the assumption of Lemma \ref{lem:1-5}, the critical marking ${\Theta}$ also satisfies the following two properties:
\begin{itemize}
\item [(C6)] Consider a periodic argument $t$ which participates in  $\FFFF$ and an argument $t'\in\Q/\Z$. If $\ell_{{\Theta}}^{+}(t)=\ell_{{\Theta}}^{+}(t')$, then $t=t'$.
\item [(C7)] Consider preperiodic angles $t\in\LLLL_j$ and $t'\in\LLLL_k$.  If $\ell_{{\Theta}}^{-}(\tau^i(t))=\ell_{{\Theta}}^{-}(t')$ for some $i$, then $\tau^i(t)=t'$.
\end{itemize}
\end{lemma}

\begin{proof}
For (C6), by Lemma \ref{lem:itinarary}, we have that $\RRR_P^+(t)$ and $\RRR_P^+(t')$ lands at the same point. Since $t$ is chosen a left-supporting angle, then $\ell_{{\Theta}}^{+}(t)=\ell_{{\Theta}}^{+}(t')$ also implies $t=t'$.

For (C7), still by Lemma \ref{lem:itinarary}, the rays $\RRR_P^-(\tau^i(t))$ and $\RRR_P^-(t')$ land at a common point $z$. The assumptions (2),(3) for admissible partial \pf polynomials imply that $\LLLL_j,\LLLL_k$ belong to $\JJJJ$. By the construction of $\JJJJ$, the angle $\tau^i(t)$ is the preferred angle in $\LLLL_k$. If $\tau^i(t)\not=t'$, then their itineraries have distinct initial digits, a contradiction.
\end{proof}

\subsection{Convergence of critical markings}
The main result in this subsection is the following convergence theorem of critical markings:

\begin{proposition}\label{pro:convergence4}
Let $P$ be a monic \pf polynomial of degree $d$, and $P_n,n\geq1$ be monic, visible partial \pf polynomials such that $P_n\to P$ as $n\to\infty$.
Suppose that $\Theta_n$ is a critical marking of $P_n$ for each $n$ and ${\Theta}_n\to {\Theta}$ as $n\to\infty$ (in Hausdorff metric on $\T$). Then ${\Theta}$ is a weak critical marking of $P$ (Definition \ref{def:weak-critical-marking}).
\end{proposition}

This result was proved in \cite[Proposition 9.16]{GT} when all $P_n$ are postcritically-finite. The argument in the partial \pf case is similar, so we just give a sketch of the proof. The argument relies on Lemmas \ref{lem:support} and \ref{lem:convergence3} below. These two lemmas corresponds to Lemmas 9.13 and 9.15 in \cite{GT} (also for the \pf case) respectively, with  completely the same argument. Hence we omit the proof.

\begin{lemma}\label{lem:support}Let $P$ be a monic, \pf polynomial of degree $d$, and $P_n,n\geq1,$ be monic, partial \pf polynomials converging to $P$ as $n\to\infty$. Assume that $\RRR_P(\theta)$ left/right-supports a Fatou component $U$ at a  periodic point $z$. Then for large $n$, the external ray $\RRR_{P_n}(\theta)$ left/right-supports the deformation $U_n$ of $U$ at the continuation $z_n$ of $z$.
\end{lemma}

\begin{lemma}\label{lem:convergence3}
Let $P$ be a monic, \pf polynomial of degree $d$, and $P_n,n\geq1,$ be monic, partial \pf polynomials converging to $P$ as $n\to\infty$. Assume that the arguments $\theta_n$ converge to $\theta$, then $\limsup_{n\to\infty}\RRR_{P_n}^{\pm}(\theta_n)=\ov{\RRR_P(\theta)}$.
\end{lemma}


\begin{proof}[Proof of Proposition \ref{pro:convergence4}](Sketch)
Let $U$ be a critical Fatou component of $P$ and denote by $U_n$ the deformation of $U$ at $P_n$. We write $\Theta(U_n)$ the critical marking of $P_n$ associated to $U_n$.

In the periodic case, each $\Theta(U_n)$ contains a unique periodic angle $\theta_n$ with period equal to that of $U_n$ and hence of $U$. By taking a subsequence if necessary, we can assume $\theta_n=\theta$ for large $n$.
Note that any $\Theta(U_n)$ is a subset of $\tau^{-1}(\tau(\theta))$ and $\#\Theta(U_n)={\rm deg}(P_n|_{U_n})={\rm deg}(P|_U)$ (by Lemma \ref{lem:Fatou-component}), then we can further assume by taking a subsequence that $\Theta(U_n)$ is constant for large $n$, which we write as $\Theta(U)$.
 According to Lemmas \ref{lem:internal-ray} and \ref{lem:convergence3}, the rays with arguments in $\Theta(U)$ land at the boundary of $U$. Furthermore, it follows from Lemma \ref{lem:support} that the periodic angle in $\Theta(U)$ left-supports $U$.
The same situation happens in the strictly preperiodic case for $U$ by a similar argument and the induction.

Let $U_1,\ldots,U_s$ be all the critical Fatou components of $P$. The discussion above shows that the collection of sets $\{\Theta(U_1),\ldots,\Theta(U_s)\}$ is \begin{itemize}
\item [(a)] a weak Fatou critical marking of $P$ (Definition \ref{def:weak-critical-marking}.(1));
\item [(b)] a part of the Fatou critical marking of $P_n$ (contained in $\Theta_n$) for large $n$.
\end{itemize}

Now, we write each $\Theta_n$ as $\Theta_n:=\{\FFFF,\XXXX_n\}$ with $\FFFF:=\{\Theta(U_1),\ldots,\Theta(U_s)\}$ and $\XXXX_n:=\{\Theta_{n,1},\ldots,\Theta_{n,m}\}$ for all large $n$, such that $\Theta_{n,i}\to \Theta_i$ as $n\to\infty$ and $1\leq i\leq m$. It follows immediately that $\#\Theta_{n,i}=\#\Theta_i$ for any $1\leq i\leq m$.
Note that each $\Theta_{n,i}$ corresponds a critical point $c_{n,i}$ of $P_n$, and we can assume by taking subsequences that $c_{n,i}$ converge to a critical point $c_i$ of $P$, which must belong to $\JJJ_P$.

We have a fact that: if each $c_{n,i}$ is in the Fatou component $U(c_{n,i})$ of $P_n$, then the sequence of closed disks $\{\ov{U(c_{n,i})},n\geq1\}$ converge to $c_i$ (see the corresponding discussion in \cite[Proposition 9.16]{GT}).
Combining this fact and Lemma \ref{lem:convergence3}, all the external rays of $P$ with arguments in $\Theta_i$ land at $c_i$. We then write $\Theta_i$ as $\Theta_i(c_i)$.
In order to prove that $$\Theta=\{\Theta(U_1),\ldots,\Theta(U_s),\Theta_1(c_1),\ldots,\Theta_m(c_m)\}$$ is a weak critical marking of $P$, we only need to check that $\Theta_1(c_1),\ldots,\Theta_m(c_m)$ satisfy Definition \ref{def:weak-critical-marking}.(3), and this points is not difficult to check.
\end{proof}

\section{The continuity of entropy function: \pf parameters}\label{sec:continuity1}

The objective here is to prove Proposition \ref{pro:key}. The proof is divided into two parts: we first show that the entropy function $h:\PPP_d^{\rm ppf}\to \R$ is upper semi-continuous at any \pf parameter $P$, and then determine the limit inferior $\mathop{\liminf}_{Q\to P}h(Q),$ as well as  how to get it.

\subsection{Upper semi-continuity of the entropy function}

\begin{proposition}\label{pro:limsup}
Let $P\in\PPP_d$ be a \pf polynomial and $\{P_n,n\geq1\}$ be a sequence of monic partial \pf polynomials converging to $P$. Then
\[\limsup_{n\to\infty}h(P_n)\leq h(P).\]
\end{proposition}

The key to prove this proposition is the following lemma.

\begin{lemma}\label{pro:less}
Let $P$ be a partial \pf polynomial and $\Theta=\Theta_P$ a critical marking of $P$ such that properties (C1)-(C7) hold for $\Theta$.  Then we have  $h(P)\leq h(\Theta)$, where $h(\Theta)$ is the output of Thurston's entropy algorithm on $\Theta$ (see Section \ref{sec:weak-critical-marking}).
\end{lemma}

\begin{proof}[Proof of Proposition \ref{pro:limsup} under Lemma \ref{pro:less}]
 Firstly, using a standard quasi-conformal surgery, we can perturb each $P_n$ to a nearby polynomial $Q_n$ by twisting a sufficiently small angle for every escaping critical point of $P_n$, so that the external angles of the escaping critical points of $Q_n$ satisfy the  the  properties below:
 \begin{enumerate}
\item the external angles of each escaping critical point are strictly preperiodic;
\item the external angles of distinct escaping critical points have different periods;
\item the period of any external angle associate to escaping critical points is larger than the period of any external ray landing at a postcritical point.
\end{enumerate}
It follows immediately that each $Q_n$ is an admissible polynomial (see Definition \ref{def:admissible}). Note that $Q_n$ conjugates to $P_n$ on the filled-in Julia set, and can be chosen arbitrarily close to $P_n$. Then we have $h(P_n)=h(Q_n)$, and $Q_n,n\geq1,$ converge to $P$. So, without loss of generality, we can assume the initial $P_n's$ are admissible.

For each $n$, let $\Theta_n$ be a critical marking of $P_n$.  By Lemmas \ref{lem:1-5} and \ref{lem:6-7}, each $\Theta_n$ satisfies properties (C1)-(C7) in Section \ref{sec:admissible}. If $\Theta_n$ converge to a critical portrait $\Theta$, then Lemmas \ref{pro:less} and  Propositions \ref{pro:GT}, \ref{pro:algorithm}, \ref{pro:convergence4} together imply that
\[\limsup_{n\to\infty}h(P_n)\overset{{\rm Lem}.\ref{pro:less}}{\leq}\limsup_{n\to\infty}h(\Theta_n)\overset{{\rm Pro}.\ref{pro:GT}}{=}h(\Theta)\overset{{\rm Pro}.\ref{pro:convergence4},\ref{pro:algorithm}}{=}h(P).\]
According to Proposition \ref{pro:convergence4}, the sequence $\{\Theta_n,n\geq1\}$ can be divided into finitely many convergence subsequence. Then the proof is completed.
\end{proof}

In the remaining part of this subsection, we focus on the proof of  Lemma \ref{pro:less}. The outline is as follows. Since $\Theta$ is assumed to satisfy properties (C1)-(C7), by Poirier's result (Theorem \ref{thm:poi} below), there exists a monic, centered, \pf polynomial $\wt{P}$ such that $\Theta$ is a critical marking of $\wt{P}$. Based on the point that $P$ and $\wt{P}$ have a common critical marking $\Theta$, we can show $h(P)\leq h(\wt{P})$. It then follows from Proposition \ref{pro:algorithm} that $h(P)\leq h(\wt{P})=h(\Theta)$.

\begin{theorem}[Poirier]\label{thm:poi}
Let $\Theta=\{\FFFF,\LLLL\}$ be a degree $d$ critical portrait satisfying properties (C1)-(C7). Then there is a unique monic centered \pf polynomial $Q$ such that ${\Theta}$ is a critical marking of $Q$ in the sense that $\FFFF_Q=\FFFF$ and $\JJJJ_Q=\LLLL$.
\end{theorem}

Let $P$ be a partial \pf map with a critical marking $\Theta=\{\FFFF,\LLLL\}$ such that
\begin{eqnarray*}
\FFFF=\FFFF_P&=&\{\FFFF_1,\ldots,\FFFF_s\}\\
\LLLL=\LLLL_P&=&\JJJJ_P\cup\EEEE_P=\{\LLLL_1,\ldots,\LLLL_m\},
\end{eqnarray*}
where $\FFFF_P,\JJJJ_P,\EEEE_P$ denote the makings of bounded Fatou critical points, Julia critical points and escaping critical points, respectively (see Section \ref{sec:critical-marking}). We further assume that $\Theta$ satisfies properties (C1)-(C7) in Section \ref{sec:admissible}. According to Theorem \ref{thm:poi}, let $\wt{P}$ denote the unique \pf polynomial in $\PPP_d$ such that $\Theta$ is a critical marking of $\wt{P}$.

Factually, the dynamics of $P$ and $\wt{P}$ are quite related. We indicate this point by the following three lemmas.

\begin{lemma}\label{lem:map-Fatou}
There exists a bijective $\phi_1$ from the centers of Fatou components of $P$ onto those of $\wt{P}$ such that $\wt{P}\circ\phi_1=\phi_1\circ P$.
\end{lemma}

\begin{proof}
Since $\FFFF_P=\FFFF_{\wt{P}}=\FFFF$, there is a natural bijection $\phi_1$ which sends the critical point $c_i$ of $P$ associated to $\FFFF_i$ to the critical point $\wt{c_i}$ of $\wt{P}$ associated to $\FFFF_i$. We will extend $\phi_1$ to all centers of bounded Fatou components of $P$.

Recall that the critical marking $\Theta$ induces a partition on $\T$ such that $\T$ is divided into $d$ $\Theta$-unlinked equivalence classes $L_1,\ldots,L_d$ (see Section \ref{sec:critical-marking}). In fact, we have a corresponding partition in the dynamical plane of $P$. A ray is called an \emph{extended ray} associated to a Fatou component $U$ of $P$ with argument $\theta$, if it is the union of an external ray of angle $\theta$ and an internal ray in $U$ which land at a common point on $\partial U$.

For each $\FFFF_i\in\FFFF_P$, we denote $\RRR_P(\FFFF_i)$ the union of extended rays associated to $U_i$ with arguments in $\FFFF_i$, where $U_i$ denote the critical Fatou component of $P$ corresponding to $\FFFF_i$; for each $\JJJJ_j\in\JJJJ_P$, we write $\RRR_P(\JJJJ_j)$ as the union of external rays with arguments in $\JJJJ_j$ which land at $c_j$ together with $c_j$, where $c_j$ denotes the critical point of $P$ corresponding to $\JJJJ_j$; and for each $\EEEE_k\in\EEEE_P$, we let $\RRR_P(\EEEE_k)$ be the union of external radii with arguments in $\EEEE_k$ together with $c_k$, where $c_k$ is the escaping critical point of $P$ corresponding to $\EEEE_k$. Thus, the union of $\RRR_P(\FFFF_i),\RRR_P(\JJJJ_j)$ and $\RRR_P(\EEEE_k)$ divide $\C$ into $d$ parts, which we denote by $V_1,\ldots,V_d$ such that $\RRR_P^*(\theta)\subseteq V_i$ if and only if $\theta\in L_i$. We call $V_1,\ldots,V_d$ the \emph{$\Theta$-unlinked  classes in the dynamical plane of $P$}. It is easy to see that the restriction of $P$ on each $V_i$ is injective and  $\C\setminus P(V_i)$  consists of finitely many extended rays/external rays/segments of external rays, with arguments in $\tau_d(\Theta)$.

Let $c$ be the center of a bounded Fatou component $U(c)$ of $P$, such that $c=c_0\mapsto c_1\mapsto\cdots\mapsto c_{k-1}\mapsto c_k$ are the first $k$ terms in the orbit of $c$ with $P'(c_i)\not=0$ for $0\leq i\leq k-1$, but $P'(c_k)=0$. Then $c_k$ corresponds to an element $\FFFF_*$ of $\FFFF$, with the preferred angle denoted as $\theta_k$, and the ray $\RRR_P^\sigma(\theta_k)$ left-supports $U(c_k)$ with $\sigma\in\{+,-\}$. Note that each of $c_0,\ldots,c_{k-1}$ belongs to one of $V_1,\ldots,V_d$, so we write $c_i\in V_{\epsilon(i)}$ for each $0\leq i\leq k-1$. By the argument in last paragraph, the ray $\RRR_P^\sigma(\theta_k)$, or the angle $\theta_k$, can be lifted along to orbit of $c$, such that, for each $0\leq i\leq k-1$, the ray $\RRR_P^{\sigma}(\theta_i)$ is the unique one in $V_{\epsilon(i)}$ with $P(\RRR_P^\sigma(\theta_i))=\RRR_P^\sigma(\theta_{i+1})$. It is clear that the ray $\RRR_P^\sigma(\theta_i)$ left-supports $U(c_i)$ for each $0\leq i\leq k-1$.

In the dynamical plane of $\wt{P}$, the ray $\RRR_{\wt{P}}(\theta_k)$ left-supports the Fatou component of $\wt{P}$ corresponding to $\FFFF_*\in\FFFF$.  Then each $\RRR_{\wt{P}}(\theta_i)$ also left-supports a (unique) Fatou component of $\wt{P}$. We define $\phi_1(c)$ to be the center of the Fatou component of $\wt{P}$ left-supported by $\RRR_{\wt{P}}(\theta_0)$. We thus obtain a map $\phi_1$ from the centers of Fatou components of $P$ to those of $\wt{P}$. Note that all  processes above can be reversed from the dynamical plane of $\wt{P}$ to that of $P$, so the map $\phi_1$ is a bijection. The formula $\wt{P}\circ\phi_1=\phi_1\circ P$ follows directly from the construction of $\phi_1$.
\end{proof}

\begin{lemma}\label{lem:map-Julia}
Let $\RRR_P^\sigma(\theta)$ and $\RRR_P^{\sigma'}(\theta')$ land at the same point in the dynamical plane of $P$ with $\sigma,\sigma'\in\{+,-\}$. Then the rays $\RRR_{\wt{P}}(\theta)$ and $\RRR_{\wt{P}}(\theta')$ also land at the same point in the dynamical plane of $\wt{P}$. As a consequence, we get a surjection $\phi_2:\JJJ_P\to\JJJ_{\wt{P}}$ with $\wt{P}\circ\phi_2=\phi_2\circ P$.
\end{lemma}

\begin{proof}
We first claim that if $\theta,\theta'$ belong to the closure of a $\Theta$-unlinked class and $\RRR_{\wt{P}}(\tau(\theta)),\RRR_{\wt{P}}(\tau(\theta'))$ land at a common point, then $\RRR_{\wt{P}}(\theta),\RRR_{\wt{P}}(\theta')$ also land together.

We assume $\theta,\theta'\in \ov{L}$ with $L$ a $\Theta$-unlinked  class write by
\[L=(\theta_0,\theta_1)\cup\ldots\cup(\theta_{2p-1},\theta_{2p})\]
as in Lemma \ref{lem:d-parts}. In the case of  $\tau(\theta)=\tau(\theta')$, by Lemma \ref{lem:d-parts}, there exists $1\leq i\leq p$ such that $\{\theta,\theta'\}=\{\theta_{2i-1},\theta_{2i}\}$. Combining the fact that $\RRR_P^\sigma(\theta),\RRR_P^{\sigma'}(\theta')$ land together, the angles $\theta,\theta'$ must belong to an element of $\LLLL$. It follows that $\RRR_{\wt{P}}(\theta),\RRR_{\wt{P}}(\theta')$ land at a Julia critical point of $\wt{P}$, proving the claim.
In the case of $\tau(\theta)\not=\tau(\theta')$, let $\ov{\RRR_{\wt{P}}(\theta)}\cup\RRR_{\wt{P}}(\alpha)$ be a lift by $\wt{P}$ within $\ov{V}$ of the arc $\ov{\RRR_{\wt{P}}(\tau(\theta))}\cup\RRR_{\wt{P}}(\tau(\theta'))$, where $V$ denote the $\Theta$-unlinked class in the dynamical plane of $\wt{P}$ corresponding to $L$. We see that $\alpha$ is a preimage by $\tau$ of $\tau(\theta')$ within $\ov{L}$. If $\theta'\in L$, it is the unique preimage by $\tau$ of $\tau(\theta')$ within $\ov{L}$, and hence $\alpha=\theta'$, which complete the proof of the claim. Otherwise, by Lemma \ref{lem:d-parts} again, we have $\alpha,\theta'\in\{\theta_{2i-1},\theta_{2i}\}$ for some $1\leq i\leq p$, and we just need to deal with the case of $\alpha\not=\theta'$. According to property (C2), there are three possibilities:
\begin{itemize}
\item $\theta_{2i-1}$ belongs to an element of $\LLLL$ and $\theta_{2i}$ belong to an element of $\FFFF$;
\item $\theta_{2i-1},\theta_{2i}$ belong to a common element of $\FFFF$;
\item $\theta_{2i-1},\theta_{2i}$ belong to a common element of $\LLLL$;
\end{itemize}
However, since $\RRR_{P}^{\sigma}(\theta),\RRR_{P}^{\sigma'}(\theta')$ land together, only the third case happens. It follows that $\RRR_{\wt{P}}(\alpha),\RRR_{\wt{P}}(\theta')$ land together, and the claim is then proved.

We start to prove the lemma. If $\theta$ is not preperiodic, the fact that $\RRR^{\sigma}_P(\theta),\RRR_P^{\sigma'}(\theta')$ land together implies $\theta$ and $\theta'$ have the same itineraries, and  hence $\RRR_{\wt{P}}(\theta),\RRR_{\wt{P}}(\theta')$ land at the same point by Lemma \ref{lem:itinarary}. So we can assume that $\theta,\theta'$ are preperiodic. The proof goes by induction on the preperiod of $\theta$ (hence $\theta'$).

 Suppose $\theta$ is period. If the orbits of both $\theta$ and $\theta'$ avoids the arguments in $\Theta_P$, then $\theta,\theta'$ must have the same itinerary (because $\RRR^{\sigma}_P(\theta),\RRR_P^{\sigma'}(\theta')$ land together), and  Lemma \ref{lem:itinarary} shows that $\RRR_{\wt{P}}(\theta),\RRR_{\wt{P}}(\theta')$ land together. Otherwise, we can assume that $\RRR_{P}^{\sigma}(\theta)$ left-supports a periodic critical Fatou component $U$ at $z\in\partial U$, and $\RRR^{\sigma'}_{P}(\theta')$ lands at $z$. In this case, we have $\ell^-_\Theta(\theta)=\ell^-_\Theta(\theta')$. Thus, using Lemma \ref{lem:itinarary} again, the ray pair  $\RRR_{\wt{P}}(\theta),\RRR_{\wt{P}}(\theta')$ land together.

 Now, assume that $\RRR_{\wt{P}}(\theta),\RRR_{\wt{P}}(\theta')$ land together for all $\theta,\theta'$ with  preperiods strictly less than $n\geq1$. Let $\theta,\theta'$ have preperiod $n$. By the assumption of the induction, the rays  $\RRR_{\wt{P}}(\tau(\theta)),\RRR_{\wt{P}}(\tau(\theta'))$ land together.
Since $\RRR_P^{\sigma}(\theta)$ and $\RRR_P^{\sigma'}(\theta')$ land at a common point, then either the arguments $\theta,\theta'$ belong to the closure of a $\Theta$-unlinked class, or there exists an element $\JJJJ_*$ of $\JJJJ_P$ and two angles $\alpha,\alpha'\in\JJJJ_*$ such that $\theta,\alpha$ (resp. $\theta',\alpha'$) belong to the closure of a $\Theta$-unlinked class. In the former case, we get that $\RRR_{\wt{P}}(\theta),\RRR_{\wt{P}}(\theta')$ land at a common point directly by the claim. In the latter case, following the claim above, we have that $\RRR_{\wt{P}}(\theta),\RRR_{\wt{P}}(\alpha)$ (resp. $\RRR_{\wt{P}}(\theta),\RRR_{\wt{P}}(\theta')$) land together. Since $\alpha,\alpha'\in\JJJJ_*\in \JJJJ_P\subseteq\JJJJ_{\wt{P}}$, then $\RRR_{\wt{P}}(\alpha),\RRR_{\wt{P}}(\alpha')$ land at a common point, and hence $\RRR_{\wt{P}}(\theta),\RRR_{\wt{P}}(\theta')$ also land together.

We now construct the map $\phi_2:\JJJ_P\to\JJJ_{\wt{P}}$ as follows. Let $z\in \JJJ_P$ and the ray $\RRR_P^{\sigma}(\theta)$ land at $z$. We define $\phi_2(z)$ the landing point of the external ray $\RRR_{\wt{P}}(\theta)$. This map is well-defined by the conclusion proved above, and is a surjection satisfying $\wt{P}\circ\phi_2=\phi_2\circ P$ by its definition.
\end{proof}

\begin{lemma}\label{lem:Julia-Fatou}
Let $U$ be a Fatou component of $P$ and $\RRR_P(\theta)$ a preperiodic external ray lands on $z\in\partial U$ such that the orbit of $z$ avoids the landing point of $\RRR^\sigma_P(t)$ for any $t$ participating $\Theta_P$ and $\sigma\in\{+,-\}$.  Then the ray $\RRR_{\wt{P}}(\theta)$ lands on the boundary of the Fatou component $\wt{U}$ with the property that $\phi_1$ sends the center of $U$ to the center of $\wt{U}$.
\end{lemma}
\begin{proof}
Recall that $L_1,\ldots,L_d$ denote the $\Theta$-unlinked classes, and $V_1(P),\ldots,V_d(P)$, \emph{resp}. $V_1(\wt{P})$, $\ldots,V_d(\wt{P})$, denote the corresponding $\Theta$-unlinked classes in the dynamical plane of $P$ (resp. $\wt{P}$).  The assumption on $z$ implies that any point in the orbit of $z$ belongs to one of $V_1(P),\ldots,V_d(P)$. Then $z$ has the same itinerary as that of $\theta$ in the sense that, if $\ell_\Theta(\theta)=\epsilon_0\epsilon_1\ldots$, then $P^i(z)\in V_{\epsilon_i}(P)$ for any $i\geq0$.

Let $\wt{U}$ be the Fatou component of $\wt{P}$ such that $\phi_1$ sends the center of $U$ to the center of $\wt{U}$. Since $P$ and $\wt{P}$ have the same critical marking $\Theta$, it is not difficult to see that there is a unique preperiodic point $\wt{z}\in\partial \wt{U}$ such that $\wt{P}^i(\wt{z})\in V_{\epsilon_i}(\wt{P})$ for any $i\geq0$. Let $\RRR_{\wt{P}}(\alpha)$ be a external ray landing at $\wt{z}$. It then follows that $\alpha$ and $\theta$ have the same itineraries. By Lemma \ref{lem:itinarary}, the ray $\RRR_{\wt{P}}(\theta)$ lands at $\wt{z}\in\partial \wt{U}$.
\end{proof}


\begin{proof}[Proof of Lemma \ref{pro:less}]

Let $\HHH=\HHH_P$ denote the Hubbard forest of $P$, with the vertex set $V(\HHH)$ consisting of bounded critical/postcritical points of $P$ and the branched points of $\HHH$.
We establish a map $\phi:V(\HHH)\to\C$ such that
\[\phi(z)=\left\{
            \begin{array}{ll}
              \phi_1(z), & \hbox{if $z\in\FFF_P$;} \\
              \phi_2(z), & \hbox{if $z\in\JJJ_P$,}
            \end{array}
          \right.
\]
where $\phi_1,\phi_2$ are maps given in Lemmas \ref{lem:map-Fatou}, \ref{lem:map-Julia} respectively. It follows that $\wt{P}\circ\phi(z)=\phi\circ P(z)$.

For the discussion below, we represent each point in $V(\HHH)$ and $\phi(V(\HHH))$ by an angle. Let $v\in V(\HHH)$. If $v\in\FFF_P$, we define $\theta(v)$  the unique iterated preimage by $\tau$ of preferred angles participating $\FFFF$ such that $\RRR_P^\sigma(\theta(v))$ left-supports $U(v)$ where $\sigma\in\{+,-\}$ and $U(v)$ the Fatou component of $P$ containing $v$. According to the argument in Lemma \ref{lem:map-Fatou}, the ray $\RRR_{\wt{P}}(\theta(v))$ left-supports the Fatou component $\wt{U}(\phi(v))$ of $\wt{P}$ which contains $\phi(v)$, and we represent $\phi(v)$ also by $\theta(v)$. If $v\in\JJJ_P$, let $\theta(v)$ be an arbitrary angle in ${\rm arg}_v(P)$. By Lemma \ref{lem:map-Julia}, the angle $\theta(v)$ belong to ${\rm arg}_{\phi(v)}(\wt{P})$, and we represent $\phi(v)$ also by $\theta(v)$.

For every component $H$ of $\HHH$, let $V(H):=V(\HHH)\cap H$ denote the vertex of $H$. We define by $\wt{H}$  the regulated hull of $\phi(V(H))$ within $\KKK_{\wt{P}}$. It is a tree in the dynamical plane of $\wt{P}$, and the vertex set $V(\wt{H})$ of $\wt{H}$ is defined as the union of $\phi(V(H))$ and the branched points of $\wt{H}$.

\noindent\emph{Claim 1. The restriction of $\phi$ on $V(H)$ is injective}.

\begin{proof}
By Lemma \ref{lem:map-Fatou}, the restriction of $\phi$ on $V(H)\cap\FFF_P$ is injective. Let $z,w$ be any distinct points in $V(H)\cap \JJJ_P$. If the segment $[z,w]$ intersects a Fatou component $U$ of $P$, we can find two preperiodic angles $\alpha,\beta$ such that $\RRR_P(\alpha),\RRR_P(\beta)$ land at the boundary of $U$, the periods of $\alpha,\beta$ are larger than the angles participating $\Theta$ and those in ${\rm arg}_z(P),{\rm arg}_w(P)$, and the arc
\[\G:=\RRR_P(\theta)\cup I_\alpha\cup I_\beta\cup \RRR_P(\beta)\]
separates $z,w$, with $I_\alpha,I_\beta$ the internal rays in $U$ which have the same landing points as $\RRR_P(\alpha),\RRR_P(\beta)$ respectively. This arc also separates the rays $\RRR_P(\theta(z))$ and $\RRR_P(\theta(w))$ with $\theta(z),\theta(w)$ the angles representing $z,w$.

In the dynamical plane of $\wt{P}$, Lemma \ref{lem:Julia-Fatou} implies that $\RRR_{\wt{P}}(\alpha),\RRR_{\wt{P}}(\beta)$ land on the same Fatou component $\wt{U}$ of $\wt{P}$ where $\phi$ sends the center of $U$ to the center of $\wt{U}$. Moreover, the arc
\[\wt{\G}:=\RRR_{\wt{P}}(\theta)\cup \wt{I}_\alpha\cup \wt{I}_\beta\cup \RRR_{\wt{P}}(\beta)\]
separates $\RRR_{\wt{P}}(\theta(z))$ and $\RRR_{\wt{P}}(\theta(w))$, with $\wt{I}_\alpha,\wt{I}_\beta$ the internal rays in $\wt{U}$ which have the same landing points as $\RRR_{\wt{P}}(\alpha),\RRR_{\wt{P}}(\beta)$ respectively. Since the periods of $\alpha,\beta$ are chosen larger than $\theta(z),\theta(w)$, then the landing points of $\RRR_{\wt{P}}(\theta(z)),\RRR_{\wt{P}}(\theta(w))$, which are exactly $\phi(z),\phi(w)$, are disjoint with the $\wt{\G}$. It means that $\wt{\G}$ separates $\phi(z),\phi(w)$ and hence $\phi(z)\not=\phi(w)$.

If the segment $[z,w]\subseteq\JJJ_P$, there exist preperiodic angles $\alpha,\beta$ such that $\RRR_P(\alpha),\RRR_P(\beta)$ land at the same point $x\in(z,w)$, the preperiods of $\alpha,\beta$ are larger than  those of $\theta(z),\theta(w)$, and the arc
\[\G:=\RRR_P(\theta)\cup\{x\}\cup \RRR_P(\beta)\]
separates $z,w$. In the dynamical plane of $\wt{P}$, Lemma \ref{lem:map-Julia} implies that $\RRR_{\wt{P}}(\alpha),\RRR_{\wt{P}}(\beta)$ land on the same point $\wt{x}$, and the arc
\[\wt{\G}:=\RRR_{\wt{P}}(\theta)\cup \{\wt{x}\}\cup \RRR_{\wt{P}}(\beta)\]
separates $\RRR_{\wt{P}}(\theta(z))$ and $\RRR_{\wt{P}}(\theta(w))$. Since the preperiods of $\alpha,\beta$ are chosen larger than those of  $\theta(z),\theta(w)$, then the landing points of $\RRR_{\wt{P}}(\theta(z)),\RRR_{\wt{P}}(\theta(w))$, which are exactly $\phi(z),\phi(w)$, are disjoint with  $\{\wt{x}\}$. It means that $\phi(z)\not=\phi(w)$. We complete the proof of the injection of $\phi|_{V(H)}$.
\end{proof}

The map $\phi|_{V(H)}:V(H)\to V(\wt{H})$ can be extended to a map,  denoted by $\phi_H$, from $H$ to $\wt{H}$.
Let $e=e(x,y)$ be an edge of $H$ with endpoints $x,y$. We define its image $\phi_H(e)$ to be the regulated arc $[\phi(x),\phi(y)]$ within $\wt{H}$. Note that the endpoints of $\wt{H}$ belong to $\phi(V(H))$ (since $\wt{H}$ is defined as the regulated hull  of $\phi(V(H))$ within $\KKK_{\wt{P}}$), then the map $\phi_H:H\to\wt{H}$ is surjective.

\noindent \emph{Claim 2. The map $\phi_H:H\to\wt{H}$ is a homeomorphism.}

In the proof of Claim 1, we represent each vertex of $V(H)$ by an angle. To prove Claim 2, we will represent each edge of $H$ by a subset of $\T$ consisting of two open arcs.

Let $e=e(z_1,z_2)$ be an arc of $H$ with $z_1,z_2\in V(H)$.  In the case of $z_1\in\FFF_P$, we choose two preperiodic angles $\alpha_1,\alpha_1'$ with periods larger than those of the angles participating $\Theta$, such that $\RRR_P(\alpha_1),\RRR_P(\alpha_1')$ land on the boundary of $U(z_1)$, and the region $W(\alpha_1,\alpha_1')$ contains all edges of $H$ starting from $z_1$ except $e$, where $\wh{\RRR}_P(\alpha_i)$ denote the extended ray of $P$ associated to $U(z_1)$ with argument $\alpha_i,i=1,2$, and $W(\alpha_1,\alpha_1')$ is bounded by $\wh{\RRR}_P(\alpha_1),\wh{\RRR}_P(\alpha_2)$ such that $\RRR_P(\alpha)\in W(\alpha_1,\alpha_1')$ for all $\alpha\subseteq(\alpha_1,\alpha_1')$.  In the case of $z_1\in\JJJ_P$, let $\alpha_1,\alpha_1'\in {\rm arg}_z(P)$ such that  $\ov{W(\alpha_1,\alpha_1')}$ contains all external rays landing at $z_1$, where $W(\alpha_1,\alpha_1')$ is bounded by ${\RRR}^{\sigma_1}_P(\alpha_1),{\RRR}^{\sigma_1'}_P(\alpha_1')$ such that $\RRR^{\pm}_P(\alpha)\in W(\alpha_1,\alpha_1')$ for all $\alpha\subseteq(\alpha_1,\alpha_1')$ (see Figure \ref{fig:represent}). We similarly choose a pair of angles $\alpha_2,\alpha_2'$ for the other endpoint $z_2$ of $e$.
\begin{figure}[http]
 \includegraphics[scale=0.63]{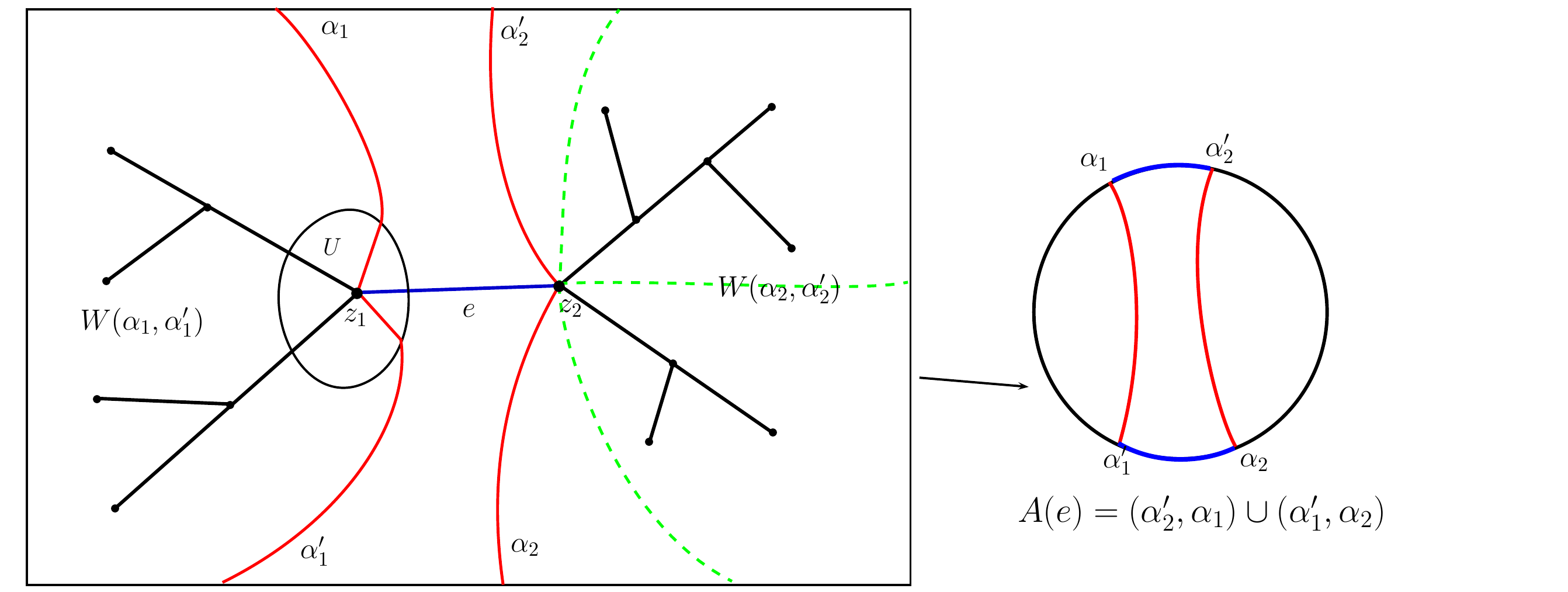}
\caption{The construction of $W(\alpha_1,\alpha_1'), W(\alpha_2,\alpha_2')$ and $A(e)$. Here we assume that $z_1$ is the center of a Fatou component $U$ and $z_2\in\JJJ_{P}$.}
\label{fig:represent}
\end{figure}

Now, we define $A(e)$ the union of two disjoint arcs $(\alpha_1',\alpha_2),(\alpha_2',\alpha_1)$, which is called the \emph{representation of $e$}. By the construction, it is easy to see that

{\noindent{\bf Fact:} \emph{the representing angle $\theta(v)$ of a vertex $v\in V(H)$ belongs to $A(e)$ only if $v\in \FFF_P$ and $v$ equals to either $z_1$ or $z_2$}.}

\begin{proof}[Proof of Claim 2]
Since $\phi_H:H\to\wt{H}$ is surjective and injective on $V(H)$,  it is enough to show that $\phi_H$ maps an edge of $H$ to an edge of $\wt{H}$. Let $e=e(z_1,z_2)$ be an edge of $H$ with its representation $A(e)=(\alpha_1',\alpha_2)\cup(\alpha_2,\alpha_1)$. By Lemmas \ref{lem:map-Julia} and \ref{lem:Julia-Fatou}, the ray pair $\RRR_{\wt{P}}(\alpha_i),\RRR_{\wt{P}}(\alpha_i')$ either land on the boundary of the Fatou component $U(\phi_H(z_i))$ of $\wt{P}$ or on the common point $\phi_H(z_i)$, according to $z_i$ belongs to $\FFF_P$ or not, for $i=1,2$. We then obtain a region $W$ bounded by the four external or extended rays of $\wt{P}$ with arguments $\alpha_1,\alpha_1',\alpha_2,\alpha_2'$. Clearly, the open segment $(\phi_H(z_1),\phi_H(z_2))$ belongs to $W$.

On the contrary, if $[\phi_H(z_1),\phi_H(z_2)]$ is not an edge of $\wt{H}$, there must be a point $z\in V(H)$ distinct with $z_1,z_2$ such that $\phi_H(z)\in W$. It implies the representation $\theta(z)$ of $\phi_H(z)$ belongs to $A(e)$, a contradiction to the Fact above.
\end{proof}

\noindent\emph{Claim 3. Let $H_1,H_2$ be two distinct connected components of $\HHH$. Then their images $\wt{H_1},\wt{H_2}$ by $\phi_{H_1},\phi_{H_2}$ intersect only possibly at their endpoints}.

\begin{proof}
Since $H_1,H_2$ belongs to different components of $\KKK_P$, there exist a pair of angles $\eta_1,\eta_2$ such that the external radii of $P$ with arguments $\eta_1,\eta_2$ terminates at a pre-critical point $z$, the arc $\RRR^*_P(\eta_1)\cup\{z\}\cup\RRR^*_P(\eta_2)$ separates $H_1,H_2$, and the rays $\RRR^+_P(\eta_1),\RRR^-_P(\eta_2)$ land together. As a consequence, the angles representing $V(H_1)$ and those representing $V(H_2)$ are contained in different components of $\T\setminus\{\eta_1,\eta_2\}$ except possibly the ones coincide with $\eta_1,\eta_2$.

In the dynamical plane of $\wt{P}$,  the external rays of $\wt{P}$ with arguments $\eta_1,\eta_2$ land at a common point according to Lemma \ref{lem:map-Julia}, and the union of these two rays separates the vertices of $\wt{H_1}$ and those of $\wt{H_2}$ except possibly the one coincide with the landing point. Then the claim is proved.
\end{proof}

Now, let $\wt{\HHH}$ denote the union of all $\wt{H}$ with $H$ going though the components of $\HHH$. By Claims 1, 2 and 3, the set $\wt{\HHH}$ is a forest contained in the Hubbard tree $\HHH_{\wt{P}}$, with its vertex set $V(\wt{\HHH})$ equal to $\phi(V(\HHH))$, and we have the formula
\begin{equation}\label{eq:00}
\wt{P}\circ\phi=\phi\circ P\text{ on $V(\HHH)$}.
\end{equation}
By this formula, we have $\wt{P}(V(\wt{\HHH}))\subseteq V(\wt{\HHH})$ and $\wt{P}(\wt{\HHH})\subseteq \wt{\HHH}$. According to Claim 3, the interior of any edge of $\wt{\HHH}$ contains no critical points of $\wt{P}$, then $\wt{P}:\wt{\HHH}\to \wt{\HHH}$ is a Markov map.

It is known from Proposition \ref{pro:algorithm} that $h(\wt{P})=h(\Theta)$, and from Lemma \ref{Do2} that $h_{top}(\wt{P}|_{\wt{\HHH}})\leq h(\wt{P})$. Thus,
to complete the proof of Lemma \ref{pro:less}, we only need to show $h(P):=h_{top}(P|_\HHH)=h_{top}(\wt{P}|_{\wt{\HHH}})$.

By enumerating the edges of $\HHH$, we obtain an incidence matrix $D=D_{(\HHH,P)}$ of $(\HHH,P)$, and the topological entropy $h_{top}(P|_\HHH)$ equals to $\log\rho(D)$ with $\rho(D)$ the leading eigenvalue of $D$ (see Section \ref{sec:topological-entropy}). By Claims 2 ,3, we get a bijection $\phi_*:E(\HHH)\to E(\wt{\HHH})$ such that $\phi_*(e)=\phi_H(e)$ if $e$ is an edge of $\HHH$ contained in a component $H$ of $\HHH$. Thus, the enumeration on $E(\HHH)$ induces an enumeration on $E(\wt{\HHH})$ such that $e$ and $\phi_*(e)$ have the same label.

Let $e=e(x,y)$ be any edge of $\HHH$ with endpoints $x,y\in V(\HHH)$. Since $P:\HHH\to\HHH$ is a Markov map, then $P(e)=[P(x),P(y)]$ consists of several edges of $\HHH$, denoted as $e_1,\ldots,e_r$. The edge $\phi_*(e)$ of $E(\wt{H})$ has the endpoints $\phi(x),\phi(y)$. Since the map $\wt{P}:\wt{\HHH}\to\wt{\HHH}$ is also Markov,  the image $\wt{P}(\phi(e))$ of $\phi(e)$ equals to $[\wt{P}\circ\phi(x),\wt{P}\circ\phi(y)]$. Using Claim 2 successively on $e_1,\ldots,e_r$, we get that the segment $\wt{P}(\phi(e))$ consists of exactly the edges $\phi_*(e_1),\ldots,\phi_*(e_r)$ of  $E(\wt{\HHH})$. It means that the two Markov maps $(\HHH,P)$ and $(\wt{\HHH},\wt{P})$ have the same incidence matrix, and hence the same topological entropy.  We then complete the proof of Lemma \ref{pro:less}.
\end{proof}

\subsection{The limit inferior of the entropy function}
According to Proposition \ref{pro:limsup}, the continuity of the entropy function at a \pf parameter $P$ depends only on the relation of $h(P)$ and the limit inferior $\liminf_{Q\to P}h(Q)$: the entropy function is continuous at $P$ if and only if $h(P)\leq \liminf_{Q\to P}h(Q)$. In this subsection, we will find this limit inferior by the dynamics of $P$ and prove Proposition \ref{pro:key}.




Let $P$ be a monic \pf polynomial, and $\JJJJ:=\{\Theta_1(c_1),\ldots,\Theta_r(c_r)\}$ a weak Julia critical marking of $P$ (see Definition \ref{def:weak-critical-marking}).  
We are able to construct a puzzle induced by $\JJJJ$. Let $W:=\{z\in\C:G_P(z)<2\}$, and $\RRR_i$ denote the union of external rays of $P$ with arguments in $\Theta_i(c_i)$ and the critical points $c_i$ for every $i=1,\ldots,m$.
We say that a subset $B\subseteq \ov{W}$ is a \emph{$\JJJJ$-puzzle piece of level $0$} if $B$ is maximal such that any two points in $B$ are not separated by $\RRR_i$ for all $i=1,\ldots,m$.
Then each $\JJJJ$-puzzle piece $B$ of level $0$
 is a full continuum (i.e., a non-trivial, connected, compact set in $\ov{\C}$ such that its complementary is also connected), with boundary consisting of segments of external rays and equipotential line of potential $2$,
and $P(B)=\{z\in\C:g_P(z)\leq 2^d\}$ (see the left one in Figure \ref{fig:puzzle}).
\begin{figure}[http]
\begin{tikzpicture}
\node at (-4,0) {\includegraphics[width=7.5cm]{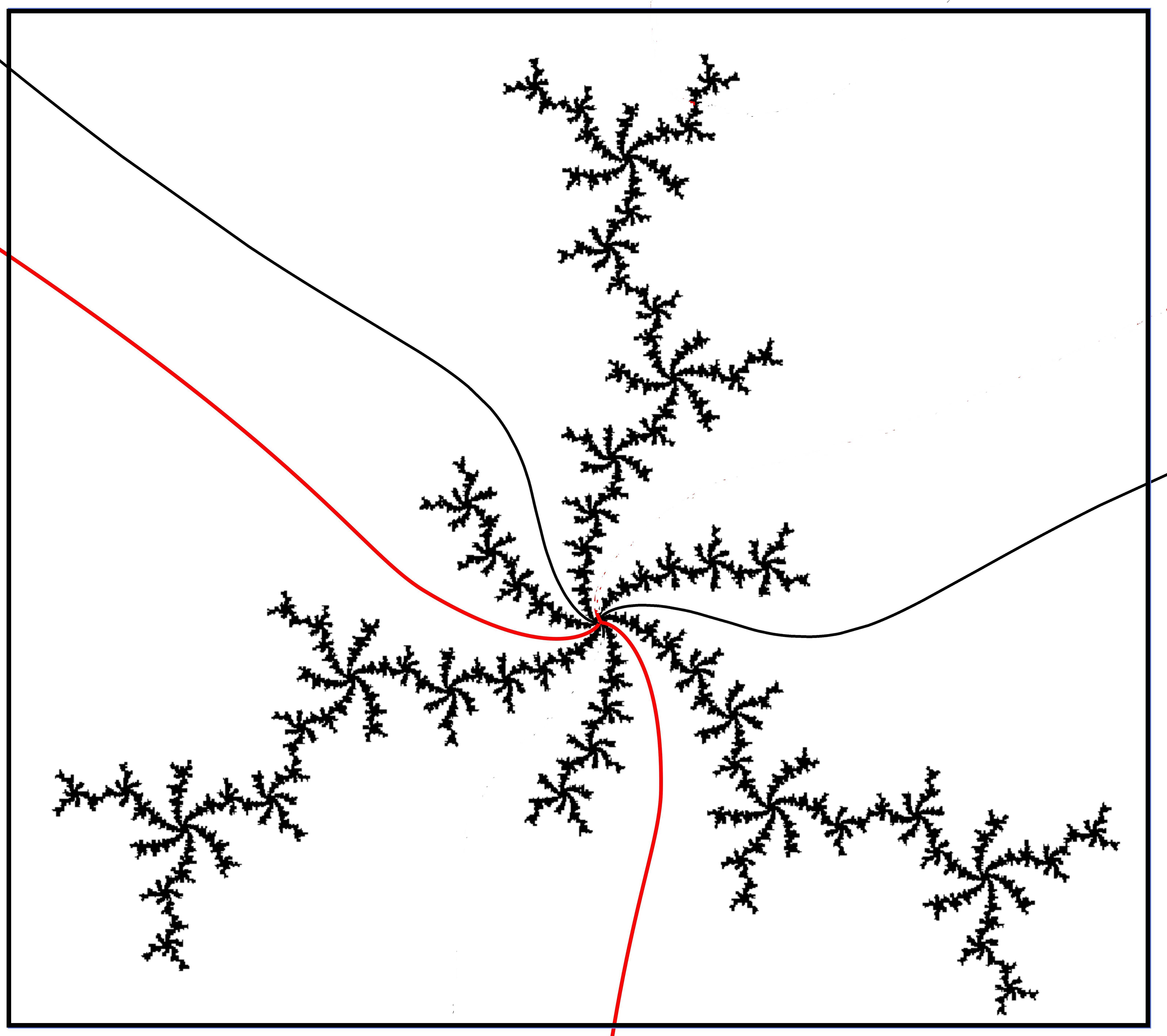}};
\node at(4,0){\includegraphics[width=7.5cm]{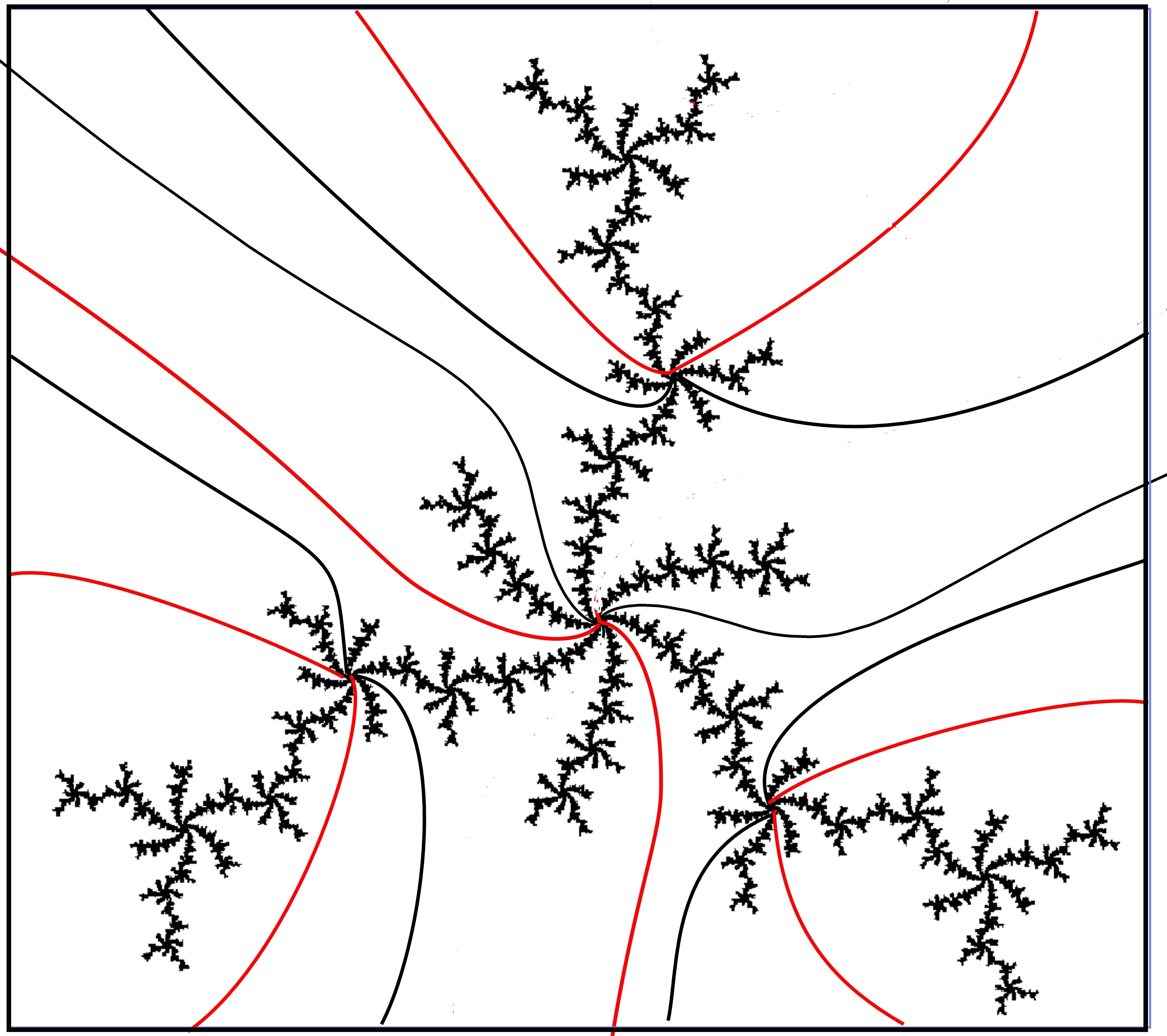}};
\node at(-7,2.75){\footnotesize{$\frac{83}{216}$}};
\node at(-7.25,1.75){\footnotesize{$\frac{89}{216}$}};
\node at(-4,-3){\footnotesize{$\frac{161}{216}$}};
\node at(-0.75,0.4){\footnotesize{$\frac{11}{216}$}};
\node at(-6.75,-0.5){$B_1$};
\node at (-6.25,1.25){$B_2$};
\node at(-1.25,-1.25){$B_2$};
\node at(-1.5,2){$B_3$};
\node at(1,-1){\footnotesize{$\textcircled{$\mathbf{1}$}$}};
\node at(1,0){\footnotesize{$\textcircled{$\mathbf{2}$}$}};
\node at(2.5,-2.5){\footnotesize{$\textcircled{$\mathbf{2}$}$}};
\node at(3.5,-1.5){\footnotesize{$\textcircled{$\mathbf{3}$}$}};
\node at(2,1){\footnotesize{$\textcircled{$\mathbf{4}$}$}};
\node at(5.5,-1){\footnotesize{$\textcircled{$\mathbf{4}$}$}};
\node at(7,-0.8){\footnotesize{$\textcircled{$\mathbf{5}$}$}};
\node at(5,-3){\footnotesize{$\textcircled{$\mathbf{5}$}$}};
\node at(6,-2.75){\footnotesize{$\textcircled{$\mathbf{6}$}$}};
\node at(5,0.25){\footnotesize{$\textcircled{$\mathbf{7}$}$}};
\node at(3,2.25){\footnotesize{$\textcircled{$\mathbf{8}$}$}};
\node at(6.26,1.25){\footnotesize{$\textcircled{$\mathbf{8}$}$}};
\node at(5.5,2.5){\footnotesize{$\textcircled{$\mathbf{9}$}$}};
\end{tikzpicture}
\caption{Consider the polynomial $P(z)=z\mapsto z^3+0.22036+1.18612 i$ in Figure \ref{fig:weak-portrait}. Then $\JJJJ=\large\{\ \{11/216,83/216\},\{89/216,161/216\}\ \large\}$ is a Julia weak critical marking of $P$. Its $\JJJJ$-puzzle pieces of level $0$ are label in the left figure by $B_1,B_2,B_3$ (including the boundary rays). In the right figure, the closure of each region labeled by $\textcircled{1},\ldots,\textcircled{9}$ corresponds to a $\JJJJ$-puzzle piece of level 1.}
\label{fig:puzzle}
\end{figure}
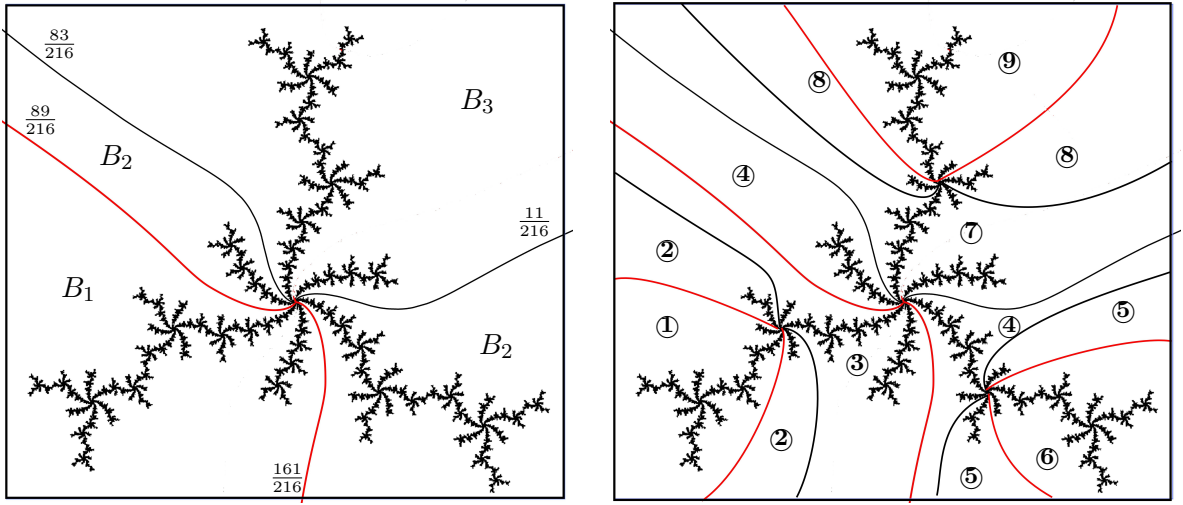

The $\JJJJ$-puzzle pieces of level $1$ is defined as follows. Let $B,B'$ be any pair of $\JJJJ$-puzzle piece of level $0$. Since $P(B)$ covers $B'$, we call each component of $P|_B^{-1}(B')$  a $\JJJJ$-puzzle piece of level $1$ (see the right one in Figure \ref{fig:puzzle}). By this definition, the following results hold:
\begin{enumerate}
\item each $\JJJJ$-puzzle piece of level $1$ is a full continuum, with boundary consisting of segments of external rays and equipotential line of potential $2/d$, and different puzzle pieces of level $1$ have pairwise disjoint interiors;
\item each puzzle piece of level $1$ is contained in a puzzle piece of level $0$;
\item the map $P$ sends a puzzle piece of level $1$ onto a puzzle piece of level $0$.
\end{enumerate}
Inductively, we can define the $\JJJJ$-puzzle pieces of level $k+1$ ($k\geq1$) as the components of $(P|_{B_k})^{-1}(B_k')$ for all pairs $B_k,B_k'$ of puzzle pieces of level $k$ with $B_k'\subseteq P(B_k)$, and the following properties can be inductively checked:
\begin{enumerate}
\item each $\JJJJ$-puzzle piece of level $k+1$ is a full continuum, with boundary consisting of segments of external rays and equipotential line of potential $2/d^k$, and different puzzle pieces of level $k+1$ have pairwise disjoint interiors;
\item each puzzle piece of level $k+1$ is contained in a puzzle piece of level $k$;
\item the map $P$ sends a puzzle piece of level $k+1$ onto a puzzle piece of level $k$;
\end{enumerate}

Let $\{B_n,n\geq0\}$ be a sequence of nested $\JJJJ$-puzzle pieces such that $B_n$ has level $n$. The intersection $E:=\cap_{n\geq0}B_n$ is called a \emph{$\JJJJ$-end}. The $\JJJJ$-ends can be characterized by \emph{$\JJJJ$-itinerary}.  We say that $z,w\in \C$ have the same \emph{$\JJJJ$-itinerary} if the pair $\{P^k(z),P^k(w)\}$ are not separated by $\RRR_i$ for all $k\geq0$ and $1\leq i\leq m$.

\begin{lemma}[properties of $\JJJJ$-ends]\label{lem:end-property}
Let $E$ be a $\JJJJ$-end. Then we have
\begin{enumerate}
\item the end $E$ is a full continumm in $\KKK_P$ with its boundary in $ \JJJ_P$;
\item the image $P(E)$ is  an $\JJJJ$-end;
\item $E$ is a maximal connected set in $\KKK_P$ with the same $\LLLL$-itinerary, i.e., if $E'\subseteq \KKK_P$ is a connected set containing $E$ such that any two points in $E'$ have the same $\LLLL$-itinerary, then $E=E'$.
\item the end $E$ is non-trivial if and only if it contains Fatou components of $P$.
\end{enumerate}
\end{lemma}
\begin{proof}
Let $\{B_n,n\geq1\}$ be a sequence of nested $\JJJJ$-puzzle pieces such that $E=\cap_{n\geq1}B_n$. Point (1) follows directly from Property (1) of the $\JJJJ$-puzzle pieces. By Property (3) of the $\JJJJ$-puzzle pieces, $\{P(B_n),n\geq1\}$ is also a sequence of nested $\JJJJ$-puzzle pieces. Then $P(E)\subseteq E':=\cap_{n\geq1}P(B_n)$. To get the equality, let $y$ be a point in $E'$ and $x_n\in B_n$ a preimage of $y$ by $P$. Since $x_n$ has finitely many choices, and $\{B_n,n\geq1\}$ a nested sequence, there exists a preimage $x$ by $P$ of $y$ which belongs to $B_n$ for all large $n$. Then point (2) is proved.

To see point (3), we first show that any two points in $E$ have a common $\LLLL$-itinerary. Let $x,y$ be any two distinct points in $E$. Then $x,y\in B_n$ for any $n\geq1$. According to Properties (2),(3) of the $\JJJJ$-puzzle pieces, the pair $\{P^i(x),P^i(y)\}$ is not separated by any $\RRR_i$ for $i=0,\ldots,n$. As $n$ is arbitrary, the points $x,y$ have the same $\JJJJ$-itinerary.

 Let $E'\subseteq \KKK_P$ be a connected set such that $E\subseteq E'$ and $E'\setminus E\not=\emptyset$. Then there exists a minimal level $k$ such that $E$ belongs to a $\JJJJ$-puzzle pieces $B_k$ of level $k$ and $E'\setminus B_k\not=\emptyset$. Then there exists an arc $\G$ consisting of two external rays and their common landing point such that $\G$ separates $E'$, i.e., both components of $\C\setminus\G$ intersects $E'$, and  $P^k:\G\to\RRR_i$ is injective for some $1\leq i\leq m$. It follows that $P^k(\G)\subseteq \RRR_i$ separates $P^k(E')$, which complete the proof of point (3).

For (4), we just need to prove that if $E\subseteq\JJJ_P$, then $E$ is a singleton. On the contrary, let $z,w\subseteq E$ be distinct points such that $[z,w]\subseteq E$. For each $j\geq0$, the set $P^j(E)$ belongs to $\JJJ_P$ and is not separated by $\RRR_1,\ldots,\RRR_m$. Then, for each $j\geq1$, the map $P^j:[z,w]\to P^j([z,w])$ is a homeomorphism although some image $P^j([z,w])$ possibly intersects critical points of $P$. By the uniform expansionary of $P$ near $\JJJ_P$ (Lemma \ref{lem:sub-hyperbolic}), the length of $P^j([z,w])=[P^j(z),P^j(w)]$ converge to $\infty$ as $j\to\infty$, a contradiction.
\end{proof}

Let $\EEE_\JJJJ$ denotes the collection of periodic $\JJJJ$-ends which contain Fatou critical/postcritical points of $P$. Then $\EEE_\JJJJ$ is $P$-invariant. By Lemma \ref{lem:end-property}.(4), the set $\EEE_\JJJJ$ is empty if and only if $\KKK_P=\JJJ_P$, i.e., $P$ is \Mi. For each element $E$ of $\EEE_\JJJJ$, we define the tree $H_E$ as the regulated convex hull within $\KKK_P$ of the Fatou critical/postcritical points contained in $E$. We claim that $H_E\subseteq E$:
because for each $\JJJJ$-puzzle piece $B$ containing $E$, the tree $H_E$ lie in $B$ by Property (1) of $\JJJJ$-puzzle pieces, and then $H_E\subseteq E$.
The tree $H_E$ is a subtree of the Hubbard tree $\HHH_P$, and we define its vertex set $V(H_E)$ as the union of  critical/postcritical points of $P$ contained in $H_E$ and the branched points of $H_E$.

By the definition, if $E,E'$ are two elements of $\EEE_\JJJJ$ with $P(E)=E'$, then we have $P(V(H_E))\subseteq V(H_{E'})$ and $P(H_E)\subseteq H_{E'}$. So each $H_E$ is periodic under $P$. We claim that $E_1\cap E_2=\emptyset$ for any distinct  $E_1,E_2\in\EEE_\JJJJ$. Otherwise, the set $E_1\cup E_2$ is connected and contains $E_1,E_2$. Thus, by Lemma \ref{lem:end-property}.(3), there exists a minimal $k$ such that $P^k(E_1)$ and $P^k(E_2)$ are separated by $\RRR_i$ for some $1\leq i\leq m$. Without loss of generality, we assume that $k=0$. Then the intersection $E_1\cap E_2$ coincides with $c_{i}$.
On the other hand, since $E_1,E_2$ are periodic, there exists $p$ such that $P^p(E_j)\subseteq E_j$ for $j=1,2$. It follows that $P^p(c_{i})$ belongs to $E_1\cap E_2$, and hence equal to $c_{i}$. This is impossible because $c_{i}$ is strictly preperiodic.


Let $\HHH_\JJJJ$ denote the union of $H_E$ for all elements $E$ of $\EEE_\JJJJ$. Here we allow $\HHH_\JJJJ=\emptyset$, and this happens if and only if $\EEE_\JJJJ=\emptyset$, and if and only if $\JJJ_P=\KKK_P$ by Lemma \ref{lem:end-property}.(4). By the claim above, each $H_E$ with $E\in\EEE_\JJJJ$ is a component of $\HHH_\JJJJ$. So we define the vertex set $V(\HHH_\JJJJ)$ of $\HHH_\JJJJ$ as the union of $V(H_E)$ for all $E\in\EEE_\JJJJ$. Moreover, the forest $\HHH_\JJJJ$ is a $P$-invariant, contained in $\HHH_P$, and the map $P:\HHH_\JJJJ\to\HHH_\JJJJ$ is Markov.

Let $\{P_n,n\geq1\}$ be a sequence of partial \pf polynomials converging to $P$. We say  the sequence $\{P_n,n\geq1\}$ has \emph{type $\JJJJ$} if there exists a subset $\XXXX_n$ of a critical marking $\Theta_n$ of $P_n$ such that $\XXXX_n\to\JJJJ$ as $n\to\infty$. We call $\{P_n,n\geq1\}$ a \emph{maximal-hyperbolic} sequence if each $P_n$ is hyperbolic and the bounded critical points of $P_n$ converge to Fatou critical points of $P$.

\begin{proposition}\label{pro:liminf}
Let $\{P_n,n\geq1\}$ be a sequence of monic partial \pf polynomials converging to $P$ with type $\JJJJ$. Then
\begin{equation}\label{eq:11}
h_{top}(P|_{\HHH_\JJJJ})\leq h(P_n),
\end{equation}
for large $n$, and the equality holds if $\{P_n,n\geq1\}$ is maximal-hyperbolic. As a consequence, let $\mu_P$ denote the minimality of $h_{top}(P|_{\HHH_\JJJJ})$ with $\JJJJ$ going through all weak Julia critical markings of $P$, then we have
$\mu_P\leq \liminf\limits_{Q\to P}h(Q)$ with $Q$ chosen in monic partial \pf polynomials.
\end{proposition}

\noindent\emph{Proof.}
 The idea of the proof is similar to that of Proposition \ref{pro:limsup}: we will embed $\HHH_\JJJJ$ into the Hubbard forest $\HHH_{P_n}$ by an injective  $\phi_n:\HHH_\JJJJ\to\HHH_{P_n}$ for all large $n$, and prove that $h_{top}(P|_{\HHH_\JJJJ})=h_{top}(P_n|_{\phi_n(\HHH_\JJJJ)})$. As explained in the proof of Proposition \ref{pro:limsup}, we can assume these $P_n's$ are visible so that their Julia critical markings can be defined.

We first show the inequality \eqref{eq:11}. If $\HHH_\JJJJ=\emptyset$, the conclusion is obvious (because $h_{top}(P|_{\HHH_\JJJJ})=0$). So we assume in the following that $\HHH_\JJJJ\not=\emptyset$.

Since the sequence $\{P_n,n\geq1\}$ has type $\JJJJ=\{\Theta_i(c_i):1\leq i\leq m\}$, there exist a critical marking $\Theta_n:=\{\FFFF_n,\XXXX_n\}$ of $P_n$ with
\begin{equation}\label{eq:99}
\FFFF_n:=\{\Theta(U_{n,1}),\ldots,\Theta(U_{n,s})\};\ \XXXX_n:=\{\Theta(c_{n,1}),\ldots,\Theta(c_{n,m})\},
\end{equation}
such that $\Theta(c_{n,i})\to\Theta_i(c_i)$ for each $1\leq i\leq m$;
where $U_1,\ldots,U_s$ denote all critical Fatou components of $P$, and $U_{n,j}$ is the continuation of $U_j$ at $P_n$ for $j=1,\ldots,s$.
In the dynamical plane of $P_n$, we define $\RRR_{n,i}$ as the closures of the union of external rays/extended rays/external radii of $P_n$ with arguments in $\Theta(c_{n,i})$ which land  or terminate at $c_{n,i}$, for $i=1,\ldots,m$.
Two points $z,w\in \C$ is said to have the same \emph{$\XXXX_n$-itineraries} if $P_n^k(z),P_n^k(w)$ are not separated by $\RRR_{n,i}$ for all $k\geq0$ and $1\leq i\leq m$.

Fix an element $E$ of $\EEE_\JJJJ$.

\noindent\emph{Claim 1. Let $x,y$ be the centers of distinct Fatou components in $E$, and $x_n,y_n$ denote their continuations at $P_n$ (Lemma \ref{lem:Fatou-component}). Then $x_n,y_n$ have a common $\XXXX_n$-itinerary for large $n$}.

 \noindent \emph{Proof of Claim 1.} On the contrary, assume that there exist  infinite $n$ such that $x_n,y_n$ have different $\XXXX_n$-itinerary. For each such $n$, there exists a minimal number $k_n$ such that $P_n^{k_n}(x_n),P_{n,E}^{k_n}(y_n)$ are separated by
some $\RRR_{n,i_n}$.
Since all $x_n$ (resp. $y_n$) have the same preperiod and period, the number $k_n$ is uniformly bounded above. So, without loss of generality and by taking subsequences, we can assume that $x_n,y_n$ are separated by $\RRR_{n,i}$ for some $1\leq i\leq m$ and all large $n$. Note that $x_n\to x, y_n\to y$ and $\RRR_{n,i}\to \RRR_i$ (by Proposition \ref{pro:convergence4}), then $\RRR_i$ separates $x,y$,  contradicting that $x,y$ belong to a common $\JJJJ$-end.
\hfill\qedsymbol

 We denote by $M_E$ the Fatou-type vertices of $H_E$. For each large $n$, we define a map $\phi_{n,E}:M_E\to\C$ as $\phi_{n,E}(z)=:z_n$ with $z_n$ the continuation of $z$ at $P_n$. Claim 1 implies the points in $\phi_{n,E}(M_E)$ have a common $\XXXX_n$-itinerary.

\noindent\emph{Claim 2. Let $z$ be a Juila preperiodic point in $E$. Then there exists a unique point $z_n$ such that $z_n$ has the same preperiod and period by $P_n$ as those of $z$ by $P$, $z_n\to z$ as $n\to\infty$ and $z_n$ has the same $\XXXX_n$-itinerary as the points in $\phi_{n,E}(M_E)$.}

\noindent\emph{Proof of Claim 2}. We first assume that $z$ is periodic. Since $z$ is repelling, then there exists a unique continuation $z_n$ of $z$ by Implicity Function Theorem. We just need to show that $z_n$ has the same $\XXXX_n$-itinerary as the points in $\phi_{n,E}(M_E)$. On the contrary, assume that there exists $x\in M_E$ such that $\phi_{n,E}(x)$ and $z_n$ have different $\XXXX_n$-itineraries for all large $n$. As explained in the proof of Claim 1, one can assume that $\phi_{n,E}(x)$ and $z_n$ are separated by $\RRR_{n,i}$ for some $1\leq i\leq m$. Since $\phi_{n,E}(x)\to x,z_n\to z$ and $\RRR_{n,i}\to\RRR_i$, then either  $z=c_i$, or $\RRR_i$ separates $x,z$. Both of the cases lead to a contradiction: in the former case, this is because  $z$ is periodic but $c_i\in \JJJ_P$ is strictly preperiodic; and in the latter case, it contradict that $x,z$ belong to a common $\JJJJ$-end.

Without loss of generality, we assume that $E$ is $P$-invariant.  Let $z'$ be a Julia preperiodic point in $E$. Inductively, we will prove that Claim 2 holds for $z'$ under the assumption that Claim $2$ holds for $z:=P(z')$.

Let $z_n$ be the unique continuation of $z$ satisfying the requirements in Claim 2. If $z'$ is not a critical point of $P$, then $M_E$ and $z'$ are disjoint with $\cup_{i=1}^m\RRR_i$ and not separated by it. In this case, there is a unique preimage $z_n'$ of $z_n$ by $P_n$ such that $z_n'\to z'$. As $\cup_{i=1}^n\RRR_{n,i}\to\cup_{i=1}^n\RRR_i$, then $z_n'$ and $\phi_{n,E}(M_E)$ are not separated by $\cup_{i=1}^m\RRR_{n,i}$. Combining the assumption on $z_n$, then the point $z_n'$ have the same $\XXXX_n$-itinerary as $\phi_{n,E}(M_E)$.

Now let $c=z'$ be a Julia critical point of $P$, and denote $I(c)=\{i:1\leq i\leq m,c_i=c\}$. The plane $\C$ is divided into ${\rm deg}(P|_c)$ closed sets, called $\JJJJ(c)$-puzzle pieces, such that the points in each of the puzzle pieces are not separated by $\cup_{i\in I(c)}\RRR_i$.
\begin{figure}[http]
\includegraphics[scale=0.6]{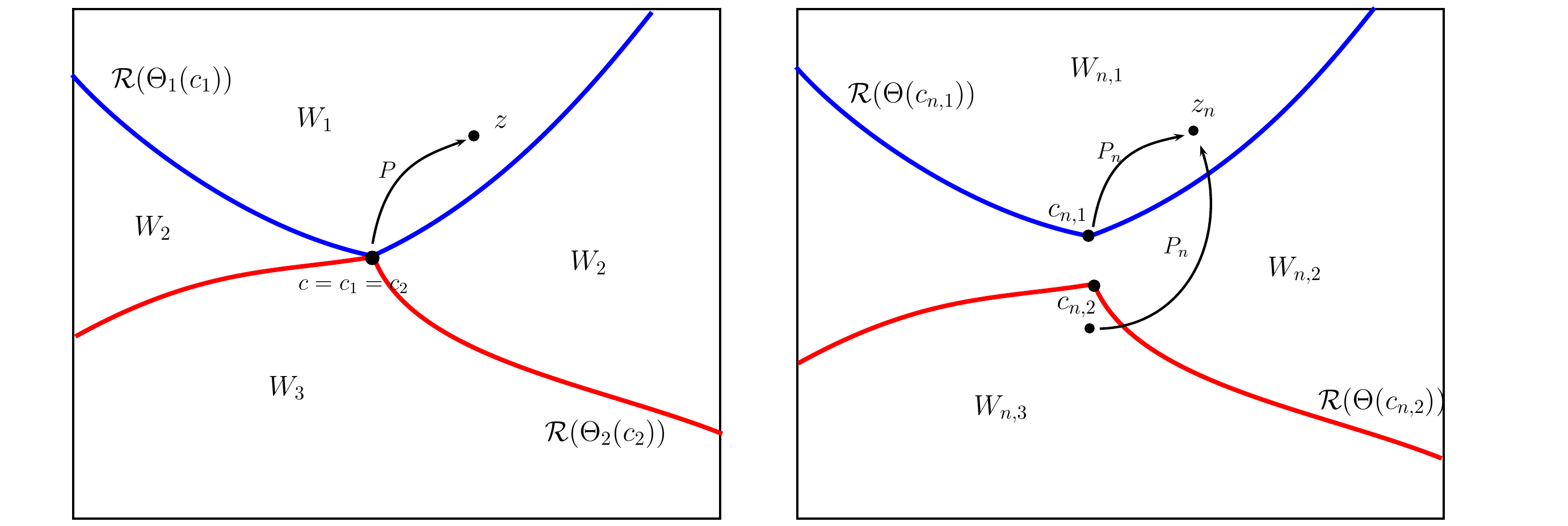}
\caption{$\JJJJ(c)$-puzzle pieces and $\XXXX_n(c)$-puzzle pieces. Here $I(c)=\{1,2\}$, ${\rm deg}(P|_c)=3$, $\Theta(c_{n,i})\to\Theta_i(c_i)$ for $i=1,2$. The $\JJJJ(c)$-puzzle pieces of $P$ are $\ov{W_1},\ov{W_2},\ov{W_3}$; and the $\XXXX_n(c)$-puzzle pieces of $P_n$ are $\ov{W_{n,1}},\ov{W_{n,2}},\ov{W_{n,3}}$.}
\end{figure}
Note that $z=P(c)$ has exactly one preimage by $P$ near $c$ in each  $\JJJJ(c)$-puzzle piece, and all these preimages coincide with $c$. In the dynamical plane of $P_n$, the plane $\C$ can be also divided into ${\rm deg}(P|_c)$ closed sets, called $\XXXX_n(c)$-puzzle pieces, such that the points in each of the puzzle pieces are not separated by $\cup_{i\in I(c)}\RRR_{n,i}$. Since $\phi_{n,E}(M_k)$ is not separated by $\cup_{i=1}^m\RRR_{n,i}$, there exists a unique $\XXXX_n(c)$-puzzle pieces $A_n$ with $\phi_{n,E}(M_k)\subseteq A_n$. Similarly, the point $z_n$ has exactly one preimage by $P_n$ near $c$ in each $\XXXX_n(c)$-puzzle piece, and we denote the one in $A_n$ by $z_n'$. Then $z_n'$ and $\phi_{n,E}(M_k)$ are not separated by $\cup_{i=1}^m\RRR_{n,i}$. Combining the assumption on $z_n$, the point $z_n'$ has the same $\XXXX_n$-itinerary as $\phi_{n,E}(M_E)$.
\hfill\qedsymbol

By Claim 2, we can extend $\phi_{n,E}$ to $V(H_E)$ such that $\phi_{n,E}(v):=v_n$ satisfy the properties in Claim 2 for all $v\in V(H_E)$. Then all points in $\phi_{n,E}(V(H_E))$ have a common $\XXXX_n$-itinerary.

\noindent\emph{Claim 3. The set $\phi_{n,E}(V(H_E))$ is contained in a component of $\KKK_{P_n}$ for large $n$.}

\noindent\emph{Proof of Claim 3.} We first observe that, for any distinct Fatou centers $x,y\in E$, their continuations $x_n,y_n$ are contained in the same component of $\KKK_{P_n}$. On the contrary, by taking subsequences if necessary, one can find an arc $\G_n$ consisting of two external radii and their common terminating point, such that $\G_n$ separates $x_n,y_n$, and $P_n^k:\G_n\to \RRR_{n,i}$ is injective for some $k\geq0$ and some $i\in\{1,\ldots,m\}$. By Lemma \ref{lem:convergence3}, the arc $\G_n$, by taking subsequences if necessary, converges to an arc $\G$ which consisting of two external rays of $P$ and their common landing point denoted by $z$, and the map $P^k:\G\to \RRR_i$ is injective. Since $z\not\in\{x,y\}$, the arc $\G$ separates $E$, and thus $\RRR_i$ separates $P^k(E)$, a contradiction.

This observation shows that $\phi_{n,E}(M_E)$ belong to a common component of $\KKK_{P_n}$, denoted as $K_{n,E}$. Let $z$ be a Julia vertex of $H_E$. If $z$ is on the boundary of a Fatou component in $E$, then $\phi_{n,E}(z)\in K_{n,E}(z)$ according to the observation above. We now assume that $z$ is not on the boundary of Fatou components in $E$ and $\phi_{n,E}(z)\not\in K_{n,E}$ for large $n$. Then by taking subsequences if necessary, one can find an arc $\G_n$ consisting of two external radii and their common terminating point, such that $\G_n$ separates $K_{n,E},\phi_{n,E}(z)$, and $P_n^k:\G_n\to \RRR_{n,i}$ is injective for some $k\geq0$ and some $i\in\{1,\ldots,m\}$. For simplicity, we assume $P(E)\subseteq E$.

Let $\Theta=\{\Theta_1,\ldots,\Theta_r\}$ be any weak critical marking of $P$. We denote by $\RRR(\Theta_i)$ the union of extended/external rays of $P$ with arguments in $\Theta_i$, and $\RRR(\Theta):=\cup_{i=1}^r\RRR(\Theta_i)$. By the assumption on $z$, one can choose a Fatou center $a\in E$ sufficiently near $z$, such that the pair $\{P^j(a),P^j(z)\}$ is not separated by $\RRR(\Theta)$ for any $j\in\{0,\ldots,k\}$ and any weak critical marking $\Theta$ of $P$. By the observation, we get $\phi_{n,E}(a)\in K_{n,E}$.

Recall \eqref{eq:99}, for every $i\in\{1,\ldots,s\}$, we denote $\RRR_n(U_{n,i})$ the union of extended rays of $P_n$ with arguments in $\Theta(U_{n,i})$. We claim that for each $0\leq j\leq k$, the points $\phi_{n,E}(P^j(a)),\phi_{n,E}(P^j(z))$ are not separated by $\cup_{i=1}^s\RRR_n(U_{n,i})$ for large $n$. If not, by taking subsequences, we have $\Theta(U_{n,i})\to\Theta_i$ for $1\leq i\leq s$. By Proposition \ref{pro:convergence4}, the collection $\{\Theta_1,\ldots,\Theta_s\}$ is a Fatou weak critical marking of $P$, and $\RRR_n(U_{n,i})\to \RRR(\Theta_i)$ for $i=1,\ldots,s$. Since $\phi_{n,E}(P^j(a)),\phi_{n,E}(P^j(z))$ is supposed to be separated by $\cup_{i=1}^s\RRR_n(U_{n,i})$ for infinite $n$, and $z$ is not on the boundary of Fatou components in $E$, it follows that $P^j(a),P^j(z)$ are separated by $\cup_{i=1}^s\RRR_n(\Theta_i)$, a contradiction to the choice of $a$.

By the claim above and the fact that $\phi_{n,E}(a),\phi_{n,E}(z)$ have the same $\XXXX_n$-itinerary (Claims 1, 2), the pairs of points $\{\phi_{n,E}(P^j(a)),\phi_{n,E}(P^j(z))\}$, $j=0,\ldots,k$, factually belong to the closure of a $\Theta_n$-unlinked class in the dynamical plane of $P_n$ . As $\G_n$ can not cross any $\RRR_n(U_{n,i})$ and $\RRR_{n,j}$ for $1\leq i\leq s,1\leq j\leq m$, then $\G_n$ belongs to the same equivalence class as $\phi_{n,E}(a),\phi_{n,E}(z)$ (because $\G_n$ separates them). Note that $P_n$ is injective on each $\Theta_n$-unlinked class, then $P_n(\G_n)$ separates $P_n\circ\phi_{n,E}(z)=\phi_{n,E}(P(z))$ and $P_n\circ \phi_{n,E}(a)=\phi_{n,E}(P(a))$. Inductively use the argument above, we get that $P_n^k(\G_n)=\RRR_{n,i}$ separates $\phi_{n,E}(P^k(z))$ and $\phi_{n,E}(P^k(a))$. It means that $a$ and $z$ have distinct $\XXXX_n$-itinerary, a contradiction.
\hfill\qedsymbol

 By Claim 3, we define the tree $H_E^n$ in the dynamical plane of $P_n$ as the convex hull of $\phi_{n,E}(V(H_E))$ within the component of $\KKK_{P_n}$ containing $\phi_{n,E}(V(H_E))$. The vertex set of $H_E^n$ is defined as the union of $\phi(V(H_E))$ and the branched points of $H_E^n$. The map $\phi_{n,E}:V(H_E)\to V(H_E^n)$ can be extended to a map, also denoted by $\phi_{n,E}$, from $H_E$ to $H_E^n$, such that it maps an edge $e(x,y)$ of $H_E$ to an segment $[\phi_{n,E}(x),\phi_{n,E}(y)]$ in $H_E^n$. Since the endpoints of $H_E^n$ are contained in $\phi_{n,E}(V(H_E))$, the map $\phi_{n,E}:H_E\to H_E^n$ is thus surjective.

\noindent\emph{Claim 4. Then map $\phi_{n,E}:H_E\to H_E^n$ is a homeomorphism.}

\noindent\emph{Proof of Claim 4.}
Since $\phi_{n,E}$ is surjective and injective on $V(H_E)$, it is enough to show that $\phi_{n,E}$ maps an edge of $H_E$ to an edge of $H_E^n$. Let $e=e(x,y)$ be an edge of $H_E$. By definition,  its image $\phi_{n,E}(e)$ is the segment $[\phi_{n,E}(x),\phi_{n,E}(y)]$. By taking subsequences if necessary, we only need to exclude the following two cases.

\noindent {\bf Case 1}. There exists $z\in V(H_E)$ such that $\phi_{n,E}(z)$ belongs to the open arc $(\phi_{n,E}(x),\phi_{n,E}(y))$ for large $n$.

We first suppose $z\in\FFF_P$, and let the preperiod of $\phi_{n,E}(z)$ is $k$. Let $p$ be sufficiently large integer. We can choose a preperiodic point of preperiod $k$ and period $p$ in each of the two components $\partial U(\phi_{n,E}(z))\setminus (\phi_{n,E}(x),\phi_{n,E}(y))$, denoted as $u_n,v_n$. By taking subsequences if necessary, we assume $u_n\to u,v_n\to v$. By Lemma \ref{lem:internal-ray}, the points $u,v$ are preperiod points of $P$ contained in $\partial U(z)$ with preperiod $k$ and period $p$. Since $p$ is large enough, the orbits of $u,v$ avoid the critical points of $P$. Hence $u_n,v_n$ are the continuation of $u,v$ respectively.

Let $\alpha\in{\rm arg}_P(u)$ and $\beta\in{\rm arg}_P(v)$. By Lemma \ref{lem:perturbation1}, the external rays $\RRR_{P_n}(\alpha)$ and $\RRR_{P_n}(\beta)$ land at $u_n$ and $v_n$ respectively for all large $n$. Let $\wh{\RRR}_{P_n}(\alpha),\wh{\RRR}_{P_n}(\beta)$ denote the corresponding extended rays associated to $U(\phi_{n,E}(z))$.  According to Lemmas \ref{lem:perturbation1} and \ref{lem:internal-ray}, the arcs $\G_n:=\wh{\RRR}_{P_n}(\alpha)\cup \wh{\RRR}_{P_n}(\beta)$ converge to the arc $\G:=\wh{\RRR}_{P}(\alpha)\cup \wh{\RRR}_{P}(\beta)$. Note that every $\G_n$ separates $x_n$ and $y_n$, then $\G$ separates $x$ and $y$. It implies $z$ belongs to the open arc $(x,y)$, a contradiction.

The situation of $z\in\JJJ_P$ is similar. By taking subsequences, there exists an arc $\G_n:=\RRR_{P_n}(\alpha)\cup\{\phi_{E,n}(z)\}\cup\RRR_{P_n}(\beta)$ separating $\phi_{E,n}(x)$ and $\phi_{E,n}(y)$ for all large $n$. By Lemma \ref{lem:convergence3}, the arc $\G_n$ converge to the arc $\G:=\RRR_{P}(\alpha)\cup\{z\}\cup\RRR_{P}(\beta)$, which separates $x,y$, also a contradiction.

\noindent {\bf Case 2}. There exists a point $z_n\in V(H_E^n)\setminus\phi_{E,n}(V(H_E))$ contained in the open arc $(\phi_{n,E}(x),\phi_{n,E}(y))$ for every large $n$.

By the definition of $V(H_E^n)$, the point $z_n$ is a branched point.
Suppose on the contrary that this case happens. Then, for each $n$, we can find three external rays/extended rays of $P_n$ with arguments $\alpha_n,\beta_n,\eta_n$ landing together at $z_n$, and a endpoint  $v_n$ of $H_E^n$ disjoint with $\phi_{n,E}(x),\phi_{n,E}(y)$, such that the three components of $\C\setminus \RRR_{P_n}(\alpha_n,\beta_n,\eta_n)$ contains $\phi_{n,E}(x),\phi_{n,E}(y)$ and $v_n$ respectively,
 where $\RRR_{P_n}(\alpha_n,\beta_n,\eta_n)$ denotes the union of the three rays with arguments $\alpha_n,\beta_n,\eta_n$ together with their common landing point $z_n$. By taking subsequences, we can assume $\alpha_n\to\alpha,\beta_n\to\beta,\eta_n\to\eta,z_n\to z$ and $v_n=\phi_{n,E}(v)$ for a point $v\in V(H_E)$ since the endpoints of $H_E^n$ belong to $\phi_{n,E}(V(H_E))$.

As explained in Case 1, one can suitably choose the rays in $\RRR_{P_n}(\alpha_n,\beta_n,\eta_n)$ such that $$\limsup_{n\to\infty}\RRR_{P_n}(\alpha_n,\beta_n,\eta_n)=\RRR_P(\alpha,\beta,\eta),$$
where $\RRR_P(\alpha,\beta,\eta)$ consists of external/extended rays of $P$ with arguments $\alpha,\beta,\eta$, together with their common landing point $z$.
If $z\not\in\{x,y,v\}$, then $\RRR_P(\alpha,\beta,\eta)$ separates $x,y,v$, and hence $z$ is a branched point of $H_E$ contained in $(x,y)$, a contradiction.
Otherwise, one of $\{x,y,v\}$, say $x$, coincides with $z$. In this case, the distance of $\phi_{n,E}(x)$ and $z_n$ converge to $0$, so that the length of the orbit of $z_n$ converge to $\infty$. Since $z_n$ is a branched point of $H_E^n$, it follows that the number of the branched points of $H_E^n$ goes to infinity, which contradicts that the number of the endpoints of $H_E^n$ is bounded by $\#V(H_E)$ for all $n$.
\hfill\qedsymbol

By this claim, we see that $V(H_E^n)=\phi_{n,E}(V(H_E))$, and the endpoints of $H_E^n$ are contained in $\phi_{n,E}(M_E)$. Hence $H_E^n$ is a subtree of $\HHH_{P_n}$.

\noindent\emph{Claim 5. Let $E,E'$ be two elements of $\EEE_\JJJJ$ such that $E'=P(E)$. Then $P_n(H_E^n)\subseteq H_{E'}^n$ and $P_n:H_E^n\to H_{E'}^n$ is Markov.}

\noindent\emph{Proof.} By the definition of $\phi_{n,E}$, we have $P_n\circ\phi_{n,E}=\phi_{n,E'}\circ P$ on $V(H_E)$. Since $P(V(H_E))\subseteq V(H_{E'})$, it follows that $$P_n(V(H_E^n))=P_n\circ\phi_{n,E}(V(H_E))=\phi_{n,E'}\circ P(V(H_E))\subseteq\phi_{E',n}(V(H_{E'}))=V(H_{E'}^n).$$
Note that the image $P_n(H_E^n)$ equals to the regulated hull of $P_n(V(H_E^n))$ in $\KKK_{P_n}$, then one get $P_n(H_E^n)\subseteq H_{E'}^n$. To prove that $P_n:H_E^n\to H_{E'}^n$ is Markov, we just need to check that the restriction of $P_n$ on each edge of $H_E^n$ is injective.

Let $e=e(\phi_{n,E}(x),\phi_{n,E}(y))$ be an edge of $H_E^n$ with endpoints $\phi_{n,E}(x),\phi_{n,E}(y)$. Observe that the restriction of $P_n$ on $e$ is not injective only if $e$ is separated by $\RRR_{n,i}$ for some $1\leq i\leq m$. In this case $\phi_{n,E}(x)$ and $\phi_{n,E}(y)$ are separated by $\RRR_{n,i}$. However, it is impossible because $\phi_{n,E}(x)$ and $\phi_{n,E}(y)$ have the same $\XXXX_n$-itinerary.
\hfill\qedsymbol

\noindent \emph{Claim 6. For any different elements $E_1,E_2$ of $\EEE_\JJJJ$, we have $H_{E_1}^n\cap H_{E_2}^n=\emptyset$  for all large $n$}.

\noindent\emph{Proof.}
By Lemma \ref{lem:end-property}.(3), there exists a minimal $k$ such that $P^k(E_1)$ and $P^k(E_2)$ are separated by $\RRR_i$ for some $1\leq i\leq m$. Without loss of generality, we assume that $k=0$. As $V(H_{E_j}^n)\to V(H_{E_j})$ for $j=1,2$ and $\cup_{i=1}^m\RRR_{n,i}\to\cup_{i=1}^m\RRR_i$, then trees $H_{E_1}^n$ and $H_{E_2}^n$ are separated by $\RRR_{n,i}$ for all large $n$. Consequently, if $H_{E_1}^n\cap H_{E_2}^n\not=\emptyset$, the intersection equals to $c_{n,i}$.
On the other hand, since $E_1,E_2$ are periodic, by Claim 5, there exists $p$ independent on $n$ such that $P_n^p(H_{E_j}^n)\subseteq H_{E_j}^n$ for $j=1,2$. It follows that $P_n^p(c_{n,i})$ belongs to $H_{E_1}^n\cap H_{E_2}^n$, and hence equal to $c_{n,i}$. This is impossible because $c_{n,i}$  is either non-periodic, or periodic with period converges to $\infty$.
\hfill\qedsymbol

Now we define $\HHH_\JJJJ^n$ as the union of $H_E^n$ with $E\in\EEE_\JJJJ$. Since each $H_E^n$ is a connected component of $\HHH_\JJJJ^n$ (by Claim 6),  the vertex set $V(\HHH_\JJJJ^n)$ is defined as the union of $V(H_E^n)$ for all $E\in\EEE_\JJJJ$.
By Claim $5$, we have $P_n(V(\HHH_\JJJJ^n))\subseteq V(\HHH_\JJJJ^n)$ the map $P_n:\HHH_\JJJJ^n\to \HHH_\JJJJ^n$ is Markov.
Note that the trees $H_E,E\in\EEE_\JJJJ$ are pairwise disjoint, we then obtain a homeomorphism $\phi_n:\HHH_\JJJJ\to \HHH^n_\JJJJ$ for large $n$ such that $\phi_n(z):=\phi_{n,E}(z)$ if $z\in H_E$ and $E\in\EEE_\JJJJ$.

To prove $h_{top}(P|_{\HHH_\JJJJ})\leq h(P_n)$, it is enough to show that $h_{top}(P|_{\HHH_\JJJJ})=h_{top}(P_n|_{\HHH_\JJJJ^n})$, since $\HHH_\JJJJ^n\subseteq \HHH_{P_n}$.
 Let $D$ be the incidence matrix of $(\HHH_\JJJJ,P)$. By the homeomorphism $\phi_n:\HHH_\JJJJ\to \HHH_\JJJJ^n$,  the enumeration of $E(\HHH_\JJJJ)$ induces an enumeration on $E(\HHH_\JJJJ^n)$ such that $e$ and $\phi_n(e)$ have the same label.
Let $e=e(x,y)$ be any edge of $\HHH_\JJJJ$ with endpoints $x,y\in V(\HHH_\JJJJ)$. Then $P(e)=[P(x),P(y)]$ consists of several edges $\HHH_\JJJJ$, denoted as  $e_1,\ldots,e_r$. The edge $\phi_n(e)$ of $\HHH_\JJJJ^n$ has the endpoints $\phi_n(x),\phi_n(y)$, and its image $P_n(\phi_n(e))$ under $P_n$ equals to $[P_n\circ\phi_n(x),P_n\circ\phi_n(y)]$. Since $\phi_n$ is a homeomorphism, the segment $P_n(\phi_n(e))$ consists of exactly the edges $\phi_n(e_1),\ldots,\phi_n(e_r)$ of  $\HHH_\JJJJ^n$. It means that $(\HHH_\JJJJ,P)$ and $(\HHH_\JJJJ^n,P_n)$ have the same incidence matrix, and hence the same topological entropy. We then complete the proof of \eqref{eq:11}.

To prove the equality, we just need to show that, if $\{P_n,n\geq1\}$ is maximal-hyperbolic (of type $\JJJJ$), then any non-trivial periodic component $H_n$ of $\HHH_{P_n}$ is contained in $\HHH_\JJJJ^n$ for large $n$.
Let $M_{n}$ denote the Fatou critical/postcritical points of $P_n$ contained in $H_n$. Then $H_n$ is the regulated convex hull of $M_n$ within $\KKK_{P_n}$.

In the dynamical plane of $P$, we can find a set $M$ such that each point in $M_n$ is a continuation of a point in $M$. To complete the proof the proposition, it is enough to show that $M$ belong to a $\JJJJ$-end. On the contrary, we postulate that $x,y\in M$ belong to different $\JJJJ$-ends. Then one can find a minimal $k$ and a $\JJJJ$-puzzle piece $B_k$ such that $x\in B_k$ and $y\not\in B_k$. It follows that there exist an arc $\G$ consisting of two external rays and their common landing point, and a index $1\leq i\leq m$ such that $\G$ separates $x,y$ and $P^k(\G)\subseteq \RRR_i$. Correspondingly, in the dynamical plane of $P_n$, we can find an arc $\G_n$ consisting of two external radii and their common terminating point such that $P^k_n(\G_n)\subseteq \RRR_{n,i}$ and $\G_n\to \G$ as $n\to\infty$. Since $x_n\to x,y_n\to y$, we then get that $\G_n$ separates $x_n,y_n$, contradicting that $x_n,y_n$ belongs to the same component of $\KKK_{P_n}$.

The remaining part of the proposition follows directly from Proposition \ref{pro:convergence4} because the space of critical portraits of degree $d$ is compact (under the Hausdorff metric).
\hfill\qedsymbol


\begin{proposition}\label{pro:liminf-strong}
Let $P$ be a monic, centered \pf polynomial. Then we can find a sequence $\{P_n,n\geq1\}$ of partial \pf polynomials converging to $P$ by perturbing $P$ with capture surgery such that $\mu_P=h(P_n)$ for large $n$, where $\mu_P$ is defined in Proposition \ref{pro:liminf}, and the number $\mu_P$ equals to the limit inferior $\liminf_{Q\to P} h(Q)$ with $Q$ chosen in $\PPP_d^{\rm ppf}$.
\end{proposition}

\begin{proof}
Let $\JJJJ$ be a weak Julia critical marking of $P$ such that $h(P|_{\HHH_\JJJJ})=\mu_P$. By Proposition \ref{pro:liminf}, we just need to construct a maximal-hyperbolic polynomial sequence converging to $P$ of type $\JJJJ$ by perturbing $P$ with capture surgery.

Denote $\JJJJ:=\{\Theta_1(c_1),\ldots,\Theta_m(c_m)\}$ as in Section \ref{sec:critical-marking}. Let $\c$ be the vector consists all Julia critical points of $P$, and $\vec{\s}_n=(\c,\DD_n,\ZZ_n,\XX_n),n\geq1$ be a sequence of perturbation data for capture surgery defined in Section \ref{sec:capture} such that
\begin{itemize}
\item ${\rm diam}(\DD_n):=\max_{c\in\c}{\rm diam}(D_n(c))\to 0$ as $n\to\infty$;
\item for each $n$ and $c\in\c$, the invariant domain $Z_n(c)\in\ZZ_n$ has one component in each component of $D_n(c)\setminus\KKK_P$.
\end{itemize}
To get an expected polynomial sequence by capture surgery, we need to specifically define the surgery mappings $\XX_n,n\geq1$.

Let $c$ be a Julia critical point of $P$, and $I_c:=\{i:1\leq i\leq m,c_i=c\}$ where the notations $c_i$ come from $\JJJJ=\{\Theta_1(c_1),\ldots,\Theta_m(c_m)\}$. For any index $i\in I_c$, we denote by $V_{n,i}$ the component of $P(D_{n,c})\setminus \KKK_P$ such that the external ray of $P$ with argument $\theta_i:=\tau(\Theta_i(c_i))$ lands at $P(c)$ through $V_{n,i}$, and denote $Y_{n,i}:=P(Z_n(c))\cap V_{n,i}$. We also denote by $v_{n,i}$ the intersection of $\RRR_P(\theta_i)$ and $\partial P(D_{n,c})$, and by $w_{n,i,j},j=1,\ldots,k_i$ the intersections of $\RRR_P(\alpha_{i,j})$ and $\partial D_{n,c}$ where $\Theta_i(c_i):=\{\alpha_{i,1},\ldots,\alpha_{i,k_i}\}$ (see Figure \ref{fig:capture-mapping}).
\begin{figure}[http]
 \includegraphics[scale=0.65]{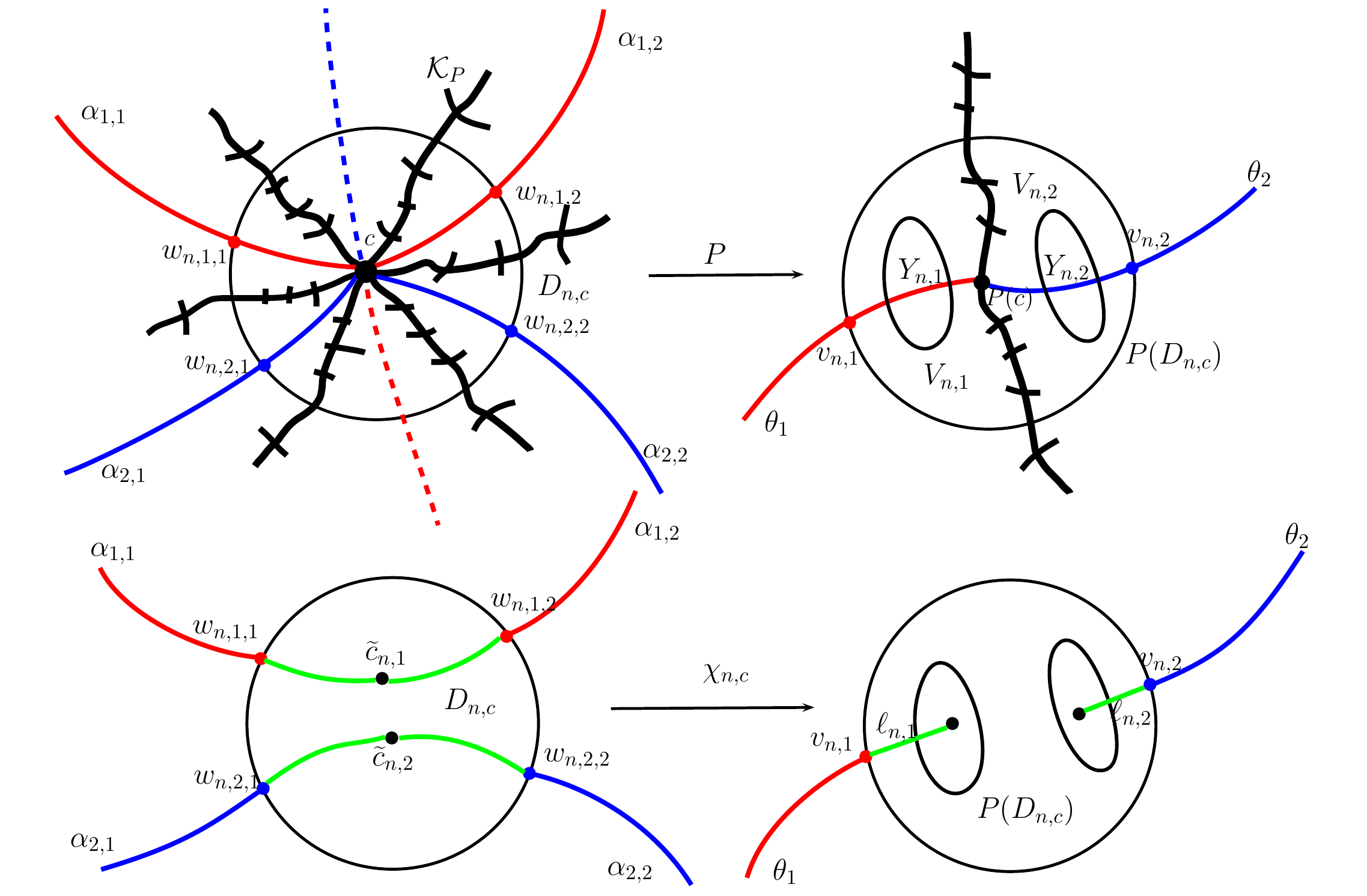}
\caption{The construction of perturbation mappings. The picture above shows the local behavior of $P$ near $c$ and the one below illustrates the perturbation of $P$ near $c$. Here $c=c_1=c_2$, $\Theta_1(c_1)=\{\alpha_{1,1},\alpha_{1,2}\}$ and $\Theta_2(c_2)=\{\alpha_{2,1},\alpha_{2,2}\}$.}
\label{fig:capture-mapping}
\end{figure}
Then all $w_{n,i,j},j=1,\ldots,k_i$ are preimages of $v_{n,i}$ by $\chi_{n,i}$ since $P=\chi_{n,c}$ on $\partial D_{n,c}$.
For the definition of $\chi_{n,c}$, we require that the critical point $c$ of $P$ splits into $\#I(c)$ critical points $\wt{c}_{n,i},i\in I(c),$ of the surgery map $\chi_{n,c}$, such that, for each $i$,
\begin{itemize}
\item the critical point $\wt{c}_{n,i}$ (of $\chi_{n,c}$) has the multiplicity $\#\Theta_i(c_i)-1$ and $\chi_{n,c}(\wt{c}_{n,i})\in Y_{n,i}$;
\item let $\ell_{n,i}$ be an arc in $V_{n,i}$ joining $\chi_{n,i}(\wt{c}_{n,i})$ and $v_{n,i}$, then the lift of $\ell_{n,i}$ by $\chi_{n,c}$ based at each $w_{n,i,j}$, with $j=1,\ldots,k_i$, has the terminal point $\wt{c}_{n,i}$.
\end{itemize}

Let $\wt{P}_{n}$ be the topological perturbation of $P$ by capture surgery with data $\vec{\s}_n$ specified above. By Lemma \ref{lem:no-Thurston-obstruction}, there exists a unique polynomial $P_n$ c-equivalent to $\wt{P}_{n}$ by a pair of normalized homeomorphisms $(\phi_{n,0},\phi_{n,1})$ such that they fix two  points near $\infty$ independent on $n$. This normalized property implies that $\phi_{n,0}(z)/z\to 1$ as $z\to\infty$, and hence $P_{n}$ is monic. For each Julia critical point $c$ of $P$ and any $i\in I(c)$, we denote $c_{n,i}:=\phi_{n,0}(\wt{c}_{n,i})$.
Then the Julia critical point $c$ of $P$ splits into $\#I(c)$ escaping critical points $c_{n,i},i\in I(c)$ of $P_n$. By Proposition \ref{pro:surgery-convergence}, it follows that the sequence $\{P_n,n\geq1\}$ is maximal-hyperbolic and converge to $P$.

To prove $h(P_n)=\mu_P$ for large $n$, we remains to show that the sequence $\{P_n,n\geq1\}$ has type $\JJJJ$. Let $c_{n,i}$ be a escaping critical point of $P_n$ with $i\in I(c)$, where $c$ is a Julia critical point of $P$. We define $\Theta(c_{n,i})$ as the set of the arguments of the external radii of $P_n$ which terminate at $c_{n,i}$. By taking subsequence if necessary, we assume that $\Theta(c_{n,i})\to \Theta_i$ as $n\to\infty$. According to Lemma \ref{lem:convergence3}, the external rays of $P$ with arguments in $\Theta_i$ land at $c$. Our specific construction of $\XX_n$ ensure that the set $\Theta_i$ is exactly $\Theta_i(c_i)$. Hence the Julia critical marking $\{\Theta(c_{n,i}):1\leq i\leq m\}$ of $P_n$ converge to $\JJJJ$ as $n\to\infty$.

Notice that the polynomials $P_n,n\geq1,$ constructed above by capture surgery are monic, but not necessary centered. To testify $\mu_P$ the limit inferior $\liminf_{Q\to P}h(Q)$ with $Q\in\PPP_d^{\rm ppf}$, we will find a sequence $\{Q_n,\geq1\}\subseteq \PPP_d^{\rm ppf}$ such that $h(Q_n)=h(P_n)$ and $Q_n\to P$.
Let $T_n$ be a the translation such that $Q_n:=T_n\circ P_n\circ T_n^{-1}$ is a monic, centered polynoimal. Then $h(P_n)=h(Q_n)$. By taking a subsequence, we can assume $T_n\to T$. Hence $Q_n,n\geq1,$ converge to a monic, centered polynomial $Q=T\circ P\circ T^{-1}$. Since $P$ itself is monic, centered, and $T$ is a translation, if follows that $T=id$, and then $P=Q$.
\end{proof}

\begin{proof}[Proof of Proposition \ref{pro:key}]
The upper semi-continuous of $h$ at $P$ is duo to Proposition \ref{pro:limsup}. Let $\JJJJ$ be a Julia weak critical marking of $P$ such that $\mu_P=h_{top}(P|_{\HHH_\JJJJ})$. We define $\HHH_P^*:=\HHH_\JJJJ$. Then the remaining part of Proposition \ref{pro:key} follows directly from Propositions \ref{pro:liminf} and \ref{pro:liminf-strong}.
\end{proof}

\section{The core entropy for \pf Newton maps}\label{sec:newton-graph}

In this section, we summarize the construction of \emph{extended Newton graphs} given in \cite{LMS1}, define the core entropy for \pf Newton maps, and give a formula to compute the core entropy. Throughout this subsection, the map $f$ denote a \pf Newton map of degree $d\geq3$.  

\subsection{Newton graphs and the induced puzzles}
Let $B_1,\ldots, B_d$ are the fixed Fatou components
of $f$, with centers $a_1,\ldots,a_d$, respectively. Let $\{(B_i,\Phi_i)_{1\leq i\leq d}\}$
be the B\"{o}ttcher coordinates of $B_1,\ldots,B_d$. For each $B_k$, the $d_k-1$ fixed internal
rays $$I_{B_k}(\frac{i}{d_k-1}),1\leq i\le d_k-1,$$
must land at fixed points in $\partial B_k$, which
can only possible be $\infty$. Thus these rays possesses a common landing point.
The union of all these $\sum_k(d_k-1)$
 fixed internal rays  together
with their landing point $\infty$, usually denoted by $\De_0$:
\[\De_0:=\bigcup_{k=1}^d\bigcup_{i=1}^{d_k-1}\ov{I_{B_k}\big(i/(d_k-1)\big)}.\]

 The set $\De_0$ is called the \emph{channel diagram} of $f$, following the notation in \cite{HSS}. Clear $f(\De_0)=\De_0$. For any $n\geq0$, denote by $\De_n$ the connected component of $f^{-n}(\De_0)$ that contains $\De_0$. Following \cite{LMS1}, we call $\De_n$  the \emph{Newton graph} of $f$ at level $n$. The vertex set $V(\De_n)$ of $\De_n$ consists of the points in $\De_n$ which are iterated by $f$ to the fixed points.

\begin{lemma}[\cite{MRS},Theorem 3.4]\label{lem:connected}
There exists $N\geq0$ such that the Newton graph  $\De_N$ contains all poles of $f$. Then $\De_{n+1}=f^{-1}(\De_n)$ and $\De_n\subseteq\De_{n+1}$ for any $n\geq N$.
\end{lemma}

By this lemma, the Newton graphs naturally induce a puzzle as follows. Let $N$ be the minimal number such that $\De_N$ contains all poles of $f$. For each $k\geq0$, let $\WWW_k$ denote the collection of all connected components of $\ov{\C}\setminus \De_{N+k}$. We call $\WWW=\cup_{k\geq0}\WWW_k$ the \emph{$\De$-puzzle of $f$,} or  \emph{$\De(f)$-puzzle}, and call each element of $\WWW_k$  a \emph{$\De$-puzzle piece (of $f$) at level $k$}. By Lemma \ref{lem:connected}, we know that
 \begin{enumerate}
 \item each $\De$-puzzle piece is a simply-connected domain;
 \item two distinct $\De$-puzzle pieces are either disjoint or nested;
 \item each $\De$-puzzle piece of level $k+1$ ($k\geq0$) is contained in a $\De$-puzzle piece of level $k$;
 \item each $\De$-puzzle piece of level $k+1$ ($k\geq0$) is a connected component of the preimage by $f$ of a $\De$-puzzle piece of level $k$.
\end{enumerate}



\subsection{Puzzle-renormalization of Newton maps} \label{sec:puzzle-renormalization}
We call $\rho=(f^p,W,W')$ a \emph{puzzle-renormalization  triple} (of period $p$) if the following properties hold:
\begin{enumerate}
\item $W\subseteq W'$ are distinct $\De$-puzzle pieces;
\item $f^i(W),0\leq i\leq p-1,$ are pairwise disjoint $\De$-puzzle pieces, $f^p(W)=W'$, and ${\rm deg}(f^p|_W)\geq 2$;
\item the \emph{filled-in Julia set} of $\rho$, defined as
\[K_\rho:=\{z\in W\mid f^n(z)\in W\text{ for all $n\geq0$}\}\]
is connected.
\end{enumerate}

Note that the definition of puzzle-renormalization is different from that of renormalization given in Definition \ref{def:renormalization}, because it is possible that either $\ov{W}\not\subseteq W'$ (although $W\subseteq W'$) or $W,W'$ are not disks (although they are simply-connected). However, the following result shows that puzzle-renormalization implies renormalization.


\begin{proposition}[\cite{LMS1}, Lemma 4.19]\label{pro:puzzle-renormalization}
Let $(f^p,W,W')$ be a puzzle-renormalization triple. Then there exists a pair of Jordan domain $U\subseteq V$, such that $(f^p,U,V)$ is a renormalization triple of $f$ in the sense of Definition \ref{def:renormalization}, and the filled-in Julia set of $(f^p,U,V)$ equals to that of $(f^p,W,W')$.
\end{proposition}

Notice that, by definition, the filled-in Julia set of any puzzle-renormalization  triple is disjoint from the Newton graphs $\De_n$ for all $n\geq0$. So the filled-in Julia sets of different puzzle-renormalization triples either coincides or are disjoint. On the other hand, let $\rho_0=(f^p,W,W')$ be a puzzle-renormalization triple of period $p$, such that the level of $W'$ is larger than $p$. Then, for each $0\leq i\leq p$, the triple $\rho_i:=(f^p,f^i(W),f^i(W'))$ is also a puzzle-renormalization triple with its filled-in Julia set $K_{\rho_i}$ equal to $f^i(K_{\rho_0})$.
Therefore, the distinct filled-in Julia sets of all puzzle-renormalization triples are pairwise disjoint, and can be divided into finitely many orbits under the iteration of $f$.

Let $\rho=(f^p,W,W')$ be any puzzle-renormalization triple. By Proposition \ref{pro:puzzle-renormalization}, we can define an \emph{canonical extend Hubbard tree} $H_\rho$ of $\rho$ as the regular convex hull within $K_\rho$ of all critical/postcritical points and fixed points of $f^p|_{K_\rho}$ (following \cite{LMS1}). It is clear that $f(H_{\rho_i})\subseteq H_{\rho_{i+1}}$ with $\rho_i:=f^i(\rho)$ for $i=0,\ldots,p-1$.

A critical/postcritical point of $f$ is called \emph{free} if it is not iterated by $f$ to the fixed points. So the free critical/postcritical points are disjoint with Newton graphs of any level. Let $c$ be a free critical/postcritical point of $f$. If there exists a puzzle-renormalization triple $\rho$ such that $f^k(c)\in H_\rho$ for some minimal $k\geq0$, let $H_c$ be the component of $f^{-k}(H_\rho)$ containing $c$; otherwise we define $H_c=\{c\}$. It is obvious that, for any two free critical/postcritical points $c,c'$, the trees $H_c$ and $H_{c'}$ either coincide or are disjoint.
\begin{definition}[canonical Hubbard forest]\label{def:Hubbard-forest}
We define $\HHH_f$ as the union of $H_c$ (constructed above) for all free critical/postcritical points of $f$,
and call it the \emph{canonical (puzzle-renormalization) Hubbard forest} of $f$.
\end{definition}
The vertex set $V(\HHH_f)$ of $\HHH_f$ consists of the endpoints and branched points of $\HHH_f$, together with the free critical/postcritical points of $f$.
Then the Hubbard forest $\HHH_f$ is $f$-invariant, and $f:\HHH_f\to\HHH_f$ is Markov. Note also that every non-trivial periodic component of $\HHH_f$ is the canonical extended Hubbard tree of a puzzle renormalization triple.

\begin{lemma}\label{lem:renormalization}
Let $\De_n$ be the  Newton graph of level $n\geq0$, and $H$ a non-trivial, periodic component of $\HHH_f$ with period $p$. Then there exists a renormalization triple $(f^p,U_H,V_H)$ (in the sense of Definition \ref{def:renormalization}) such that $V_H\cap \De_n=\emptyset$ and $H$ is an extended Hubbard tree of the renormalization triple.
\end{lemma}
\begin{proof}
It follows directly from Proposition \ref{pro:puzzle-renormalization} and the fact that the filled-in Julia set of any puzzle-renormalization triple of $f$ is disjoint with the Newton graph of any level.
\end{proof}

\subsection{The construction of extended Newton graphs} By now we have two disjoint $f$-invariant objectives: one is a Newton graph $\De_n$ with $\De_n$ containing all poles of $f$; and the other one is the canonical Hubbard forest $\HHH_f$ of $f$. To obtain a connected $f$-invariant graph, one need to find some invariant arcs joining the two objectives together. These invariant arcs are constructed in \cite{LMS1} as \emph{Newton rays}.

A continuous injective map $\g:[0,1)\to \ov{\C}$ is called a \emph{ray}. Let $\De_n$ ($n\geq0$) be any Newton graph of $f$. A ray $\g:[0,1)\to\ov{\C}$ is called a \emph{Newton ray} (with respect to $\De_n$) if $\g(0)\in\De_n$, $\g(0,1)\cap \De_n=\emptyset$  and $\g(t)$ is iterated to $\De_0$ by $f$ for any $t\in[0,1)$. In other words, a Newton ray is a ray with image in $(\cup_{m\geq0}\De_m)\setminus \De_n$ except  its starting point.  The sub-hyperbolicity of $f$ implies that any Newton ray converges to a unique point as $t$ tends to $1$. We say that $\g$ \emph{lands} on this point.

A Newton ray $\g$ with respect to $\De_n$ is said to be \emph{periodic} if there exists an integer $m\geq1$ such that $f^m(\g)=\g\cup\EEE$ with $\EEE$ a subset of $\De_n$. The smallest such $m$ is called the \emph{period} of $\g$.


\begin{lemma}[\cite{LMS1}, Lemma 4.14]\label{lem:separate}
For any sufficiently large $n$, the Newton graph $\De_n$ satisfies the following two properties:
\begin{enumerate}
\item all poles and the non-free critical/postcritical points of $f$ are contained in $\De_n$;
\item any two components of $\HHH_f$ lie in different complementary components of $\De_n$.
\end{enumerate}
\end{lemma}

Let $N=N_f$ be the minimal integer such that the Newton graph $\De_N$ satisfies the two properties in Lemma \ref{lem:separate}. Then $\De_N$ is called the \emph{canonical Newton graph} of $f$. Note that the canonical Newton graph and the canonical Hubbard forest are uniquely determined by $f$.

Roughly speaking, an extended Newton graph of $f$ consists of three parts: the canonical Newton graph $\De_f$, the canonical Hubbard forest $\HHH_f$, and finitely many preperiodic Newton rays which join each compnent of $\HHH_f$ to $\De_f$.
We do not intend to write the specific construction of an extended Newton graph here (one can refer to the proof of \cite[Theorems 6.2]{LMS1}). Instead, we list some properties of  extended Newton graphs that will be used below.

\begin{proposition}[\cite{LMS1}, Theorem 6.2]\label{pro:extended-graph}
Let $G=G_f$ be an extended Newton graph of a given Newton map $f$. Then the pair $(G,f)$ satisfies the following properties.
\begin{enumerate}
\item The graph $G$ is a finite connected graph consisting of the canonical Newton graph, the canonical Hubbard forest of $f$ and several preperiodic Newton rays. As a consequence, all critical/postcritical points of $f$ belong to $G$.
\item The vertex set of $G$ consists of the vertices of the canonical Newton graph and the vertices of the canonical Hubbard forest. In other words, each Newton ray in $G$ is an edge of $G$ with one endpoint in $V(\De_f)$ and the other one in $V(\HHH_f)$.
\item The edge set of $G$ consists of: the edges of the canonical Newton graph, the edges of the canonical Hubbard forest, and the Newton rays in $G$.
\item The graph $G$ is $f$-invariant and $f|_G:G\to G$ is a Markov map.
\end{enumerate}
\end{proposition}

According to the proposition, the only vague and possibly non-unique part of an extended Newton graph comes from  the Newton rays. In fact, the Newton rays satisfying the properties of Proposition \ref{pro:extended-graph} are not unique.

\subsection{The entropy formula for  \pf Newton maps}\label{sec:entropy-newton-map}

Let $f$ be a given \pf Newton map of degree $d\geq 3$, and $G$ be an extended Newton graph of $f$. Let $\De=\De_f$ be the canonical Newton graph of $f$, $\HHH=\HHH_f$ be the canonical Hubbard forest, and $\RRR=\RRR(G)$ the union of Newton rays contained in $G$. We know from Proposition \ref{pro:extended-graph} that $G$ is the disjoint union of $\De,\HHH,\RRR$ except the vertices, $f(\De)\subseteq\De$, $f(\RRR)\subseteq \De\cup\RRR$ and $f(\HHH)\subseteq\HHH$. Recall we stipulate in the paper that $h_{top}(f|_\emptyset):=0$.

\begin{lemma}\label{lem:well-defined}
The topological entropy $h_{top}(f|_G)$ equals to $h_{top}(f|_\HHH)$.
\end{lemma}
\begin{proof}
Note that $f(\HHH)\subseteq \HHH$ and  $f(\De\cup\RRR)\subseteq \De\cup \RRR$, by Propositions \ref{Do2}, \ref{Do3}, we have
\[h_{top}(f|_G)=\max\{h_{top}(f|_{\De\cup \RRR}),h_{top}(f|_\HHH)\}.\]
Let $D_{\De\cup\RRR}$ denote the incidence matrix of $(\De\cup\RRR,f)$. Since $\De$ is $f$-invariant, then the matrix $D_{\De\cup\RRR}$ has the form
$$\left(
\begin{array}{cc}
D_\RRR&O\\
*& D_\De
\end{array}
\right),$$
where $D_\De$ denotes the incidence matrix of $(\De,f)$ and $O$ denotes a zero matrix. As each Newton ray in $\RRR$ is preperiodic with respect to $\De$, and any two rays in $\RRR$ are disjoint except at their endpoints, then the leading eigenvalue of $D_\RRR$ is $1$. It follows that $h_{top}(f|_{\De\cup\RRR})=h_{top}(f|_\De)$.
Note also that $\De$ is finally iterated to the channel diagram $\De_0$ and $f:\De_0\to\De_0$ is a homeomorphism, then $h_{top}(f|_{\De\cup\RRR})=h_{top}(f|_\De)=h_{top}(f|_{\De_0})=0$.
\end{proof}


By this lemma, although the construction of extended Newton graphs $G$ is not unique, the topological entropy $h_{top}(f|_G)$ is independent on the choices of $G$ since $\HHH$ is uniquely determined by $f$.  Therefore, the following definition is well-defined.

\begin{definition}[core entropy of Newton map]
Let $f$ be a \pf Newton graph, and $G$ an extended Newton graph of $f$. The \emph{core entropy} $h(f)$ of $f$ is defined as the topological entropy of the restriction of $f$ on $G$, i.e., $h(f)=h_{top}(f|_G)$.
\end{definition}

In the following, we will develop some  precise formulas to compute the core entropy of Newton maps. A collection $\OOO=\{H_i,0\leq i\leq p-1\}$ of components of $\HHH_f$ is called a \emph{cycle in ${\rm Comp}(\HHH_P)$ of period $p$} if $p$ is the minimal number such that $f(H_i)\subseteq H_{i+1}$ for each $0\leq i\leq p-1$ and $H_p:=H_0$.
Let $\HHH_\OOO$ denote the union of $H_0,\ldots,H_{p-1}$.  Then, by Propositions \ref{Do2} and \ref{Do3}, we have the entropy formula
\begin{equation}\label{eq:entropy1}
h(f)=\max_{\OOO\subseteq {\rm Comp}(\HHH_f)}h_{top}(f|_{\HHH_\OOO}).
\end{equation}



\begin{lemma}\label{lem:orbit-entropy}
Let $\OOO=(H_0,\ldots,H_{p-1})$ be a cycle in ${\rm Comp}(\HHH_f)$ of period $p$. Then for any element $H$ of $\OOO$, we have $$h_{top}(f|_{\OOO})=h_{top}(f^p|_H)/p.$$
\end{lemma}
\begin{proof}
By Propositions \ref{Do1} and \ref{Do2}, we have $$h_{top}(f|_{\HHH_O})=\frac{1}{p}h_{top}(f^p|_{\HHH_\OOO})=\frac{1}{p}\max_{0\leq i\leq p-1}h_{top}(f^p|_{H_i}).$$
So it is enough to prove that $h_{top}(f^p|_{H_i})=h_{top}(f^p|_{H_{i+1}})$ for every $0\leq i\leq p-1$. Let $0\leq i\leq p-1$. We have the commutative diagram
$$\begin{array}{rcl}H_i &\xrightarrow[]{\ f^p\ }  &H_i\\ f\Big\downarrow &&\Big\downarrow f \vspace{-0.1cm} \\ H_{i+1} &  \xrightarrow[]{\ f^p\ } & H_{i+1}\vspace{-0.1cm}.\end{array}$$
Let $\wt{H}_i$ be the component of $f^{-1}(H_{i+1})$ containing $H_i$. Then the commutative graph
$$\begin{array}{rcl}\wt{H}_i &\xrightarrow[]{\ f^p\ }  &\wt{H}_i\\ f\Big\downarrow &&\Big\downarrow f \vspace{-0.1cm} \\ H_{i+1} &  \xrightarrow[]{\ f^p\ } & H_{i+1}\vspace{-0.1cm}\end{array}$$
holds. In this case $f:\wt{H}_i\to H_{i+1}$ is surjective and $\#f^{-1}(y)$ is uniformly bounded for all $y\in H_{i+1}$. According to Proposition \ref{Do4}, we have $h_{top}(f^p|_{\wt{H}_i})=h_{top}(f^p|_{H_{i+1}})$.

By Lemma \ref{lem:renormalization}, there exists a renormalization triple $(f^p,U_{H_i},V_{H_i})$ such that $H_i$, hence $\wt{H}_i$, are extended Hubbard trees of this triple. Then, following Straightening Theorem and Lemma \ref{lem:extended-tree}, the topological entropies $h_{top}(f^p|_{H_i})$ and $h_{top}(f|_{\wt{H}_i})$ are equal.
\end{proof}

For each periodic component $H$ of $\HHH_f$, we denote $p_H$ the period of $H$. In each cycle $\OOO$ in ${\rm Comp}(\HHH_f)$, we assign any element $H_\OOO$.  By Lemma \ref{lem:orbit-entropy}, we obtain two finer formulas for $h(f)$ as
\begin{eqnarray}
h(f)&=&
           \max\limits_{\mbox{
\tiny
$\begin{array}{c}
H\in{\rm Comp}(\HHH_f)\\
{\rm non-trivial}\\
{\rm periodic}\end{array}$
}
}\frac{1}{p_H}h_{top}(f^{p_H}|_H)\label{eq:entropy2}\\[5pt]
&=&
       \max\limits_{\mbox{
\tiny
$\begin{array}{c}
\OOO\subseteq {\rm Comp}(\HHH_f)\\
{\rm non-trivial\ cycle}
\end{array}$
}}\frac{1}{p_{H_\OOO}}h_{top}(f^{p_{H_\OOO}}|_{H_\OOO}).\label{eq:entropy3}
\end{eqnarray}


Let $H$ be a non-trivial periodic component of $\HHH_f$ of period $p$. By Lemma \ref{lem:renormalization} and Straightening Theorem, there exist a \pf monic, centered polynomial $P_H$ of degree $\de:={\rm deg}(f^p|_{K_\rho})$ such that $f^p:U_H\to V_H$ is hybrid conjugate to $P_H$. Such polynomial $P_H$ is unique up to conjugation by $\de-1$ roots of unit, and called a \emph{renormalization polynomial} of $f$ (associated to $H$). We denote by
\[\mathfrak{R}(f):=\{P_H:\text{ $H$ is a non-trivial periodic component of $\HHH_f$}\}.\]
The map $f$ is called \emph{(puzzle-)renormalizalbe} if $\mathfrak{R}(f)\not=\emptyset$, and \emph{non-renormalizable} otherwise.

\begin{proof}[Proof of Proposition \ref{pro:entropy-formula}]
It follows directly from Lemma \ref{lem:well-defined}, equation \eqref{eq:entropy2} and the fact of $h(P_H)=h_{top}(f^p|_{H})$ for any non-trivial periodic component of $\HHH_f$.
\end{proof}

\subsection{Capture surgery for \pf Newton maps}\label{sec:capture-N}
Let $f$ be a \pf Newton map, and $\vec{\s}=(\c,\DD,\ZZ,\XX)$ a perturbation data for capture surgery on $f$ defined in Section \ref{sec:capture}. To obtain an expected perturbation, we restrict two further properties on $\vec{\s}$:
\begin{itemize}
\item the critical points in $\c$ are not iterated to $\infty$, i.e., all critical points in $\c$ belong to $\HHH_f$;
\item the critical values of each $\chi_c$ in $\XX$ are the iterated preimages of the fixed critical points of $f$.
\end{itemize}
The first property means that we only perturb some free critical points of $f$, and the second one ensures that the perturbation map $F_{\vec{\s}}$ is a ``topological \pf Newton map''.

\begin{lemma}\label{lem:rational-realization}
Let $F=F_{\vec{\s}}$ be the topological perturbation of $f$ by capture surgery with data $\vec{\s}$. Then $F$ is c-equivalent to a \pf Newton map.
\end{lemma}
\begin{proof}
By the second property above and Proposition \ref{pro:head}, any rational map c-equivalent to $F$ must be a \pf Newton maps up to conformal conjugation. So, by Theorem \ref{thm:cui-tan-Thurston}, we just need to check that the marked map $(F,\QQQ)$ has no Thurston obstructions for some marked set $\QQQ$ of $F$.

Let $\De_N$ be the Newton graph of level $N$ such that all critical/postcritical points of $f$ not in $\HHH_f$ belong to $\De_N$, and different components of $\HHH_f$ are contained in different components of $\C\setminus \De_N$. Then $F=f$ outside the components of $\C\setminus\De_N$ at which the surgery happens. We define the marked set $\QQQ$ as the union of  the vertex set of $\De_N$ and ${\rm Post}_F$.

On the contrary, let $\wt{\G}$ be a Thursotn obstruction of $(F,\QQQ)$, then there exists a sub-multicurve $\G\subseteq \wt{\G}$, such that
\begin{itemize}
\item the leading eigenvalue $\lambda(\G)$ of the transition matrix of $\G$ is at least $1$;
\item for each curve $\g\in\G$, there exists a component in $F^{-1}(\G)$ homotopic to $\g$ in $\ov{\C}\setminus\QQQ$;
\item for each curve  $\g\in\G$, there exists a component in $F^{-1}(\g)$ which is homotopic to an element of $\G$ in $\ov{\C}\setminus\QQQ$.
\end{itemize}

 We assume that for each $\g\in\G$, $\#(\g\cap\De_N)$ is minimal in its homotopic class in $\ov{\C}\setminus\QQQ$.
For each $0\leq i\leq N$, let $\G_i$ consists of all curves in $\G$ which intersect $\De_i$, but not intersect $\De_{i-1},\ldots, \De_0$, and let $\G_\infty$ consists the curves in $\G$ disjoint with $\De_N$. Then the multicurve $\G$ equals to the disjoint union of $\G_0,\ldots,\G_N,\G_\infty$.

Note that for each $1\leq i\leq N$, any component in $F^{-1}(\G_i)$ can not be homotopic in $\ov{\C}\setminus \QQQ$ to a curve in $\G_0,\ldots,\G_i$, and any component in $F^{-1}(\G_\infty)$ can not be homotopic in $\ov{\C}\setminus \QQQ$ to a curve in $\cup_{i=0}^N\G_i$ (by the minimality of $\#(\G\cap \De_N)$). Then the incidence matrix of $\G$ has the lower-triangle form
\[\left(\begin{array}{ccccc}
D_0&&&&\\
*&O&&&\\
\vdots&\ddots&\ddots&&\\
*&\cdots&*&O\\
*&\cdots&\cdots&*&D_{\infty}
\end{array}\right),
\]
where $D_0,D_\infty$ denote the incidence matrix of $\G_0,\G_\infty$ respectively, and $O$ denotes the zero metric. To deduce a contradiction, we will show that $\lambda_F(D_0),\lambda_F(D_\infty)<1$.

We first suppose $\lambda_F(D_\infty)\ge 1$. Since the curves in $\G_\infty$ are disjoint with $\De_N$, for each non-trivial component $H$ of $\HHH_f$, we denote by $\G_H$ the curves in $\G_\infty$ belonging to the components of $\C\setminus \De_N$ that contains $H$. Since $\lambda_F(D_\infty)\geq1$, then there exists a non-trivial periodic component $H$ of $\HHH_f$ with period $p$, such that $\lambda_{F^p}(\G_H)\geq1$. If the surgery does not happens in the orbit of $H$, then $F^p$ induces a polynomial-like map in a neighborhood of $H$, and hence $\lambda_{F^p}(\G_H)<1$, a contradiction. Otherwise, the map $F^p$ induces a topological polynomial-like map in a neighborhood of $H$, which is exactly a topological perturbation by capture surgery of a polynomial-like map induced by $f^p$ near $H$. By Lemma \ref{lem:no-Thurston-obstruction}, we have $\lambda_{F^p}(\G_H)<1$, also a contradiction.

We then assume that $\lambda_F(D_0)\geq1$. Keep in mind that $\#(\g\cap\De_0)$ is minimal for each $\g\in\G_0$. Since $F:\De_0\to\De_0$ is bijective, the number $0<\#(\g\cap \De_0)<\infty$ is a constant for $\g\in\G_0$, and then there exists at most one component of $F^{-1}(\g)$ homotopic in $\ov{\C}\setminus \QQQ$ to a curve in $\G_0$ for each $\g\in\G_0$. Thus for each $\g\in\G_0$, there is one curve $\beta\in\G$ such that $F^{-1}(\beta)$ has one component homotopic in $\ov{\C}\setminus \QQQ$ to $\g$ (combining properties (2),(3) of Thurston obstruction). Therefore each entry of the transition matrix $D_0$ is less
than or equal to 1. Because $\lambda(\G_0)\geq 1$, there is a Levy cycle in $\G_0$.
For simplicity of the statement, we assume that $\G_0=\{\g\}$ is the Levy cycle, and $\wt{\g}$ denote the component of $F^{-1}(\g)$ homotopic to $\g$ in $\ov{\C}\setminus\QQQ$.

By the minimality of $\#(\g\cap \De_0)$, we have $\#(\wt{\g}\cap \De_0)\ge \#(\g\cap\De_0)$. On the other hand, since $F:\wt{\g}\to\g$ is a homeomorphism and $F(\De_0)=\De_0$, then $\#\wt{\g}\cap\De_0\leq \#\g\cap\De_0$ and $F(\wt{\g}\cap\De_0)\subseteq \g\cap \De_0$. Combining these two aspects, we get that $F:\wt{\g}\cap\De_0\to\g\cap \De_0$ is a homeomorphism. To deduce a contradiction, we only need to show that $\wt{\g}$ intersects $\De_1\setminus \De_0$ (contradicting that $F:\wt{\g}\to\g$ is a homeomorphism).

Observe that $\g$ must intersects the boundaries of the components of $\ov{\C}\setminus\De_N$ intersecting $\HHH_f$: otherwise $\g$ is a Levy cycle of the marked map $(f,V(\De_N)\cup{\rm Post}_f)$ since $(F,\QQQ)=(f,V(\De_N)\cup{\rm Post}_f)$ outside these components, a contradiction. It implies that $\g$  intersects  $\De_j\setminus \De_{j-1}$ for some $1\leq j\leq N$. By the minimality of $\#(\g\cap \De_N)$, the curve $\wt{\g}$ also intersects $\De_j\setminus\De_{j-1}$, and hence $\g=F(\wt{\g})$ intersects  $\De_{j-1}\setminus \De_{j-2}$. Inductively, we have  $\wt{\g}$  intersects $\De_1\setminus \De_0$, a contradiction.
\end{proof}

\section{Continuity of the entropy function:\pf Newton family}\label{sec:continuity2}


A \pf Newton map is called \emph{generic} if $\infty\not\in{\rm Post}_f$.
We will study the continuity of the entropy function within the \pf Newton family at generic parameters. Throughout this section, we always assume that $f$ is a generic \pf Newton map, and $\{f_n,n\geq1\}$ a sequence of \pf Newton map converging to $f$.

\subsection{Perturbation of Newton graphs}\label{sec:perturbation-puzzle}
The reason we require $f$ generic is the following perturbation lemma.

\begin{lemma}[perturbation of Newton graphs]\label{lem:perturbation-newton-graph}
 Fix any $m\geq0$. Let $\De_m$ and $\De_m^n$ denote the Newton graphs of level $m$ associated to $f$ and $f_n$ respectively. Then $\De_m$ is homeomorphic to $\De_m^n$ for all sufficiently large $n$ and $\limsup_{n\to\infty}\De_m^n=\De_m$.
\end{lemma}

\begin{proof}
Fix any $m\geq0$. We first establish an injection $\xi_n:\De_m\to\De_m^n$ for all large $n$ such that $\xi_n(\De_m)\to \De_m$ as $n\to\infty$.

Note that the vertices of $\De_m$ are iterated preimages of fixed points of $f$. Let $v\in V(\De_m)$. If $v\in\FFF_f$, it is a center of a Fatou component $U$ of $f$. By Lemma \ref{lem:Fatou-component} $v$ has a unique continuation $v_n$ at $f_n$ such that $v_n$ is the center of $U_n$, where $U_n$ is the deformation of $U$ at $f_n$. In this case, we define $\xi_n(v)=:v_n$. If $v\in\JJJ_f$, since its orbit is supposed to avoid ${\rm Crit}_f$, then $v$ has a unique continuation $v_n$ at $f_n$, and we define $\xi_n(v)=:v_n$. Thus we obtain an injection $\xi_n:V(\De_m)\to \C$ such that $\lim_{n\to\infty}\xi_n(V(\De_m))=V(\De_m)$. Since $\xi_n(v)$ is the unique continuation of $v$ for each vertex $v$ of $\De_m$, we have $f_n\circ \xi_n=\xi_n\circ f$ on $V(\De_m)$. It follows that $\xi_n(V(\De_m))\subseteq V(\De_m^n)$ for all large $n$.

We begin to extend the map $\xi_n$ to $\De_m$. Let $e$ be an edge of $\De_m$ with endpoints $x,y$. Then one point of $x$ and $y$, say $x$, is the center of a Fatou component $U$ of $f$, $y\in \partial U$ and $e$ is the internal ray in $U$ landing at $y$. We extend $\xi_n$ homeomorphically on $e$ such that its image is the internal ray in $U_n$ landing at $\xi_n(y)$ where $U_n$ denotes the deformation of $U$ at $f_n$. Clearly, the map $\xi_n:\De_m\to \C$ is injective, and by Lemma \ref{lem:internal-ray}, we have $\limsup_{n\to\infty}\xi_n(\De_m)=\De_m$. Notice that the edges of $\xi_n(\De_m)$ are all internal rays of $f_n$. Then, according to the formula $f_n\circ \xi_n=\xi_n\circ f$ on $V(\De_m)$, each edge of $\xi_n(\De_m)$ is mapped by $f^m_n$ to an fixed internal ray in some fixed Fatou component of $f_n$. It means that $\xi_n(\De_m)\subseteq \De_m^n$.

To complete the proof of the lemma, we remain to show $\xi_n(\De_m)=\De_m^n$. On the contrary, by taking subsequences, we assume that there exists a vertex $v$ of $\De_m$ and an vertex $u_n\in V(\De_m^n)\setminus \xi_n(\De_m)$ such that $\xi_n(v)$ and $u_n$ are the endpoints of an edge $e_n$ of $\De_m^n$. By Lemma \ref{lem:internal-ray}, the edge $e_n$ converge to an internal ray $e$ of $f$ which joins $v$ and $u=\lim_{n\to\infty} u_n$. Since $u_n\in V(\De_m^n)$, then $f_n^m(u_n)$ is a fixed point of $f_n$. It follows that $f^m(u)$ is a fixed point of $f$. Combining the fact that $\De_m\cup e$ is connected, we get  $u\in V(\De_m)$, and hence $u_n\in\xi_n(V(\De_m))$, a contradiction to the choice of $u_n$.
\end{proof}

As a consequence, we see that for any $\De(f)$-puzzle piece $W$ of level $k\geq0$, there exists a unique $\De(f_n)$-puzzle piece $W_n$ of level $k$ such that $\ov{W_n}\to\ov{W}$ as $n\to\infty$.

\begin{definition}[deformation of puzzle pieces]\label{def:puzzle}
Under the notations above, the $\De(f_n)$-puzzle piece $W_n$ is called the \emph{deformation of $W$ at $f_n$}.
\end{definition}

\subsection{Continuity of entropy function in non-renormalizable case}
In this subsection, we prove that the entropy function is continuous at generic, non-renormalizable parameters.

\begin{lemma}\label{lem:non-renormalizable}
Let $T_n$ be a non-trivial periodic component of the Hubbard forest $\HHH_{f_n}$  for all large $n$. Assume that a critical $c_n$ of $f_n$ belong to $T_n$ and $\lim_{n\to\infty}c_n=c$. Then, if $c$ does not belong to a periodic component of $\HHH_f$,  the period $p_n$ of $T_n$ goes to infinity.
\end{lemma}

\begin{proof}
On the contrary, assume that $p_n$ is uniformly bounded above. Then there exists an integer $M$ such that $f_n^M(T_n)\subseteq T_n$ for all $n\geq0$.
Set $b_n:=f_n^M(c_n)\in T_n$ and $b:=\lim_{n\to\infty}b_n$. Then $f^M(c)=b$. By increasing $M$ if necessary, we assume that $b$ is periodic. Hence $b\not=c$ since $c$ is not period.

Let $\De_N,\De_N^n$ be the Newton graphs of $f$ and $f_n$ respectively, with level $N$, such that all critical/postcritical points of $f$ not in $\HHH_f$ belong to $\De_N$, and different components of $\HHH_f$ lie in different components of $\C\setminus \De_N$. If $c\in \De_N$
 the point $b$ must be a fixed point of $f$.
 Since $f$ is generic, then $b\not=\infty$ and hence a  fixed critical point of $f$.  In this case, the point $b_n$ is also a fixed critical point of $f_n$ by Lemma \ref{lem:Fatou-component}.  It contradicts that $b_n\in T_n\subseteq \HHH_{f_n}$. In the case of  $c\in\HHH_f$, the points $b,c$ belong to different components of $\HHH_f$, since $c$ is not in a periodic component of $\HHH_f$ (by the assumption of the lemma) but $b$ is (because $b$ is period). Then the graph $\De_N$ separates $b,c$. It follows from Lemma \ref{lem:perturbation-newton-graph} that the graphs $\De_N^n$ also separate $b,c$ for all large $n$. On the other hand, note that $T_n$ contains $b_n$ and $c_n$ which converges to $b$ and $c$ respectively, then $T_n$ intersects $\De_N^n$ for all large $n$. It also contradicts to $T_n\subseteq \HHH_{f_n}$ (since $\HHH_{f_n}\cap \De_N^n=\emptyset$).
\end{proof}

\begin{proposition}\label{pro:non-renormalizable}
Let $f$ be a generic, non-renormalizable \pf Newton map. Then the entropy function $h:\NNN_d^{\rm pf}\to\R$ is continuous at $f$.
\end{proposition}
\begin{proof}
Let $\{f_n,n\geq1\}$ be an arbitrary sequence of \pf Newton map converging to $f$. Since $h(f)=0$ (by Proposition \ref{pro:entropy-formula}), we just need to show $\lim_{n\to\infty}h(f_n)=0$.
 Without loss of generality, we can assume that all $f_n$ are renormalizable.

We first claim that, for any non-trivial periodic component $T_n$ of $\HHH_{f_n}$ with period $p_n$, the topological entropy $h_{top}(f_n^{p_n}|_{T_n})$ is bounded by $(2d-2)\log d$. To see this point, let $\rho_n=(f_n^{p_n},U_n,V_n)$ be a renormalization triple such that $T_n$ is an extended Hubbard tree of $\rho_n$ (Lemma \ref{lem:renormalization}). Note that the orbit of the filled-in Julia set $K_{\rho_n}$ covers the critical points of $f_n$ at most $2d-2$ times, and the degree of $f_n|_{f_n^k(K_{\rho_n})}$ is at most $d$ ($k\geq0$), then the degree of $f_n^{p_n}|_{K_{\rho_n}}$ is bounded by $d^{2d-2}$. As a consequence, the topological entropy $h_{top}(f_n^{p_n}|_{T_n})$ is bounded by $(2d-2)\log d$.

Let $T_n$ be any non-trivial periodic component of the  $\HHH_{f_n}$, such that a critical point $c_n$ of $f_n$ belong to $T_n$ and $c_n\to c$ as $n\to\infty$. Since $f$ is postulated non-renormalizable, then $c$ is not in a periodic component of $\HHH_f$ (in this case, any periodic component of $\HHH_f$ is a repelling periodic point). By Lemma \ref{lem:non-renormalizable}, the period $p_n$ of $T_n$ converge to $\infty$. Combining the claim above, we have $h_{top}(f_n^{p_n}|_{T_n})/p_n\to 0$ as $n\to\infty$.
On the other hand, according to formula
(\ref{eq:entropy3}), the core entropy $h(f_n)$ equals to the maximum of $h_{top}(f_n^{p_n}|_{T_n})/p_n$ with $T_n$ going though all periodic components of $\HHH_{f_n}$ containing critical points. Hence $h(f_n)\to 0$ as $n\to \infty$.
\end{proof}

\subsection{Continuity of the entropy function in renormalizable case.}
In this part, we always assume that $f$ is renormalizable, i.e., $\HHH_f$ contains non-trivial components.


 By Lemma \ref{lem:separate}, we can choose, for each non-trivial periodic component $H$ of $\HHH_f$ (with period $p$), a pair of $\De(f)$-puzzle pieces $(W_H,W_H')$ containing $H$ such that
$(f^p,W_H,W_H')$ is a puzzle-renormalization triple, and the $\De(f)$-puzzle pieces $\{W_H':H\in{\rm Comp}(\HHH_f), \text{periodic and non-trivial}\}$ are pairwise disjoint.

In the dynamical plane of $f_n$, for each non-trivial periodic component $H$ of $\HHH_f$, we denote by $W_{H,n},W_{H,n}'$ the $\De(f_n)$-puzzle piece deformed from $W_H,W_H'$ respectively (see Definition \ref{def:puzzle}). Then the pairs of $\De(f_n)$-puzzle pieces $(W_{H,n},W_{H,n}')$ satisfy the following properties:
\begin{enumerate}
\item the $\De(f_n)$-puzzle pieces $\{W_{H,n}':H\in{\rm Comp}(\HHH_f), \text{periodic and non-trivial}\}$ are pairwise disjoint;
\item for each non-trivial, periodic component $H$ of $\HHH_f$ with period $p$, we have $W_{H,n}\subseteq W_{H,n}'$ and the map $f_n^p:W_{H,n}\to W_{H,n}'$ is a branched covering of degree equal to ${\rm deg}(f^p|_{W_H})$.
\end{enumerate}

To compare the core entropy $h(f)$ and the limits of $h(f_n)$ as $n\to\infty$, we divide the canonical Hubbard forest $\HHH_{f_n}$  into two parts: let $\HHH^n_{\rm bd}$ denote the union of the components of $\HHH_{f_n}$ which stay in $\cup_H W_{H,n}$ under the iteration of $f_n$; and  $\HHH_{\rm esc}^n$ denote the union of the periodic components of $\HHH_{f_n}$ not contained in $\HHH^n_{\rm bd}$, where the subscripts ``bd'' and ``esc'' mean ``bounded'' and ``escape'' respectively. It is clear that both $\HHH^n_{\rm bd}$ and $\HHH^n_{\rm esc}$ are $f_n$-invariant, and by Propositions \ref{Do2} and \ref{Do3}, we have
\begin{equation}\label{eq:33}
h(f_n)=\max\{\ h_{top}(f_n|_{\HHH_{\rm bd}^n}),h_{top}(f_n|_{\HHH_{\rm esc}^n})\ \}.
\end{equation}


\begin{lemma}\label{lem:infinite-period}
Let $p_n$ denote the minimal period among the non-trivial components of $\HHH_{\rm esc}^n$, then $p_n\to\infty$ as $n\to\infty$. As a consequence, we have $h_{top}(f_n|_{\HHH_{\rm esc}^n})\to 0$ as $n\to\infty$.
\end{lemma}
\begin{proof}
On the contrary, by taking subsequences if necessary, we assume that all $p_n$ have a uniform upper bound. Without loss of generality, we can assume $p_n=p$ for all $n$. Let $T_n$ be a component of $\HHH_{\rm esc}^n$ with $f^p(T_n)\subseteq T_n$. Note that at least one of $T_n, f(T_n),\ldots,f^{p-1}(T_n)$ contains a critical point of $f_n$. So, without loss of generality and by taking subsequences if necessary, we can assume that $T_n$ contains a critical point $c_n$ of $f_n$ and $c_n$ converge to a critical point $c$ of $f$.

By Lemma \ref{lem:non-renormalizable}, the point $c$ must belong to a non-trivial periodic component $H$ of $\HHH_f$. Then $c\in W_H$. Let $\wt{W}_H\subseteq W_H$
be a $\De(f)$-puzzle piece containing $H$ such that level difference between $\wt{W}_H$ and $W_H$ is sufficiently large.
It follows that, for each $0\leq i\leq p$, the $\De(f)$-puzzle piece $f^i(\wt{W}_H)$ belongs to $W_{H_i}$, where $H_i$ is the component of $\HHH_f$ which contains $f^i(H)$.
Let $\wt{W}_{H,n}$ denote the deformation of $\wt{W}_H$ for large $n$. Since $c_n\to c$ and $T_n\cap\partial \wt{W}_{H,n}=\emptyset$,  the tree $T_n$ is contained in $\wt{W}_{H,n}$ for all large $n$. Moreover, the $\De(f_n)$-puzzle piece $f_n^i(\wt{W}_{H,n})$ belongs to $W_{H_i,n}$ for each $0\leq i\leq p$ and all large $n$. It implies that the entire orbit of $T_n$ stay in $\cup_HW_{H,n}$, and hence $T_n\subseteq \HHH_{\rm bd}^n$, a contradiction. So we prove that $\lim_{n\to\infty}p_n=\infty$.

By the same argument as that in the proof of Lemma \ref{lem:orbit-entropy}, we have
$$h_{top}(f_n|_{\HHH_{\rm esc}^n})=\max_{T_n\in{\rm Comp}(\HHH_{\rm esc}^n)}\frac{1}{p_{_{T_n}}}h_{top}(f^{p_{_{T_n}}}_n|_{T_n})\leq \frac{1}{p_n}\max_{T_n\in{\rm Comp}(\HHH_{\rm esc}^n)}h_{top}(f^{p_{_{T_n}}}_n|_{T_n}),$$
where $p_{_{T_n}}$ denote the period of $T_n$. Since $h_{top}(f^{p_{_{T_n}}}_n|_{T_n})\leq (2d-2)\log d$ (see the claim in the proof Proposition \ref{pro:non-renormalizable}) and $p_n\to\infty$, then $\lim_{n\to\infty}h_{top}(f_n|_{\HHH_{\rm esc}^n})\to 0$.
\end{proof}


For any non-trivial, periodic component $H$ of $\HHH_f$, we define a forest
\begin{equation}\label{eq:nearby}
H_n:=\HHH_{\rm bd}^n\cap W_{H,n}.
\end{equation}
It is easy to see that $f_n(H_n)\subseteq H_n'$ if and only if $f(H)\subseteq H'$ where $H'$ is a non-trivial periodic component of $\HHH_f$.
By a similar argument as previous, we have an entropy formula
\begin{equation}\label{eq:44}
h_{top}(f_n|_{\HHH_{\rm bd}^n})=\max\limits_{\mbox{
\tiny
$\begin{array}{c}
H\in{\rm Comp}(\HHH_f)\\
{\rm non-trivial}\\
{\rm period}\ p_H\end{array}$
}
}
\frac{1}{p_H}h_{top}(f_n^{p_H}|_{H_n}).
\end{equation}

\begin{lemma}\label{lem:entropy-f_n}
The limit/limit superior/limit inferior of $h(f_n)$ is equal to that of $\max_H \frac{1}{p_H}h_{top}(f_n^{p_H}|_{H_n})$ with $H$ going through all non-trivial periodic components of $\HHH_f$, as $n\to\infty$.
\end{lemma}
\begin{proof}
It follows directly from formulas \eqref{eq:33}, \eqref{eq:44}, and Lemma \ref{lem:infinite-period}.
\end{proof}

Let $H$ be a non-trivial periodic component of $\HHH_f$ with period $p$. By Lemma \ref{lem:renormalization}, there exists a renormalization triple $(f^p,U_H,V_H)$ such that $\ov{V_H}\subseteq W_H$, and $f^p:U_H\to V_H$ has the same filled-in Julia set as $f^p:W_H\to W_H'$. Therefore, by shrinking $V_H$, we can assume that $\ov{f^i(U_H)}\subseteq W_{H_i}$ for $0\leq i\leq p$, where $H_i$ denotes the component of $\HHH_f$ such that $f^i(H)\subseteq H_i$. As $f_n\to f$, then in the dynamical plane of $f_n$ for large $n$, there exists a polynomial-like map $f_n^p:U_{H,n}\to V_{H,n}$ converging to $f^p:U_H\to V_H$ in the sense that $\ov{U_{H,n}}\to \ov{U_H}$ and $\ov{V_{H,n}}\to \ov{V_H}$.  Notice also that $W_{H,n}\to W_H$ for each $H$, we then get a property for all large $n$:
\begin{equation}\label{eq:property}
\ov{f^i_n(U_{H,n})}\subseteq W_{H_i,n} \text{ for each } 0\leq i\leq p.
\end{equation}

\begin{lemma}\label{lem:filled-Julia-set}
Let $K_{H,n}$ denote the filled-in Julia set of the polynomial map $f^p_n:U_{H,n}\to V_{H,n}$. Then a point belongs to $K_{H,n}$ if and only if this point belongs to $W_{H,n}$ and its orbit stay in $\cup_H W_{H,n}$. As a consequence, we have that $H_n:=\HHH_{\rm bd}\cap W_{H,n}$ is an extended Hubbard forest of $(f^p_n,U_{H,n},V_{H,n})$ (for large $n$).
\end{lemma}
\begin{proof}
If $x\in K_{H,n}$, by property \eqref{eq:property}, we have $x\in W_{H,n}$ and its orbit belongs to $\cup_{i=0}^{p-1}W_{H_i,n}$. To prove the other side, since $f^p:W_H\to W_H'$ and $f^p:U_H\to V_H$ have the same filled-in Julia set, denoted as $K_H$, we can choose a $\De(f)$-puzzle piece $\wt{W}_H$ containing $K_H$, such that for each $0\leq i\leq p$, the closure of $f^i(\wt{W}_H)$ is contained in $U_{H_i}$. Let $\wt{W}_{H,n}$ denote the $\De(f_n)$-puzzle piece deformed from $\wt{W}_H$. As $\wt{W}_{H,n}\to \wt{W}_H $ and $U_{H_i,n}\to U_{H_i} (0\leq i\leq p-1)$, it follows that
 $$f_n^i(\wt{W}_{H,n})\subseteq U_{H_i,n} \text{ for each } 0\leq i\leq p-1.\qquad\qquad(*)$$
Let $x\in W_{H,n}$ such that its orbit stay in $\cup_{i=0}^{p-1}(W_{H_i,n})$. Since the points in $W_{H,n}\setminus\wt{W}_{H,n}$ will eventually leave $W_{H,n}$ under the iteration of $f_n^p$, then $x\in \wt{W}_{H,n}$. It follows from property $(*)$ that the point $x$ belongs to $K_{H,n}$.

By the result we just proved and the definition of $H_n$, we have that $H_n$ is an $f_n^p$-invariant forest in $K_{H,n}$ and  contains all bounded critical points of $f_n^p|_{U_{H,n}}$, hence an extended Hubbard forest of $(f_n^p,U_{H,n},V_{H,n})$.
\end{proof}

Now, let $H$ be a non-trivial periodic component of $\HHH_f$, and $(f^p,U_H,V_H)$ the renormalization triple near $H$ given before Lemma \ref{lem:filled-Julia-set}. Recall that a monic, centered \pf polynomial $P_H$ hybrid equivalent to $(f^p,U_H,V_H)$ is called a renormalization polynomial of $f$ associated to $H$. By the discussion above, we also get a polynomial-like map $(f_n^p,U_{H,n},V_{H,n})$ for every large $n$ such that $(f_n^p,U_{H,n},V_{H,n})$ converges to $(f^p,U_H,V_H)$.
\begin{definition}[perturbation of $P_H$]\label{def:nearby}
A monic, centered polynomial $P_{H,n}$ hybrid equivalent to $f^p_n:U_{H,n}\to V_{H,n}$ is called a \emph{straightening-perturbation polynomial of $P_H$ at $f_n$}.
 \end{definition}
Observe that the orbit of every bounded critical point of $(f_n^p,U_{H,n},V_{H,n})$ is finite (since $f_n$ is postcritically-finite) and $H_n$ is an extended Hubbard forest of $(f_n^p,U_{H,n},V_{H,n})$ (Lemma \ref{lem:filled-Julia-set}), then the polynomial $P_{H,n}$ is partial postcritically-finite, and $\varphi_{H,n}(H_n)$ is an extended Hubbard forest of $P_{H,n}$, where $\varphi_n:V_{H,n}\to \C$ is a quasi-conformal map realizing the hybrid equivalent between $(f_n^p,U_{H,n},V_{H,n})$ and $P_{H,n}$. By a generalized version of Lemma \ref{lem:extended-tree}, we have $h_{top}(f_n^p|_{H_n})=h(P_{H,n})$. The following fact will be repeatedly used.

\begin{lemma}\label{lem:equal-entropy}
Let $H$ be a non-trivial periodic component of $\HHH_f$ with period $p$, and $H_n$ the forest defined in \eqref{eq:nearby}. Then we have $h_{top}(f^p|_H)=h_{top}(P_H)$ and $h_{top}(f_n^p|_{H_n})=h(P_{H,n})$.
\end{lemma}



\begin{lemma}\label{lem:straighten}
Let $H$ be a non-trivial periodic component of $\HHH$, with $P_H$ a renormalization polynomial of $f$ associated to $H$ (Definition \ref{def:renormalization}), and $P_{H,n}$ a straightening-perturbation of $P_H$ at $f_n$ (Definition \ref{def:nearby}). By suitably choice of $P_{H,n}$, we have $P_{H,n}\to P_H$ as $n\to\infty$.
\end{lemma}
\begin{proof}
Let $\varphi:V_H\to\C,\varphi_n:V_{H,n}\to\C$ be two quasi-conformal maps, by which the polynomial-like maps $(f^p:U_H,V_H),(f_n^p:U_{H,n},V_{H,n})$ are hybrid equivalent to $P_H,P_{H,n}$ respectively.
We are able to extend the polynomial-like maps $f_n^p:U_{H,n}\to V_{H,n}$ and $f^p:U_H\to V_H$ to topological polynomials $F_n:\C\to\C$ and $F:\C\to\C$ respectively, and extend the quasi-conformal maps $\varphi_n$ and $\varphi$ to quasi-conformal maps on $\C$, also denoted by $\varphi_n$ and $\varphi$, such that
\begin{itemize}
\item $F_n$ and $F$ are holomorphic and coincide in a neighborhood of $\infty$;
\item $\varphi_n$ and $\varphi$ are conformal in a neighborhood of $\infty$ and normalized by fixing two given points in $U_H$ (and also $\infty$);
\item $\varphi_n\circ F_n=P_{H,n}\circ \varphi_n$ and $\varphi\circ F=P_{H}\circ \varphi$ on $\C$.
\end{itemize}
The convergence of $(f^p_n,U_{H,n},V_{H,n})$ to $(f^p,U_H,V_H)$ ensures that one can construct $F_n,F,\varphi_n$ such that $F_n$ uniform converge to $F$ and all $\varphi_n$ are $K$-quasi-conformal for some constant $K>1$. Hence the collection of maps $\{\varphi_n,n\geq1\}$ is a normal family.

Let $\varphi_{n_k},k\geq1,$ be an arbitrary convergent subsequence in this family with the limit map $\psi$. Then the map $\psi$ is $K$-quasiconformal on $\C$, conformal near $\infty$, and the polynomials $P_{H,n_k}=\varphi_{n_k}\circ F_n\circ \varphi^{-1}_{n_k}$ local-uniformly converge to a monic, centered polynomial $P:=\psi\circ F\circ \psi^{-1}$ on $\C$. Note that $P_{H}=\varphi\circ F\circ \varphi^{-1}$. So $P$ and $P_{H}$ are quasi-conformal conjugate on $\C$, and conformal conjugate near $\infty$. According to Theorem \ref{thm:cui-tan-Thurston}, the monic, centered polynomials $P$ and $P_H$ are conjugate under a $(\de-1)$-th root of unit $u$ with $\de:={\rm deg}(P_H)$. Therefore, the monic, centered polynomials $\wt{P}_{H,n_k}:=u\circ P_{H,n_k}\circ u^{-1},n\geq1,$ converge to $P_H$. Since $\{\varphi_{n_k},k\geq1\}$ is arbitrary, then the conclusion of the lemma holds
\end{proof}

\begin{definition}[entropy-minimal sequence]\label{def:entropy-minimal}
Let $g$ be a partial \pf polynomial (resp. \pf Newton map), and $\{g_n,n\geq1\}$ a sequence of partial \pf polynomials (resp. \pf Newton maps) converging to $g$. We say that the sequence is \emph{entropy-minimal} if $\lim_{n\to\infty}h(g_n)=\liminf_{\wt{g}\to g}h(g)$, where $\wt{g}$ are chosen in the partial \pf polynomial family (resp. \pf Newton-map family).
\end{definition}

\begin{proposition}\label{pro:liminf-newton}
Let $f$ be a  \pf Newton map, and $H_1,\ldots,H_k$ a collection of non-trivial periodic components of $\HHH_f$ in pairwise distinct cycles. Let $P_{H_i},i=1,\ldots,k$, be a renormalization polynomial of $f$ associated to $H_i$ (Definition \ref{def:renormalization}). Then we can construct a sequence of \pf Newton map $\{f_n,n\geq1\}$ by perturbation of $f$ with capture surgery, such that $f_n\to f$ and the polynomials $\{P_{H_i,n},n\geq1\}$ (defined in Definition \ref{def:nearby}) is a entropy-minimal sequence converging to $P_{H_i}$ for each $1\leq i\leq k$.
\end{proposition}
\begin{proof}
Note that  different cycles in ${\rm Comp}(\HHH_f)$ are pairwise disjoint and the capture surgery only happens in sufficiently small neighborhood of each cycle, then, without loss of generality, we can assume that the surgery is carried out at one non-trivial cycle in ${\rm Comp}(\HHH_f)$. Let $H$ be any element in this cycle of period $p$. We just need to construct a sequence of \pf Newton maps $\{f_n,n\geq1\}$ converging to $f$ such that the strengthening-perturbation polynomials $\{P_{H,n},n\geq1\}$ of $P_H$ is an entropy-minimal sequence converging to $P_H$.

As previous, let $\rho=(f^p,W_H,W_H')$ be a puzzle renormalization triple of $f$ with $H$ an extended Hubbard tree of $\rho$, and $(f,U_H,V_H)$  a renormalization triple such that $\ov{V}\subseteq W_H$ and $(f^p,U_H,V_H)$ has the same filled in Julia set as $\rho$ (by Lemma \ref{lem:renormalization}). Denote by $P_H$ a renormalization polynomial of $f$ associated to $H$. Let $\varphi:V_H\to \C$ be a quasi-conformal map by which $f:U_H\to V_H$ is hybrid equivalent to $P_H$. We denote $U:=\varphi(U_H)$ and $V:=\varphi(V_H)$. Then $P_H:U\to V$ is a polynomial-like map with its filled-in Julia set equal to $\KKK_{P_H}$.

For each $n$, we construct a rational perturbation $f_n$ of $f$ by capture surgery as follows. Let $\{\wt{Q}_{H,n},n\geq1\}$ be a sequence of topological perturbation of $P_H$ by capture surgery, such that
\begin{enumerate}
\item the surgery domain is compactly contained in $U$;
\item each escaping critical value of $\wt{Q}_{H,n}$ is the image by $\varphi$ of an iterated preimage of a fixed critical point of $f$;
\item the rational realization $\{Q_{H,n},n\geq1\}$ of $\{\wt{Q}_{H,n},n\geq1\}$ is an entropy-minimal sequence converging to $P_H$ (by Proposition \ref{pro:liminf-strong}).
\end{enumerate}
Each capture surgery above on $P_H$  induces a topological perturbation $F_n$ of $f$ by capture surgery near the orbit $H,\ldots,f^{p-1}(H)$, such that $F^p_n(z)= \varphi^{-1}\circ\wt{Q}_{H,n}\circ\varphi(z)$ for all $z\in U_H$.
By Lemma \ref{lem:rational-realization} and Proposition \ref{pro:surgery-convergence}, each $F_n$ is c-equivalent to a \pf Newton map $f_n$ such that the sequence $\{f_n,n\geq1\}$ converges to $f$.

For each $n$, let $(\phi_{n,o},\phi_{n,1})$ be a pair of normalized homeomorphisms by which $F_n$ is c-equivalent to $f_n$. According to the homotopy lifting lemma, we get a sequence of homeomorphism $\{\phi_{n,k},k\geq1\}$ such that $\phi_{n,k}\circ F_n= f_n\circ\phi_{n,k+1}$ on $\ov{\C}$ for each $k\geq0$.
Assume $W_H$ has level $k$ as a $\De(f)$ puzzle piece. We can choose $\phi_{n,0}$ in its isotopic class so that $\phi_{n,0}:\De_0(f)\to\De_0(f_n)$ is a homeomorphism and $\phi_{n,0}\circ f(=\phi_{n,0}\circ F_n)=f_n\circ \phi_{n,0}$ on $\De_0(f)$. Since $f$ coincides with $F_n$ outside $U_H,\ldots,f^{p-1}(U_H)$, it follows that $\phi_{n,k}:\De_k(f)\to \De_k(f_n)$ is a homeomorphism and $\phi_{n,k}\circ f=f_n\circ \phi_{n,k}$ on $\De_k(f)$. So we can assume the initial $\phi_{n,0}$ has this property.

Let $W_{H,n}$ and $W_{H,n}'$ be the image of $W_H,W_H'$ by $\phi_{n,0}$. Then $(f_n^p,W_{H,n},W_{H,n}')$ is the deformation of $(f^p,W_H,W_H')$ at $f_n$ (in the sense of Definition \ref{def:puzzle}). As $\ov{V_H}\subseteq W_H$, there exists polynomial-like maps $f_n^p:U_{H,n}\to V_{H,n}$ which converge to $f^p:U_H\to V_H$ and satisfy $\ov{V_{H,n}}\subseteq W_{H,n}$. Each excepted polynomial $P_{H,n}$ is hybrid equivalent to $f_n^p:U_{H,n}\to V_{H,n}$ by $\varphi_n$ (see the following commutative graph).
\begin{equation}
\xymatrix{
  \C\ar[d]_{P_{H,n}}  & \ar[l]_{\varphi_n} U_{H,n}  \ar[d]_{f_n^p}\ar[r]^{\hookrightarrow}& W_{H,n}\ar[d]_{f_n^p}&\ar[l]_{\phi_{n,p}} W_H\ar[d]_{F_n^p}&\ar[l]_{\hookleftarrow} U_H \ar[d]_{F_n^p} \ar[r]^{\varphi} & \C \ar[d]_{\wt{Q}_{H,n}} \ar[r]^{\psi_{n,1}} & \C \ar[d]^{Q_{H,n}} \\
  \C  & \ar[l]_{\varphi_n} V_{H,n} \ar[r]^{\hookrightarrow}& W_{H,n}'&\ar[l]_{\phi_{n,0}} W_H'&\ar[l]_{\hookleftarrow}V_H \ar[r]^{\varphi} & \C \ar[r]^{\psi_{n,0}} & \C   }
\end{equation}
where $(\psi_{n,0},\psi_{n,1})$ is a pair of homeomorphisms by which $\wt{Q}_{H,n}$ is c-equivalent to $Q_{H,n}$.

Let $T_n$ be the Hubbard forest of $Q_{H,n}$ and set $H_n:=\varphi_n\circ\phi_{n,0}\circ\varphi^{-1}\circ \psi_{n,0}^{-1}(T_n)$. Using Lemma \ref{lem:filled-Julia-set}, we can see  the points in $H_n$ do not escape under the iteration of $P_{H,n}$, hence $H_n\subseteq \KKK_{P_{H,n}}$. By revising $\phi_{n,0}$ its isotopic class, we can assume $H_n$ is regulated. Clearly, $H_n$ contains all bounded critical/postcritical points of $P_{H,n}$, then $H_n$ is the Hubbard forest of $P_{H,n}$. The combinatorial equivalence between $P_{H,n}:H_n\to H_n$ and $Q_{H,n}:T_n\to T_n$ implies $h(P_{H,n})=h(Q_{H,n})$. Since the sequence $\{Q_{H,n},n\geq1\}$ is entropy-minimal, so is $\{P_{H,n},n\geq1\}$.
\end{proof}

\begin{proof}[Proof of Theorem \ref{thm:main2}]
Let $\{f_n,n\geq1\}$ be an arbitrary sequence in $\NNN_d^{\rm pf}$ converging to $f$. We first claim that
\[\limsup_{n\to\infty}h_{top}(f^{p_H}_n|_{H_n})\leq h_{top}(f^{p_H}|_H)\]
for any non-trivial periodic component $H$ of $\HHH_f$, and $H_n$ is defined in \eqref{eq:nearby}. By Lemma \ref{lem:equal-entropy}, the inequality above is equivalent to
$$\limsup_{n\to\infty}h(P_{H,n})\leq h(P_H),$$
where $P_H$ is a renormalization polynomial of $f$ associated to $H$ (Definition \ref{def:renormalization}), and $P_{H,n}$ a straightening-perturbation of $P_H$ at $f_n$ (Definition \ref{def:nearby}).
Now, this inequality follows directly from Lemma \ref{lem:straighten} and Proposition \ref{pro:key}. Then the upper semi-continuous of $h$ at $f$ is duo to the claim above, formulas \eqref{eq:entropy2},  and  Lemma \ref{lem:entropy-f_n}.


The entropy function is continuous at non-renormalizable parameters duo to
 Proposition \ref{pro:non-renormalizable}.  So we only need to deal with the renormalizable case.

Suppose first that $h$ is continuous at $f$. According to Proposition \ref{pro:liminf-newton},  we can find a sequence of \pf Newton maps $\{f_n,n\geq1\}$ by perturbing $f$ with capture surgery  such that $f_n\to f$, and the associated polynomials $\{P_{H,n},n\geq1\}$ (Definition \ref{def:nearby}) is an entropy-minimal sequence converging to $P_H$  for every non-trivial periodic component $H$ of $\HHH_f$ with $h(f)=h_{top}(f^{p_H}|_H)/p_H$. By the claim in first paragraph and Lemma \ref{lem:entropy-f_n}, by taking subsequences if necessary, the continuity of $h$ at $f$ implies that there exists such a component $H$ so that $\lim_{n\to\infty}h_{top}(f_n^{p_H}|_{H_n})=h_{top}(f^{p_H}|_H)$. Combining Lemma \ref{lem:equal-entropy}, we have $\lim_{n\to\infty}h(P_{H,n})=h(P_H)$. Note that $\{P_{H,n},n\geq1\}$ is an entropy-minimal sequence converging to $P_H$, then the entropy function $h:\PPP_\de^{\rm ppf}\to \R$ is continuous at $P_H$ duo to Proposition \ref{pro:key}, where $\de:={\rm deg}(P_H)$.

On the other hand, suppose that $H$ is a non-trivial periodic component of $\HHH_f$ with period $p$, such that $h(f)=h_{top}(f^p|_H)/p$ and the entropy function on $\PPP_\de^{\rm ppf}$ is continuous at the renormalization polynomial $P_H$ associated to $H$, where $\de:={\rm deg}(P_H)$. Let $\{f_n,n\geq1\}$ be an entropy-minimal sequence in $\NNN_d^{\rm pf}$ converging to $f$.
By Lemma \ref{lem:straighten}, there exists a sequence $\{P_{H,n},n\geq1\}$ of straightening-perturbation polynomials  of $P_H$ at $f_n$ converging to $P_H$. It follows that
\[h(f)=h_{top}(f^p|_H)/p=h(P_H)/p=\lim_{n\to\infty}h(P_{H,n})/p=\lim_{n\to\infty}h_{top}(f_n^p|_{H_n})/p\leq \liminf_{n\to\infty}h(f_n),\]
where the third equality is duo to the continuity of $h$ at $P_H$, the second and fourth equality follows form Lemma \ref{lem:equal-entropy}, and the last inequality is by formula \eqref{eq:entropy2}. Remember that the sequence $\{f_n,n\geq1\}$ is entropy-minimal, then $h$ is lower semi-continuous at $f$, and hence continuous since $h$ is proved upper semi-continuous.
\end{proof}

\begin{proof}[Proof of Theorem \ref{thm:main3}]
Let $f$ be a \pf cubic Newton map. We first show the ``if'' part. The conclusion is obvious if $f$ is hyperbolic; and follows directly from Theorem \ref{thm:main2} when $f$ is generic and non-renormalizable.  So it remains to show that $h$ is continuous at $f$ in the non-generic, non-renormalizable case, for which we need to use some particular results for cubic Newton maps given in \cite{Ro}.

Let $\{f_n,\geq1\}$ be any sequence of \pf cubic Newton maps converging to $f$. Since $h(f)=0$ (by the non-renormalizable property), one can assume that all $f_n$ are renormalizable. Denote by $p_n$ the period of the unique periodic component $T_n$ of $\HHH_{f_n}$ which contains a critical point of $f_n$. Since $h(f_n)=h_{top}(f^{p_n}_n|_{T_n})/p_n$ and $h_{top}(f^{p_n}_n|_{T_n})\leq \log2$, we only need to prove $\lim_{n\to\infty}p_n=\infty$.

On the contrary, suppose all $p_n$ have a uniform upper bound. Let $c_n$ be the unique (free) critical point of $f_n$ contained in $T_n$, with $c_n\to c$ as $n\to\infty$. By the assumption, there exist a integer $p$ such that $f_n^p(T_n)\subseteq T_n$ for all large $n$. Let $b_n:=f_n^p(c_n)$. Then $f^p(c)=b$. By increasing $p$ if necessary, we can assume that $b$ is period. Since $f$ is postulated non-generic, then $b=\infty$.

By Lemmas 3.14, 3.15 and Corollary 3.16 in \cite{Ro}, there exists a Jordan curve $\g$ consisting of the closure of finitely many preperiodic internal rays of $f$ which are iterated to the fixed Fatou components of $f$, such that $\g$ separates $c$ and $b=\infty$, and the periods of the rays in $\g$ are sufficiently large. Note that the landing point of the rays in $\g$ are not pre-critical. Then, by a similar argument as that in Lemma \ref{lem:perturbation-newton-graph}, we can find a Jordan curve $\g_n$ consisting of the closure of internal rays of $f_n$ perturbed from the internal rays of $f$ contained in $\g$. As $\g_n\to\g$, the curve $\g_n$ also separates $b=\infty,c$, and hence separates $b_n,c_n$ for all large $n$. It follows that the tree $T_n$ intersects $\g_n$ for every large $n$. Since $T_n$ is disjoint with the iterated preimages of fixed Fatou components of $f_n$, the intersection of $T_n$ and $\g_n$ are landing points of the internal rays in $\g_n$. Notice that all internal rays in $\g_n$ are iterated to fixed Fatou components, then $T_n$ intersects the boundary of a fixed Fatou component $U_n$ of $f_n$ at a preperiodic point $x_n$. It follows that $f_n^p(x_n)$ also belongs to $T_n\cap \partial U_n$ and distinct with $x_n$ by the large period of $x_n$. It contradicts \cite[Lemma 6.3]{Ro}: which says that the filled-in Julia set of any renormalization triple of a cubic Newton map intersects at most one point with the closure of a fixed Fatou component.

For the ``only if'' part, we need to show that if $f$ is non-hyperbolic and renormalizable, then $h$ is not continuous at $f$. Note that in this case, the map $f$ must be generic, since its unique free critical point belongs to the filled-in Julia set of a renormalization triple. As $f$ is not hyperbolic, the unique (quadratic) renormalization polynomial in $\mathfrak{R}(f)$ is  Misiurewicz. However, the entropy function $h:\PPP_2^{\rm ppf}\to\R$ is not continuous at \Mi parameters, because any \Mi quadratic polynomial has  positive core entropy (see \cite{Do}), but can be estimated by hyperbolic polynomials with Cantor Julia set (having $0$ core entropy).
Following Theorem \ref{thm:main2}, the function $h$ is not continuous at $f$.
\end{proof}

\section{Continuity of entropy fucntion: partial \pf parameters}


The objective here is to prove Theorem \ref{thm:main1}. Its proof is completely the same as that of Theorem \ref{thm:main2}: combining Proposition \ref{pro:key} and Straightening Theorem.
In fact, every Proposition or Lemma in Sections \ref{sec:newton-graph} and \ref{sec:continuity2} for Newton maps has a parallel version in partial \pf case with a similar argument. So we just state these parallel results without giving proofs, from which we deduce Theorem \ref{thm:main1}.

\subsection{The computation of the core entropy}
Let $P$ be a partial \pf polynomials with the Hubbard forest $\HHH_P$. Then $\HHH_P=\emptyset$ if and only if $P$ is hyperbolic with Cantor Julia set. In this case the core entropy $h(P)=0$.

In the case of $\HHH_P\not=\emptyset$, a collection $\OOO=\{H_i,0\leq i\leq p-1\}$ of components of $\HHH_P$ is called a \emph{cycle in ${\rm Comp}(\HHH_P)$ of period $p$} if $p$ is the minimal number such that $P(H_i)\subseteq H_{i+1}$ for each $0\leq i\leq p-1$ and $H_p:=H_0$.
Let $\HHH_\OOO$ denote the union of $H_0,\ldots,H_{p-1}$. Then, by Lemma \ref{Do2}, the core entropy $h(P)$ equals to the maximum of $h(P|_{\HHH_\OOO})$ with $\OOO$ going through all cycles in ${\rm Comp}(\HHH_P)$. Furthermore, we have the following results (remember $h_{top}(P|_\emptyset):=0$).

\begin{lemma}[analogy to Lemma \ref{lem:orbit-entropy}]\label{lem:equal}
Let $\OOO$ be a cycle in ${\rm Comp}(\HHH_P)$ of period $p$. Then $h_{top}(P|_{\HHH_\OOO})=h_{top}(P^p|_H)/p$ for any element $H$ of $\OOO$. As a consequence, we have entropy foumulas ($p_H$ denotes the period of $H$, and $H_\OOO$ is an arbitrary element of $\OOO$): 
\begin{eqnarray}
h(P)&=& \max\limits_{\mbox{
\tiny
$\begin{array}{c}
H\in{\rm Comp}(\HHH_P)\\
{\rm non-trivial}\\
{\rm periodic}\end{array}$
}
}\frac{1}{p_H}h_{top}(P^{p_H}|_H)\label{eq:entropy2'}\\[5pt]
&=& \max\limits_{\mbox{
\tiny
$\begin{array}{c}
\OOO\subseteq {\rm Comp}(\HHH_P)\\
{\rm non-trivial\ cycle}
\end{array}$
}}\frac{1}{p_{H_\OOO}}h_{top}(P^{p_{H_\OOO}}|_{H_\OOO}).\label{eq:entropy3'}
\end{eqnarray}
\end{lemma}

\subsection{Branner-Hubbard puzzles and renormalization}
According to Branner and Hubbard \cite{BH}, one can construct \emph{BH-puzzle pieces} for any polynomial with disconnected Julia set.

Recall that $g_P$ is the Green function of $P$. We denote by $a_P>0$ the maximum of $g_P(c)$ with $c\in{\rm Crit}_P$. Choose $r_0$ such that $a_P<r_0<da_P$, and $r_0\not=d^kg_P(c)$ for any $k\geq0$ and $c\in{\rm Crit}_P$. We define $V_0(P):=\{z\in\C:g_P(z)<r_0\}$, and $\VVV_n(P)$ the collection of bounded components of $P^{-n}(V_0(P))$ for each $n\geq0$. We call $\VVV(P):=\cup_{n\geq0}\VVV_n(P)$ the \emph{BH-puzzle} of $P$, and any element of $\VVV_n(P)$ a \emph{BH-puzzle piece of level $n$}.
The BH-puzzle pieces satisfy the following properties.
\begin{enumerate}
\item Each puzzle piece is a Jordan disk with its boundary contained in $\Omega(P)$.
\item Each puzzle piece of level $n\geq1$ is compactly contained in a puzzle piece of level $n-1$.
\item The map $P$ sends any puzzle piece of level $n\geq1$ to a puzzle piece of level $n-1$ as a branched covering.
\item The intersection of a sequence of nested puzzle pieces is a component of $\KKK_P$.
\end{enumerate}

Let $K$ be a non-trivial periodic component of $\KKK_P$ with period $m$. Since $P$ is uniform expanding near $\JJJ_P$, we have ${\rm deg}(P^m|_K)\geq2$. By properties above, there exists a renormalization triple $\rho=(P^m,U,V)$ such that $U,V$ are BH-puzzle pieces of $P$ and $K_\rho=K$. By Straightening Theorem, the map $P^m:U\to V$
is hybrid equivalent to a unique (up to a conjugation by ${\rm deg}(P^m|_K)-1$-th roots of unit) monic, centered \pf polynomial, denoted as $Q_K$, called the \emph{renormalization polynomial associated to $K$}. The polynomial $P$ is called \emph{renormalizable} if $\KKK_P$ has non-trivial periodic components.


For each $n\geq0$, we denote by $\G_n(P)$ the union of the boundaries of all BH-puzzle pieces of level $0\leq i\leq n$. Then $\G_n(P)$ is a disjoint union of circles in $\Omega(P)$. Given any integer $N>0$, the number $r_0$ is not included in $\{g_Q(z):z\in\cup_{c\in{\rm Crit}_Q}\cup_{i=0}^NQ^i(c)\}$ for all $Q$ closed enough to $P$. Hence the set $\G_N(Q)$ is also the disjoint union of circles for all $Q$ close enough to $P$.

\begin{lemma}[anology to Lemma \ref{lem:perturbation-newton-graph}]\label{lem:BH-puzzle}
The graph $\G_N(Q)$ is homeomorphic to $\G_N(P)$ for any $Q$ close enough to $P$ and $\G_N(Q)\to \G_N(P)$ as $Q\to P$.
\end{lemma}
\begin{proof}
The proof goes by induction. Note that the definition domains of $\phi_P,\phi_Q$ contains $\G_0(P),\G_0(Q)$ respectively for $Q$ close to $P$, where $\phi_P,\phi_Q$ denotes the B\"{o}ttcher coordinates of $P,Q$. Then $\G_0(Q)\to \G_0(P)$ as $Q\to P$ since $\G_0(P)=\phi^{-1}_P\{z:|z|=r_0\}$, $\G_0(Q)=\phi^{-1}_Q\{z:|z|=r_0\}$ and $\phi_Q^{-1}$ uniformly converge to $\phi_P^{-1}$ in a compact set containing $\{z:|z|=r_0\}$.

Assume that $\G_k(Q)$ is homeomorphic to $\G_k(P)$  and $\G_k(Q)\to \G_k(P)$ for some $0\leq k<N$. Let $\g_P$ be any component of $\G_{k+1}(P)$. Then its image $P(\g_P)=:\g_P'$ is a component of $\G_k(P)$. By the assumption of induction, we can find a unique component $\g_Q'$ of $\G_k(Q)$ such that $\g_Q'$ is in a sufficiently small neighborhood of $\g_P'$ for $Q$ close enough to $P$. Hence, by Rouch\'{e} Theorem, there is a unique preiamge $\g_Q$ by $Q$ of $\g_Q'$ which lies in arbitrarily small neighborhood of $\g_P$ provided $Q$ is close enough to $P$. Since $\g_P,\g_Q$ are circles, we then have $\g_Q\to\g_P$.
\end{proof}

As a consequence, we see that for any BH-puzzle piece $V_P$ of $P$ with level $k\geq0$, there exists a unique BH-puzzle piece $V_Q$ of $Q$ with level $k$ such that $\ov{V_Q}\to\ov{V_P}$ as $Q\to P$. Such $V_Q$ is called the \emph{deformation} of $V_P$ at $Q$.

\subsection{Continuity of the entropy function in partial \pf polynomial family}
Throughout this subsection, we always assume that $P$ is a partial \pf polynomial with disconnected Julia set, and $\{P_n,n\geq1\}$ a sequence of partial \pf polynomials converging to $P$.

\begin{proposition}[analogy to Proposition \ref{pro:non-renormalizable} ]\label{pro:non-renormalizable'}
Let $P$ be a non-renormalizable partial \pf polynomial. Then the entropy function $h:\PPP_d^{\rm ppf}\to\R$ is continuous at $P$.
\end{proposition}

In the following, we only deal with the renormalizable case.
One can choose a BH-puzzle piece $V_H$ of $P$
for each non-trivial periodic component $H$ of $\HHH_P$ such that $H\subseteq V_H$ and $V_H\cap V_{H'}=\emptyset$ for any two different components $H,H'$ of $\HHH_P$.
For each non-trivial periodic component $H$ of $\HHH_f$ and each large $n$, we denote by $V_{H,n}$ the BH-puzzle piece of $P_n$ deformed from $V_H$. These puzzle pieces are pairwise disjoint and ${\rm deg}(P|_{V_H})={\rm deg}(P_n|_{V_{H,n}})$.

As in the case of \pf Newton maps, the Hubbard forest $\HHH_{P_n}$ can be divided into two parts: let $\HHH^n_{\rm bd}$ denote the union of the components of $\HHH_{P_n}$ which stay in $\cup_{H} V_{H,n}$ under the iteration of $P_n$, and  $\HHH_{\rm esc}^n$ the union of periodic components of $\HHH_{P_n}$ not contained in $\HHH^n_{\rm bd}$, where the subscripts ``bd'' and ``esc'' mean ``bounded'' and ``escape'' respectively. Then both $\HHH^n_{\rm bd}$ and $\HHH^n_{\rm esc}$ are $P_n$-invariant, and by Propositions \ref{Do2}, \ref{Do3}, we have
\begin{equation}\label{eq:33'}
h(P_n)=\max\{\ h_{top}(P_n|_{\HHH_{\rm bd}^n}),h_{top}(P_n|_{\HHH_{\rm esc}^n})\ \}.
\end{equation}


\begin{lemma}[analogy to Lemma \ref{lem:infinite-period}]\label{lem:infinite-period'}
The topological entropy  $h_{top}(P_n|_{\HHH_{\rm esc}^n})$ tends to $0$ as $n\to\infty$.
\end{lemma}


For any non-trivial periodic component $H$ of $\HHH_P$, we define a forest
\begin{equation}\label{eq:55}
H_n:=\HHH_{\rm bd}^n\cap V_{H,n}.
\end{equation}
It is easy to see that $f_n(H_n)\subseteq H_n'$ if and only if $f(H)\subseteq H'$, and we have the entropy formula (apology to formula \eqref{eq:44}):
\begin{equation}\label{eq:44'}
h_{top}(P_n|_{\HHH_{\rm bd}^n})=    \max\limits_{\mbox{
\tiny
$\begin{array}{c}
H\in{\rm Comp}(\HHH_P)\\
{\rm non-trivial}\\
{\rm period}\ p_H\end{array}$
}
}
\frac{1}{p_H}h_{top}(P_n^{p_H}|_{H_n}).
\end{equation}

Combining formulas \eqref{eq:33'},\eqref{eq:44'} and Lemma \ref{lem:infinite-period'}, we get a result analogous to Lemma \ref{lem:entropy-f_n}

\begin{lemma}[analogy to Lemma \ref{lem:entropy-f_n}]\label{lem:entropy-f_n'}
The limit/limit superior/limit inferior of $h(P_n)$ is equal to that of $\max_H \frac{1}{p_H}h_{top}(P_n^{p_H}|_{H_n})$ with $H$ going through all non-trivial periodic components of $\HHH_P$, as $n\to\infty$.
\end{lemma}

Let $H$ be a non-trivial periodic component of $\HHH_f$ with period $p$. Then $\rho_H:=(P^p,V_H,P^p(V_H))$ is a renormalization triple. By Straightening Theorem, there exist a monic, centered polynomial $P_H$ of degree $\delta={\rm deg}(P|_{V_H})$ hybrid equivalent to $P^p:V_H\to P^p(V_H)$, called a \emph{renormalization polynomial of $P$ associated to $H$}. It is clear that $h_{top}(P^p|_H)=h(P_H)$.

Similarly, for each large $n$,
the map $P^p_n:V_{H,n}\to f(V_{H,n})$ is a polynomial-like map of degree $\de$;
By Straightening Theorem again, we get for each large $n$ a monic,centered polynomial $P_{H,n}$ of degree $\delta$ hybrid equivalent to $P^p_n:V_{H,n}\to f(V_{H,n})$, which we call a \emph{straightening-perturbation} of $P_H$ at $P_n$.

\begin{lemma}[analogy to Lemma \ref{lem:straighten}]\label{lem:straighten'}
Every polynomial $P_{H,n}$ is partial \pf and  $h_{top}(P^p_n|_{H_n})=h(P_{H,n})$ ($H_n$ is defined in \eqref{eq:55}). Furthermore, by suitably choice of $P_{H,n}$, we have $P_{H,n}\to P_H$ as $n\to\infty$.
\end{lemma}

\begin{proposition}[analogy to Proposition \ref{pro:liminf-newton}]\label{pro:liminf-newton'}
Let $P$ be a partial \pf polynomial, and $H_1,\ldots,H_k$ a collection of non-trivial periodic components of $\HHH_P$ in pairwise distinct cycles. Let $P_{H_i},i=1,\ldots,k$, be a renormalization polynomial of $P$ associated to $H_i$. Then there exists a sequence of  polynomials $\{P_n,n\geq1\}\subseteq \PPP_d^{\rm ppf}$ such that $P_n\to P$ and the polynomials $\{P_{H_i,n},n\geq1\}$  is a entropy-minimal sequence converging to $P_{H_i}$ for each $1\leq i\leq k$.
\end{proposition}

\begin{proof}[Proof of Theorem \ref{thm:main1}]
This proof is completely the same as that of Theorem \ref{thm:main2}. One just need to replace the formula \eqref{eq:entropy2}; Lemmas \ref{lem:straighten}, \ref{lem:entropy-f_n} and Propositions \ref{pro:non-renormalizable},
\ref{pro:liminf-newton} used in the proof of Theorem \ref{thm:main2}, by the formula \eqref{eq:entropy2'}; Lemmas \ref{lem:straighten'}, \ref{lem:entropy-f_n'} and Propositions \ref{pro:non-renormalizable'},
\ref{pro:liminf-newton'} respectively. The detail is omitted.
\end{proof}

\vspace{1cm}

\noindent Yan Gao, \\
Mathemaitcal School  of Sichuan University, Chengdu 610064,
P. R. China. \\
Email: gyan@scu.edu.cn

\end{document}